\documentclass[12pt, twoside]{article}
\usepackage{amsmath,amsthm,amssymb}
\usepackage{times}
\usepackage{enumerate}
\usepackage{epsfig,latexsym,amsfonts,amscd,
graphics,epic}

\def\titlerunning#1{\gdef\titrun{#1}}
\makeatletter
\def\author#1{\gdef\autrun{\def\and{\unskip, }#1}\gdef\@author{#1}}
\def\address#1{{\def\and{\\\hspace*{18pt}}\renewcommand{\thefootnote}{}%
\footnote {#1}}%
\markboth{\autrun}{\titrun}}
\makeatother
\def\email#1{e-mail: #1}
\def\subjclass#1{{\renewcommand{\thefootnote}{}%
\footnote{\emph{Mathematics Subject Classification (2010):} #1}}}
\def\keywords#1{\par\medskip
\noindent\textbf{Keywords.} #1}

\newtheorem{thm}{Theorem}[section]

\newtheorem{lem}[thm]{Lemma}

\theoremstyle{definition}

\newtheorem{rem}[thm]{Remark}

\numberwithin{equation}{section}

\newcommand{\Sig}{{\Sigma}}
\newcommand{\PP}{{\mathbb P}}
\newcommand{\Gam}{{\Gamma}}
\newcommand{\R}{{\mathbb R}}
\newcommand{\Del}{{\Delta}}
\newcommand{\CL}{{\mathcal L}}
\newcommand{\bpp}{{\boldsymbol P}}
\newcommand{\alp}{{\alpha}}
\newcommand{\CV}{{\mathcal V}}

\newcommand{\proofend}{\hfill$\Box$\bigskip}
\newcommand{\K}{{\mathbb K}}
\newcommand{\CH}{{\mathcal CH}}
\newcommand{\sig}{{\sigma}}
\newcommand{\eps}{{\varepsilon}}
\newcommand{\CX}{{\mathcal X}}
\newcommand{\LIC}{\text{\rm left}}
\newcommand{\val}{{\mathop{val}}}
\newcommand{\Tor}{{\mathop{Tor}}}
\newcommand{\Val}{{\mathop{Val}}}
\newcommand{\SSS}{{\mathbb S}}
\newcommand{\del}{{\delta}}
\newcommand{\Id}{{\mathop{Id}}}
\newcommand{\Z}{{\mathbb Z}}
\newtheorem{proposition}[thm]{Proposition}
\newcommand{\bx}{{\boldsymbol p}}
\newcommand{\Pic}{{\mathop{Pic}}}
\newcommand{\bp}{{\boldsymbol z}}
\newcommand{\N}{{\mathbb N}}
\newcommand{\mt}{{\mathop{mt}}}
\newcommand{\bet}{{\beta}}
\newcommand{\bd}{{\boldsymbol d}}
\newcommand{\Idim}{{\mathop{idim}}}
\newcommand{\C}{{\mathbb C}}

\begin{document}


\baselineskip=17pt


\titlerunning{Welschinger invariants
of Del Pezzo surfaces}

\title{Welschinger invariants \\
of small non-toric Del Pezzo surfaces}

\author{Ilia Itenberg
\and Viatcheslav Kharlamov \and Eugenii Shustin}

\date{}

\maketitle

\address{I. Itenberg: Universit\'{e} de Strasbourg, IRMA
and Institut Universitaire de France, 7, rue Ren\'{e} Descartes,
67084 Strasbourg Cedex, France;
\email{ilia.itenberg@math.unistra.fr} \and V. Kharlamov:
Universit\'{e} de Strasbourg and IRMA, 7, rue Ren\'{e} Descartes,
67084 Strasbourg Cedex, France;
\email{viatcheslav.kharlamov@math.unistra.fr} \and E. Shustin:
School of Mathematical Sciences, Raymond and Beverly Sackler Faculty
of Exact Sciences, Tel Aviv University, Ramat Aviv, 69978 Tel Aviv,
Israel; \email{shustin@post.tau.ac.il}}



\begin{abstract}
We give a recursive formula for purely real Welschinger invariants
of the following real Del Pezzo surfaces: the projective plane blown
up at~$q$ real and $s \leq 1$ pairs of conjugate imaginary points,
where $q+2s\le 5$, and the real quadric blown up at $s \leq 1$ pairs
of conjugate imaginary points and having non-empty real part. The
formula is similar to Vakil's recursive formula \cite{Va} for
Gromov-Witten invariants of these surfaces and generalizes our
recursive formula~\cite{IKS3} for purely real Welschinger invariants
of real toric Del Pezzo surfaces. As a consequence,  we prove the
positivity of the Welschinger invariants under consideration and
their logarithmic asymptotic equivalence to genus zero Gromov-Witten
invariants.

\end{abstract}

\tableofcontents

\begin{figure}[htb]
\hskip3.25in\includegraphics[height=6.5cm,width=5cm,angle=0,draft=false]{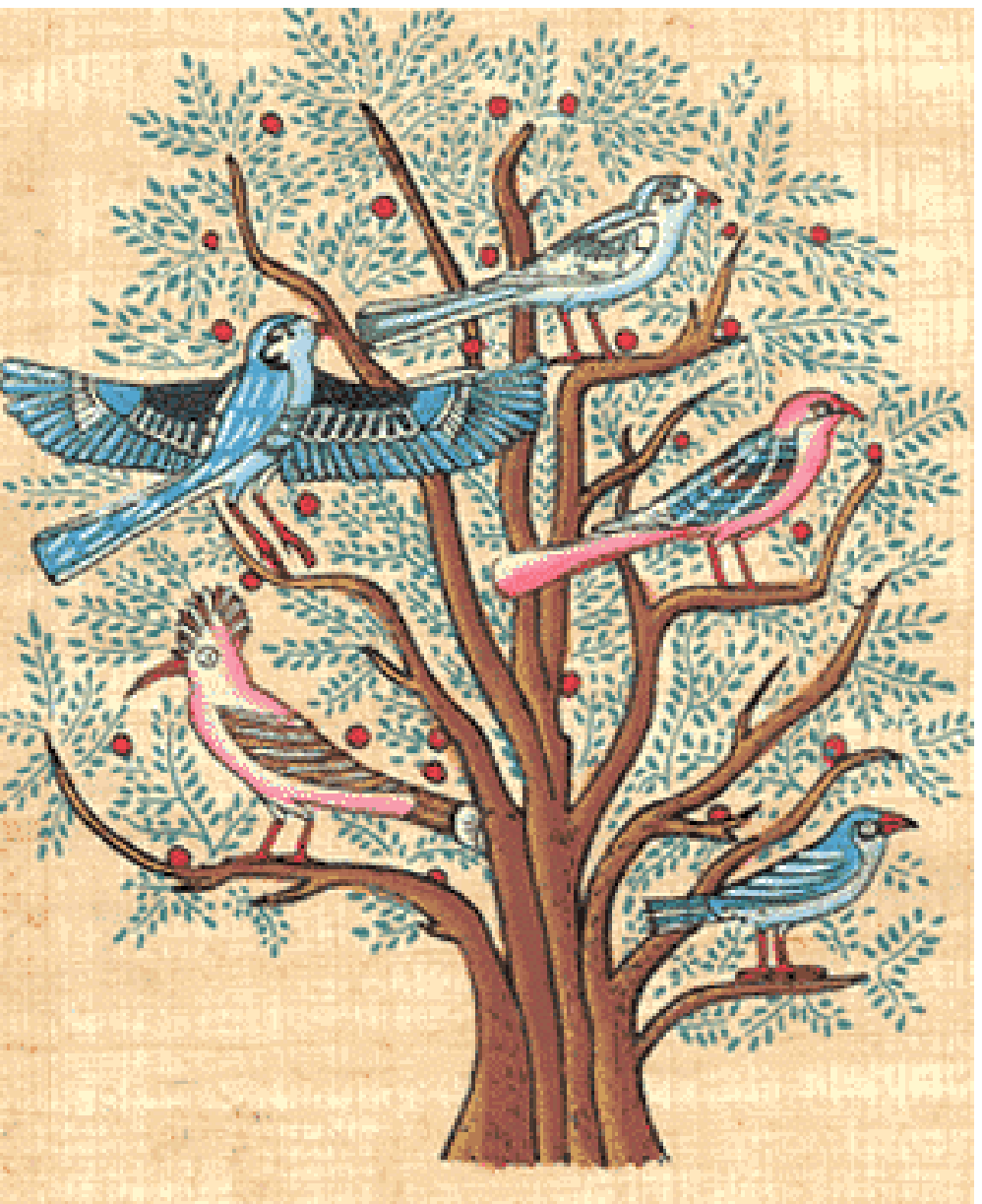}
\end{figure}
{\hskip3in \it Wall painting}

\vskip-5pt

{\hskip3in \it Tomb of Khnumhotep III}

\vskip-5pt

{\hskip3in \it Beni Hasan, Egypt, Middle Kingdom}

\section{Introduction}\label{intro}

Welschinger invariants play a role of real analogue of genus zero
Gromov-Witten invariants. As introduced in \cite{W1}, they count,
with certain signs, the real rational pseudo-holomorphic curves
which pass through given real configurations of points on a given
real rational symplectic four-fold (and on certain six-folds).
Tropical geometry together with the open Gromov-Witten theory
\cite{Sol1,Sol2} and symplectic field theory \cite{W2} provides
powerful tools for the study of Welschinger invariants.

In the case of Del Pezzo surfaces, which is treated in the present
article, the Welschinger count is equivalent to enumeration of real
rational algebraic curves. Here we restrict ourselves to
configurations of real points, and thus speak of {\it purely real}
Welschinger invariants. We use the tropical geometry techniques and
produce a recursive formula for purely real Welschinger invariants
$W(\Sig,D)$ of the following real Del Pezzo surfaces $\Sig$ (see
Theorem \ref{tn12} in section \ref{sec5}):
\begin{itemize}
\item $\PP^2_{q,s}$, the real plane blown up at a generic collection
of $q$ real points and $s$ pairs of conjugate imaginary points,
where $0\le q\le 5$, $0\le s\le 1$, $q+2s\le 5$,
\item the real quadric surface with a nonempty real part blown up at
$s\le 1$ pairs of conjugate imaginary points.
\end{itemize}
We show that together with certain explicit initial data, this
formula recursively determines all purely real Welschinger
invariants of the above surfaces (see Theorem \ref{tn12} in section
\ref{sec5}).

The formula we obtain here can be seen as a real version of Vakil's
recursive formula for Gromov-Witten invariants of these surfaces
\cite{Va}, and generalizes our earlier recursive formula for the
real toric Del Pezzo surfaces \cite{IKS3} in a similar manner as
Vakil's formula generalizes the Caporaso-Harris formula \cite{CH}.

As application, we derive a number of properties of the invariants
under consideration. In particular, we prove that for each surface
$\Sig$ as above and for any real nef and big divisor $D$ on $\Sig$,
the invariant $W(\Sig,D)$ is positive, and
$$\lim_{n\to\infty}\frac{\log W(\Sig,nD)}{n\log n}=\lim_{n\to\infty}
\frac{\log GW_0(\Sig,nD)}{n\log n}=-DK_\Sig\ ,$$ where
$GW_0(\Sig,nD)$ is the genus zero Gromov-Witten invariant (see
Theorem \ref{tn5} in section \ref{sec-pa}). The geometric meaning of
this result is that through any generic configuration of
$-DK_\Sig-1$ distinct real points of $\Sig$, one can trace a real
rational curve $C\in|D|$ and, furthermore, in the logarithmic scale
the number of such real rational curves is close to the number of
all complex rational curves $C\in|D|$ through the given
configuration.

In addition, we observe a congruence between purely real Welschinger
and genus zero Gromov-Witten invariants and show the monotone
behavior of $W(\Sig,D)$ with respect to the variable $D$ (see
Theorems \ref{tn6} and \ref{tn7} in sections \ref{sec-con} and
\ref{sec-mon}, respectively). The aforementioned positivity,
asymptotic and monotonicity properties extend our results for real
toric surfaces with standard and nonstandard real structures
obtained in \cite{IKS2,IKS4}.

The novelty of the present work consists first of all in application
of the tropical tools to {\it non-toric} Del Pezzo surfaces, namely,
to $\PP^2_{4,0}$, $\PP^2_{2,1}$, $\PP^2_{5,0}$, $\PP^2_{3,1}$ (we
call them {\it small non-toric}, since they are closer to toric
surfaces than other Del Pezzo surfaces $\PP^2_{q,s}$, $q+2s=6,7,8$).
The tools of tropical geometry, as they are developed in
\cite{Mi,Mi-icm} and \cite{Sh0,Sh1,Sh2}, and explored in
\cite{IKS,IKS2,IKS4,IKS3}, are essentially restricted to the toric
case. So, having a non-toric surface, we blow down some exceptional
divisors. Thus, we come to a toric surface, but as a price to pay,
the curves we are interested in unavoidably acquire some fixed
multiple points with prescribed multiplicities. It is a serious
problem, since for curves with fixed multiple points, correspondence
theorems similar to \cite{Mi, Sh0} are not known in general. One of
the obstacles is the fact that in this case a direct tropical
approach leads to tropical moduli spaces of wrong dimension. We
overcome this difficulty restricting our attention to very specific
configurations of points serving as constraints in our enumerative
problem (${\cal CH}$-configurations defined in section
\ref{CH-config}). We introduce a special class of tropical curves
matching ${\cal CH}$-configurations (see section \ref{sec4}) which,
on one hand, suit well for the patchworking of algebraic curves with
fixed multiple points (cf. \cite{Sh09}) and, on the other hand, are
adjusted to the {\it cutting procedure} and the proof of the {\it
tropical recursive formula} (see sections \ref{cutting} and
\ref{sec42}, \ref{sec43}). Then, we observe that the complex
tropical recursive formula, involving the numbers of complex curves
obtained by patchworking quantization of these specific tropical
curves, and Vakil's recursive formula, involving {\it all} complex
algebraic curves in count, coincide (see section \ref{adaptation}).
This allows us to get the key ingredient, a complex correspondence
theorem (section \ref{corresp-thm}). After that, we derive a {\it
real tropical recursive formula} (section \ref{sec43}) which
involves suitable tropical Welschinger multiplicities (cf.
\cite{Sh09}), and, using the complex correspondence theorem, convert
the tropical formula into a recursive formula for Welschinger
invariants.

It is natural to compare our formulas with J.~Solomon's recursive
formulas for Welschin- ger invariants. The latter formulas are
encoded in a real version of WDVV equations which was proposed by
J.~Solomon \cite{Sol2}. One of the differences is that Solomon's
formulas involve not only purely real Welschinger invariants but
also invariants associated with collections of real and conjugated
imaginary points, whereas our formulas contain only purely real
Welschinger invariants mixed with certain auxilliary tropical
numbers. Another feature is that the coefficients in Solomon's
formulas have alternating signs, whereas in our formulas all the
coefficients are {\it positive}. The latter circumstance appears to
be crucial in the proofs of the positivity, asymptotic grows
formula, and monotonicity of purely real Welschinger invariants.

The text is organized as follows. Section~\ref{tropical_curves}
contains a description of the class of tropical curves used in the
paper. In Section~\ref{algebraic-formulas} we present an adapted
version of Vakil's recursive formula. In
Section~\ref{tropical-formulas} we prove tropical recursive
formulas, and in Section~\ref{corresp-thm} we establish a
correspondence between the tropical and algebraic curves in count.
In Section \ref{Wel} we derive the recursive formula for the purely
real Welschinger invariants of the surfaces under consideration, and
in Section \ref{properties} we use the formula to prove the
aforementioned properties of Welschinger invariants.

\section{Tropical curves}\label{tropical_curves}

\subsection{Parameterized plane tropical
curves}\label{parameterization} For the reader convenience, we
recall here basic definitions and facts concerning tropical curves.
The details can be found in \cite{Mi}, \cite{Mi-icm}.

Let $\overline\Gam$ be a finite graph which has neither bivalent,
nor isolated vertices. Denote by $\Gam^0_\infty$ the set of
univalent vertices of $\overline\Gam$, and put $\Gam = \overline\Gam
\setminus \Gam^0_\infty$. Denote by $\Gam^1$ the set of edges of
$\overline\Gam$. An edge~$E$ of~$\overline\Gam$ is called an {\it
end} if~$E$ is incident to a univalent vertex, and is called a {\it
bounded edge} otherwise. We say that $\overline\Gam$ is an {\it
abstract tropical curve} if $\Gam$ is equipped with a metric such
that each bounded edge of $\overline\Gam$ is isometric to an open
bounded interval in $\R$, each end of $\overline\Gam$ incident to
exactly one univalent vertex is isometric to an open ray in $\R$,
and each end of $\overline\Gam$ incident to two univalent vertices
is isometric to $\R$.

A {\it plane parameterized tropical curve} is a pair
$(\overline\Gam, h)$, where $\overline\Gam$ is an abstract tropical
curve and $h: \Gam \to \R^2$ is a continuous map, such that
\begin{itemize}
\item for any edge $E \in \Gam^1$
the restriction of~$h$ to $E$ is a non-zero affine map, and $h(E)$
is contained in a line with rational slope,
\item for any edge $E \in \Gam^1$,
any vertex~$V$ incident to~$E$, and any point $P \in E$, one has
$dh_P(U_{V, P}(E)) = w(E)u_V(E)$, where $w(E)$ is a positive integer
number, $U_{V, P}(E)$ is the unit tangent vector to $E$ at $P$ such
that $U_{V, P}(E)$ points away from~$V$, and $u_V(E)$ is a {\it
primitive integer vector} ({\it i.e.}, a vector whose coordinates
are integer and mutually prime),
\item for each vertex~$V$ of~$\Gam$,
the following {\it balancing condition} is satisfied:
$$
\sum_{E\in\Gam^1,\ V\in\partial E} w(E) u_V(E)=0.
$$
\end{itemize}

For any edge~$E$ of a parameterized plane tropical curve
$(\overline\Gam, h)$, the number $w(E)$ is called the {\it weight}
of~$E$. The multi-set of vectors $$\{-w(E)u_V(E)\ |\
V\in\Gam^0_\infty, E \in \Gam^1, V \in \partial E\}$$ is called the
{\it degree} of a parameterized plane tropical curve
$(\overline\Gam, h)$ and is denoted by $\Del(\overline\Gam, h)$. For
any parameterized plane tropical curve $(\overline\Gam, h)$, the sum
of the vectors in $\Del(\overline\Gam, h)$ is equal to~$0$.

\subsection{Cutting procedure}\label{cutting}
Let $(\overline\Gam, h)$ be a parameterized plane tropical curve,
and~$x$ a point of~$\Gam$. We will construct a new parameterized
plane tropical curve $(\overline\Gam_x, h_x)$ out of
$(\overline\Gam, h)$ and~$x$.

If~$x$ belongs to an edge~$E$ of $(\overline\Gam, h)$, denote by
$V_1$ and $V_2$ the vertices incident to~$E$. Consider a graph
$\overline\Gam_x$ obtained from $\overline\Gam$ by removing~$E$,
introducing two new vertices $V^a_1$ and $V^a_2$ (called {\it added
vertices}), and introducing two new edges $E^a_1$ and $E^a_2$
(called {\it added edges}) such that $E^a_i$ is incident to $V_i$
and $V^a_i$, $i = 1, 2$. We say that the edges~$E^a_1$ and~$E^a_2$
{\it match} each other. The edge~$E$ is called the {\it predecessor}
of $V^a_1$, $V^a_2$, $E^a_1$, and $E^a_2$. Extend the metric on
$\Gam \setminus E$ to a metric on $(\Gam \setminus E) \cup (E^a_1
\cup E^a_2)$ in such a way that $E^a_i$, $i = 1, 2$, is isometric to
a ray of $\R$ if $V_i$ is not univalent, and $E^a_i$ is isometric to
$\R$ otherwise. This turns $\overline\Gam_x$ into an abstract
tropical curve (see Figure~\ref{new-picture1}). Denote by $\Gam_x$
the complement in $\overline\Gam_x$ of the univalent vertices, and
consider a map $h_x: \Gam_x \to \R^2$ such that
\begin{itemize}
\item $h_x|_{\Gam_x \setminus (E^a_1 \cup E^a_2)} = h|_{\Gam \setminus E}$,
\item $(\overline\Gam_x, h_x)$ is a parameterized plane tropical curve,
\item the vector $u_{V_i}(E^a_i)$, $i = 1, 2$,
of the curve $(\overline\Gam_x, h_x)$ coincides with the vector
$u_{V_i}(E)$ of the curve $(\overline\Gam, h)$,
\item the weight $w(E^a_i)$, $i = 1, 2$, of the edge $E^a_i$
of the curve $(\overline\Gam_x, h_x)$ coincides with the weight
$w(E)$ of the edge~$E$ of the curve $(\overline\Gam, h)$.
\end{itemize}
(Note that if the valencies of~$V_1$ and~$V_2$ are both greater
than~$1$, then the last two conditions in the definition of~$h_x$
follow from the first two conditions.)

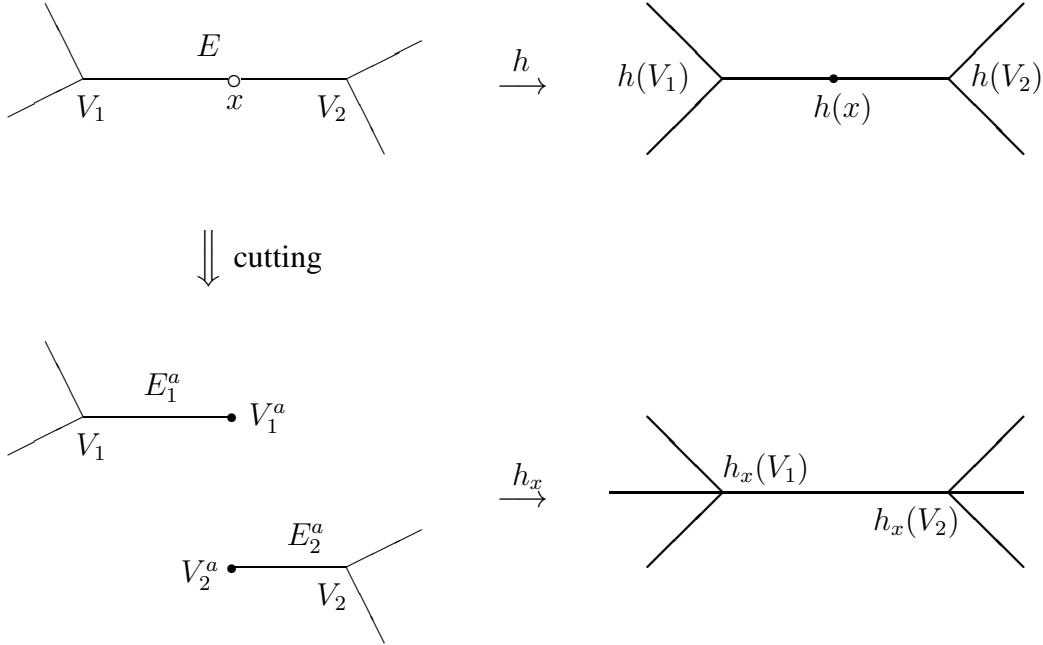
\begin{figure}
\setlength{\unitlength}{1cm}
\begin{picture}(14,9)(0,0)
\thinlines \put(0.5,7){\line(2,1){1}}\put(1.5,7.5){\line(-1,2){0.5}}
\put(1.5,7.5){\line(1,0){1.91}}\put(3.6,7.5){\line(1,0){1.4}}
\put(5,7.5){\line(2,1){1}}\put(5,7.5){\line(1,-2){ 0.5}}
\put(0.5,2.5){\line(2,1){1}}\put(1.5,3){\line(-1,2){0.5}}
\put(1.5,3){\line(1,0){2}}\put(3.5,1){\line(1,0){1.5}}\put(5,1){\line(1,-2){0.5}
} \put(5,1){\line(2,1){1}} \thicklines
\put(9,6.5){\line(1,1){1}}\put(10,7.5){\line(-1,1){1}}\put(10,7.5){\line(1,0){3}
}\put(13,7.5){\line(1,1){1}}
\put(13,7.5){\line(1,-1){1}}\put(8.5,2){\line(1,0){5.5}}\put(10,2){\line(-1,1){1
}}
\put(9,1){\line(1,1){1}}\put(13,2){\line(1,-1){1}}\put(13,2){\line(1,1){1}}
\put(7,7.3){$\longrightarrow$}\put(7.2,7.6){$h$}\put(11.4,7.35){$\overset{\bullet}{}$}
\put(3.4,2.84){$\overset{ \bullet}{}$}\put(3.4,0.83){$\overset{
\bullet}{}$}\put(1.4,7){$V_1$}
\put(1.4,2.5){$V_1$}\put(3.4,7.1){$x$}
\put(11.2,7){$h(x)$}\put(4.6,7){$V_2$}\put(4.6,0.5){$V_2$}
\put(3.7,2.9){$V_1^a$}\put(2.8,0.8){$V_2^a$}\put(8.6,7.4){$h(V_1)$}\put(13.3,
7.4){$h(V_2)$}\put(3.4,7.36){$\circ$}
\put(7,1.8){$\longrightarrow$}\put(7.2,2.1){$h_x$}\put(10,2.2){$h_x(V_1)$}
\put(12,1.5){$h_x(V_2)$}\put(3,7.8){$E$}\put(2.3,3.3){$E_1^a$}\put(4.1,1.3){
$E_2^a$} \put(3,5){$\Big\Downarrow$}\put(3.5,5){\rm cutting}
\end{picture}
\caption{Cutting, I}\label{new-picture1}
\end{figure}

If the point~$x$ coincides with a vertex~$V$ of $\Gam$, denote by
$E_1$, $\ldots$, $E_k$ the edges incident to~$V$, and denote by
$V_i$, $i = 1$, $\ldots$, $k$, the vertex incident to~$E_i$ and
different from~$V$. In this case, consider a graph $\overline\Gam_x$
obtained from $\overline\Gam$ by removing~$V$, $E_1$, $\ldots$,
$E_k$, introducing new vertices $V^a_1$, $\ldots$, $V^a_k$ (called
{\it added vertices}), and introducing new edges $E^a_1$, $\ldots$,
$E^a_k$ (called {\it added edges}) such that $E^a_i$ is incident to
$V_i$ and $V^a_i$, $i = 1$, $\ldots$, $k$. For each $i = 1$,
$\ldots$, $k$, both elements $V$ and $E_i$ of the graph
$\overline\Gam$ are called {\it predecessors} of $V^a_i$ and
$E^a_i$. Extend the metric on $\Gam \setminus (V \cup E_1 \cup
\ldots \cup E_k)$ to a metric on $(\Gam \setminus (V \cup E_1 \cup
\ldots \cup E_k))
 \cup (E^a_1 \cup \ldots \cup E^a_k)$
in such a way that $E^a_i$, $i = 1$, $\ldots$, $k$, is isometric to
a ray of $\R$ if $V_i$ is not univalent, and $E^a_i$ is isometric to
$\R$ otherwise. This turns $\overline\Gam_x$ into an abstract
tropical curve (see Figure~\ref{new-picture2}). Denote by $\Gam_x$
the complement in $\overline\Gam_x$ of the univalent vertices, and
consider a map $h_x: \Gam_x \to \R^2$ such that
\begin{itemize}
\item $h_x|_{\Gam_x \setminus (E^a_1 \cup \ldots \cup E^a_k)}
= h|_{\Gam \setminus (V \cup E_1 \cup \ldots \cup E_k)}$,
\item $(\overline\Gam_x, h_x)$ is a parameterized plane tropical curve,
\item the vector $u_{V_i}(E^a_i)$, $i = 1$, $\ldots$, $k$,
of the curve $(\overline\Gam_x, h_x)$ coincides with the vector
$u_{V_i}(E_i)$ of the curve $(\overline\Gam, h)$,
\item the weight $w(E^a_i)$, $i = 1$, $\ldots$, $k$, of the edge $E^a_i$
of the curve $(\overline\Gam_x, h_x)$ coincides with the weight
$w(E_i)$ of the edge~$E_i$ of the curve $(\overline\Gam, h)$.
\end{itemize}
(Again, if all vertices $V_1$, $\ldots$, $V_k$ have valencies
greater than~$1$, then the last two conditions in the definition
of~$h_x$ follow from the first two conditions.)

\begin{figure}
\setlength{\unitlength}{1cm}
\begin{picture}(11,10)(-1.5,0)
\thinlines
\put(2.45,8.55){\line(-1,1){0.95}}\put(2.55,8.52){\line(2,1){0.95}}
\put(3.5,9){\line(1,0){1}}\put(3.5,9){\line(0,1){0.5}}\put(2.5,6.5){\line(0,1){1.95
}}
\put(2.5,7){\line(-1,-1){0.5}}\put(2.5,7){\line(1,-1){0.5}}\put(2,3){\line(-1,1)
{1}}\put(3,3){\line(2,1){1}}
\put(4,3.5){\line(1,0){1}}\put(4,3.5){\line(0,1){0.5}}\put(2.5,0){\line(0,1){2}}
\put(2.5,0.5){\line(-1,-1){0.5}}\put(2.5,0.5){\line(1,-1){0.5}}\put(8.45,0){
\line(0,1){4}} \thicklines
\put(8.57,0.5){\line(-1,-1){0.5}}\put(8.57,0.5){\line(1,-1){0.5}}\put(7,3.5){
\line(1,-1){3}}
\put(7,1.25){\line(2,1){3}}\put(10,2.75){\line(1,0){1}}\put(10,2.75){\line(0,1){
0.5}} \put(8.5,8.5){\line(-1,1){1}}\put(8.5,8.5){\line(2,1){1}}
\put(9.5,9){\line(1,0){1}}\put(9.5,9){\line(0,1){0.5}}\put(8.5,6.5){\line(0,1){2
}} \put(8.5,7){\line(-1,-1){0.5}}\put(8.5,7){\line(1,-1){0.5}}

\put(1.4,9.4){$\overset{\bullet}{}$} \put(2.41,8.4){$\circ$}
\put(8.46,8.35){$\overset{\bullet}{}$}\put(0.9,3.87){$\overset{\bullet}{}$}
\put(1.92,2.88){$\overset{\bullet}{}$}\put(2.92,2.84){$\overset{\bullet}{}$}
\put(2.44,1.87){$\overset{\bullet}{}$}

\put(6,7.9){$\longrightarrow$}\put(6,2.4){$\longrightarrow$}\put(6.2,8.2){$h$}
\put(6.1,2.7){$h_x$}\put(0.9,9.4){$V_1$}\put(2,8.2){$V$}\put(3,9.1){$V_2$}
\put(2,7.05){$V_k$}\put(2.1,9){$E_1$}\put(2.9,8.3){$E_2$}\put(2.6,7.6){$E_k$}
\put(3.5,7.5){$\vdots$}\put(8.6,8.1){$h(V)$}\put(0.5,3.9){$V_1$}\put(1.5,2.6){
$V_1^a$}
\put(3.1,2.6){$V_2^a$}\put(4,3){$V_2$}\put(1.9,1.9){$V_k^a$}\put(1.9,0.5){$V_k$}
\put(1.5,3.7){$E_1^a$}\put(3.1,3.5){$E_2^a$}\put(2.6,1){$E_k^a$}
\put(2.3,5.1){$\Big\Downarrow$}\put(2.7,5.2){\rm cutting}
\put(3.5,1.5){$\vdots$}
\end{picture}
\caption{Cutting, II}\label{new-picture2}
\end{figure}
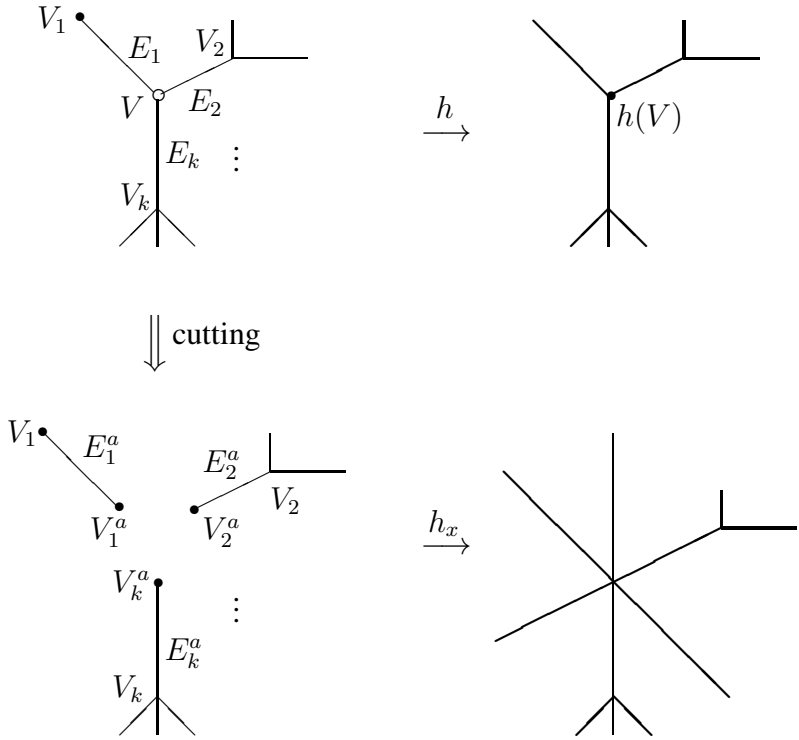

In the both cases considered above, the parameterized plane tropical
curve $(\overline\Gam_x, h_x)$ is called a {\it cut} of
$(\overline\Gam, h)$ at~$x$.

\subsection{$\CL$-curves}\label{sec4}
Put $\widehat\R^2 = (\{-\infty\} \cup \R) \times \R$ and
$L_{-\infty} = \{-\infty\} \times \R \subset \widehat\R^2$. Let
$(\overline\Gam, h)$ be a parameterized plane tropical curve, and
${\cal V}$ a subset of $\Gam^0_\infty$. Put $\widehat\Gam =
\overline\Gam \backslash {\cal V}$. We say that $(\overline\Gam,
{\cal V}, h)$ is an {\it $\CL$-curve} if
\begin{itemize}
\item for any univalent vertex~$V$ of $\widehat\Gam$
one has $u_V(E) = (1, 0)$, where $E$ is the end incident to~$V$,
\item for any univalent vertex~$V' \in {\cal V}$
one has $u_{V'}(E) \ne (1, 0)$, where~$E$ is the end incident
to~$V'$.
\end{itemize}

If $(\overline\Gam, {\cal V}, h)$ is an $\CL$-curve, then $h$
naturally extends to a map ${\widehat h}: \widehat\Gam \to
\widehat\R^2$, and the image under $\widehat h$ of any univalent
vertex of $\widehat\Gam$ belongs to $L_{-\infty}$. The ends
of~$\overline\Gam$ which are incident to univalent vertices
of~$\widehat\Gam$ are called {\it left}.

The {\it degree} $\Del(\overline\Gam, {\cal V}, h)$ of an
$\CL$-curve $(\overline\Gam, {\cal V}, h)$ is the degree
$\Del(\overline\Gam, h)$ of the parameterized plane tropical curve
$(\overline\Gam, h)$. An $\CL$-curve $(\overline\Gam, {\cal V}, h)$
is {\it irreducible}, if the graph $\overline\Gam$ is connected. The
{\it genus} of an irreducible $\CL$-curve $(\overline\Gam, {\cal V},
h)$ is the first Betti number $b_1(\overline\Gam)$
of~$\overline\Gam$. The irreducible $\CL$-curves of genus~$0$ are
called {\it rational}.

Let $(\overline\Gam, {\cal V}, h)$ be an $\CL$-curve, and
$\overline\Gam_1$, $\ldots$, $\overline\Gam_n$ the connected
components of~$\overline\Gam$. For any integer $j = 1$, $\ldots$,
$n$, put $\Gam_j = \overline\Gam_j \setminus \Gam^0_{\infty}$ and
denote by ${\cal V}_j$ the vertices belonging simultaneously to
$\cal V$ and $\overline\Gam_j$. The $\CL$-curves $(\overline\Gam_j,
{\cal V}_j, h|_{\Gam_j})$ are called {\it irreducible components} of
$(\overline\Gam, {\cal V}, h)$.

An edge of an $\CL$-curve $(\overline\Gam, {\cal V}, h)$ is said to
be {\it horizontal}, if the image of this edge under~$h$ is
contained in a horizontal line. An $\CL$-curve $(\overline\Gam,
{\cal V}, h)$ is {\it horizontal}, if $\overline\Gam$ is a segment
and $h(\Gam)$ is a horizontal line in $\R^2$. An irreducible
$\CL$-curve $(\overline\Gam, {\cal V}, h)$ is called {\it
one-sheeted} if among the vectors of its degree $\Del(\overline\Gam,
{\cal V}, h)$ there are exactly two vectors with non-zero second
coordinate, and each of these two vectors is of the form $(A, \pm
1)$, where~$A$ is an integer number.

An $\CL$-curve $(\overline\Gam, {\cal V}, h)$ is {\it
non-degenerate}, if for any non-univalent vertex~$V$
of~$\widehat\Gam$ the vectors $u_V(E_1)$, $\ldots$, $u_V(E_k)$
(where $E_1$, $\ldots$, $E_k$ are the edges incident to~$V$) span
$\R^2$. A non-degenerate $\CL$-curve $(\overline\Gam, {\cal V}, h)$
is called {\it simple} if any non-univalent vertex of~$\widehat\Gam$
has valency~$3$, and is called {\it pseudo-simple}, if for any
non-univalent vertex~$V$ of~$\widehat\Gam$
\begin{itemize}
\item there are exactly
three distinct vectors among the vectors $u_V(E_1)$, $\ldots$,
$u_V(E_k)$, where $E_1$, $\ldots$, $E_k$ are the edges incident
to~$V$,
\item an equality $u_V(E_i) = u_V(E_j)$,
where $i$ and $j$ are distinct elements of the set $\{1, 2, \ldots
k\}$, implies that $u_V(E_i) = u_V(E_j) = (1, 0)$.
\end{itemize}

If $(\overline\Gam, {\cal V}, h)$ is an $\CL$-curve, and~$x$ a point
of $\Gam$, then a cut $(\overline\Gam_x, h_x)$ of $(\overline\Gam,
h)$ at~$x$ gives rise to an $\CL$-curve $(\overline\Gam_x, {\cal
V}_x, h_x)$, where ${\cal V}_x= {\cal V} \cup {\cal V}^a$, and
${\cal V}^a$ is the set formed by the added vertices~$V^a_i$ of
$\overline\Gam_x$ such that for the unique edge~$E^a_i$ incident
to~$V^a_i$ one has $u_{V^a_i}(E^a_i) \ne (1, 0)$. The $\CL$-curve
$(\overline\Gam_x, {\cal V}_x, h_x)$ is also called a cut of
$(\overline\Gam, {\cal V}, h)$ at~$x$.

A {\it marked $\CL$-curve} $(\overline\Gam, {\cal V}, h, \bpp)$ is
an $\CL$-curve $(\overline\Gam, {\cal V}, h)$ equipped with a
$5$-tuple $\bpp = (\bpp^\flat, \bpp^\sharp, \bpp^1, \bpp^2,
\bpp^\nu)$ of disjoint finite sequences of distinct points
in~$\widehat\Gam$ such that
\begin{itemize}
\item each point in $\bpp^\flat$ is a univalent
vertex of $\widehat\Gam$,
\item each point in $\bpp^\sharp$, $\bpp^1$, and $\bpp^2$
is not a vertex of $\widehat\Gam$,
\item each point in $\bpp^\nu$
is a non-univalent vertex of $\widehat\Gam$,
\item
the connected components of the complement in $\overline\Gam$ of the
union of the sequences $\bpp^\flat$, $\bpp^\sharp$, $\bpp^1$,
$\bpp^2$, and $\bpp^\nu$ do not have loops, and each of these
components contains exactly one univalent vertex.
\end{itemize}
The elements of the union of $\bpp^\flat$, $\bpp^\sharp$, $\bpp^1$,
$\bpp^2$, and $\bpp^\nu$ are called {\it marked points}, and to
shorten the notation, we write $\bpp = \bpp^\flat \cup \bpp^\sharp
\cup \bpp^1 \cup \bpp^2 \cup \bpp^\nu$.

A left end~$E$ of a marked $\CL$-curve $(\overline\Gam, {\cal V}, h,
\bpp)$ is said to be of {\it $\alp$-type}, if it is incident to a
univalent vertex coinciding with a point in $\bpp^\flat$. We say
that an end~$E$ of a marked $\CL$-curve $(\overline\Gam, {\cal V},
h, \bpp)$ is {\it rigid}, if either~$E$ contains a marked point,
or~$E$ is left and of $\alp$-type. For any rigid end~$E$ of
$(\overline\Gam, {\cal V}, h, \bpp)$, a marked point certifying the
rigidity of~$E$ is unique and is called the {\it rigidity point}
of~$E$.

A rational marked $\CL$-curve is called {\it end-marked}, if each
marked point of this curve either coincides with a univalent vertex,
or belongs to an end. If the graph $\overline\Gam$ of such a curve
$(\overline\Gam, \CV, h, \bpp)$ is not a segment, then
$(\overline\Gam, \CV, h, \bpp)$ has exactly one non-rigid end.

Let $(\overline\Gam, \CV, h, \bpp)$ be a marked $\CL$-curve, and $x
\in \Gam$ a non-marked point. Among the connected components of
$\overline\Gam \setminus (\bpp \cup x)$ which are incident to~$x$,
there is exactly one component containing a univalent vertex. This
component is called the {\it free component} associated with~$x$.

If $(\overline\Gam, {\cal V}, h, \bpp)$ is a marked $\CL$-curve,
and~$x$ an arbitrary point of~$\Gam$, then the cut
$(\overline\Gam_x, {\cal V}_x, h_x)$ of $(\overline\Gam, {\cal V},
h)$ at~$x$ produces a marked $\CL$-curve $(\overline\Gam_x, {\cal
V}_x, h_x, \bpp_x)$ in the following way.

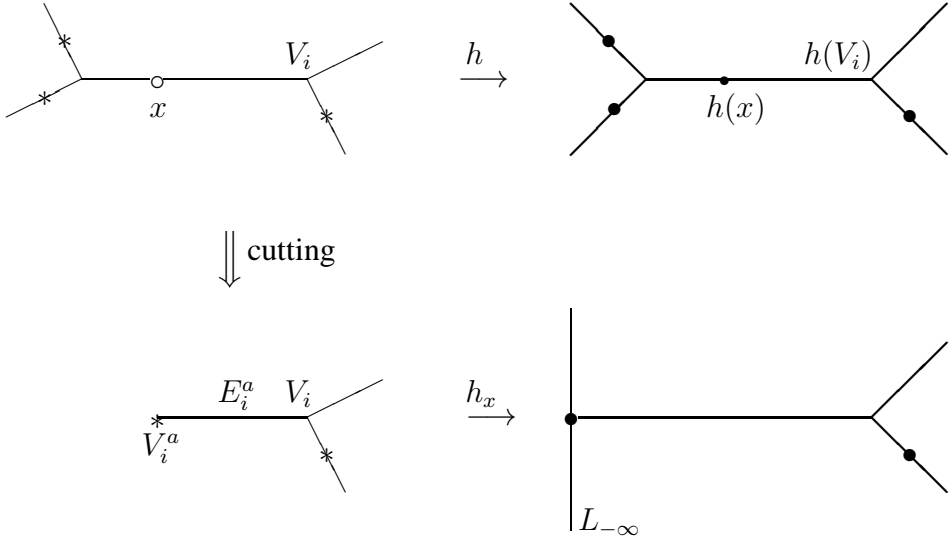
\begin{figure}
\setlength{\unitlength}{1cm}
\begin{picture}(13,7.5)(-1,0)
\thinlines \put(0,6){\line(2,1){1}}\put(0.5,7.5){\line(1,-2){0.5}}
\put(1,6.5){\line(1,0){0.9}}\put(2.09,6.5){\line(1,0){1.91}}
\put(4,6.5){\line(2,1){1}}\put(4,6.5){\line(1,-2){0.5} }
\put(2,2){\line(1,0){2}}\put(4,2){\line(2,1){1}}
\put(4,2){\line(1,-2){0.5}} \thicklines
\put(7.5,5.5){\line(1,1){1}}\put(7.5,7.5){\line(1,-1){1}}\put(8.5,6.5){\line(1,
0){3}}\put(11.5,6.5){\line(1,1){1}} \put(11.5,6.5){\line(1,-1){1}}
\put(7.6,2){\line(1,0){3.9}}\put(11.5,2){\line(1,1){1}}\put(11.5,2){\line(1,-1){
1}} \thinlines
\put(7.5,0.5){\line(0,1){1.4}}\put(7.5,2.05){\line(0,1){1.4}}

\put(1.9,6.36){$\circ$}\put(7.9,6.9){$\bullet$}\put(7.87,6){
$\bullet$}
\put(11.9,5.9){$\bullet$}\put(11.9,1.4){$\bullet$}\put(1.9,1.85){$\ast$}
\put(7.4,1.88){$\bullet$}\put(9.47,6.33){$\overset{\bullet}{}$}\put(0.4,6.15){${\ast}$}\put(0.67
,6.9){${\ast}$} \put(4.15,5.9){${\ast}$}\put(4.15,1.4){${\ast}$}
\put(1.9,6){$x$}\put(9.3,6){$h(x)$}\put(3.7,6.7){$V_i$}\put(10.6,6.7){$h(V_i)$}
\put(1.8,1.5){$V_i^a$}\put(3.7,2.2){$V_i$}\put(2.8,2.2){$E_i^a$}
\put(7.6,0.5){$L_{-\infty}$}\put(6,6.4){$\longrightarrow$}\put(6,1.9){
$\longrightarrow$}
\put(6.1,6.7){$h$}\put(6.1,2.2){$h_x$}\put(2.8,4){$\Big\Downarrow$}\put(3.2,4.1)
{\rm cutting}
\end{picture}
\caption{Cutting, III}\label{new-picture3}
\end{figure}

\begin{itemize}
\item For any $\aleph \in \{\flat, \sharp, 1, 2, \nu\}$
and any common element~$e$ (vertex or edge) of $\overline\Gam_x$ and
$\overline\Gam$, the intersections $\bpp^\aleph \cap e$ and
$(\bpp_x)^\aleph \cap e$ coincide.
\item The union of any added edge~$E^a_i$ of $\overline\Gam_x$
and the added vertex $V^a_i$ incident to $E^a_i$ contains at most
one marked point.
\item An added vertex $V^a_i$ of $\overline\Gam_x$
coincides with a marked point if and only if the following holds:
\begin{enumerate}
\item[(a)]
$u_{V^a_i}(E^a_i) = (1, 0)$, where $E^a_i$ is the edge incident to
$V^a_i$,
\item[(b)] the union of predecessors of $V^a_i$ and $E^a_i$
does not contain a marked point (notice that the union of
predecessors can contain at most one marked point),
\item[(c)] the vertex $V_i$
incident to $E^a_i$ and different from $V^a_i$ belongs to the free
component associated with~$x$ (see Figure~\ref{new-picture3}).
\end{enumerate}
\item An added edge $E^a_i$ of $\overline\Gam_x$
contains a marked point if and only if
\begin{enumerate}
\item[(a)] either
the union $U_i$ of predecessors of $V^a_i$ and $E^a_i$ contains a
marked point, where $V^a_i$ is the added vertex incident to $E^a_i$
(see Figure~\ref{new-picture4b});
\item[(b)] or $U_i$ does not contain a marked point,
the vertex $V_i$ incident to $E^a_i$ and different from $V^a_i$
belongs to the free component associated with~$x$, and
$u_{V^a_i}(E^a_i) \ne (1, 0)$ (see Figure~\ref{new-picture4a});
\end{enumerate}
moreover, if the union of the predecessors of $V^a_i$ and $E^a_i$
contains a point of $\bpp^\sharp$ (respectively, $\bpp^1$, $\bpp^2$,
$\bpp^\nu$), then the marked point of $E^a_i$ belongs to
$(\bpp_x)^\sharp$ (respectively, $(\bpp_x)^1$, $(\bpp_x)^2$,
$(\bpp_x)^1$).
\item If an added edge $E^a_i$ of $\overline\Gam_x$
contains a marked point and is incident to a non-univalent vertex
$V_i$ (such a vertex is necessarily a common vertex of
$\overline\Gam_x$ and $\overline\Gam$), then the distance (in
$\overline\Gam_x$) between $V_i$ and the marked point belonging to
$E^a_i$ is equal to the distance (in $\overline\Gam$) between $V_i$
and~$x$.
\end{itemize}

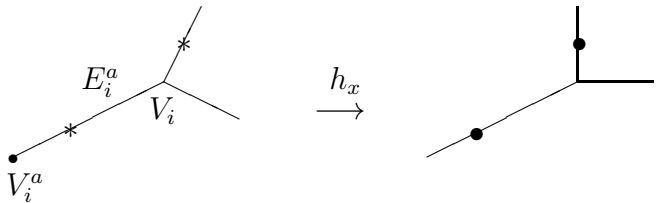
\begin{figure}
\setlength{\unitlength}{1cm}
\begin{picture}(10,15.5)(-2,0)
\thinlines \put(0.5,1){\line(2,1){2}}\put(2.5,2){\line(2,-1){1}}
\put(2.5,2){\line(1,2){0.5}}\put(2.15,5){\line(-2,3){0.6}}
\put(1.65,4.5){\line(1,1){0.5}} \put(2.2,5){\line(1,0){0.5}}
\put(1.4,6){\line(-1,0){1}}\put(1.55,6.02){\line(2,1){0.95}}\put(2.5,6.5){\line(2,-1){1
}} \put(2.5,6.5){\line(1,2){0.5}}\put(6,1){\line(2,1){2}}
\thicklines\put(8,2){\line(1,0){1}}
\put(8,2){\line(0,1){1}}\put(8,5){\line(0,1){1}}\put(7.5,4.5){\line(1,1){0.5
}}
\put(8.5,4.5){\line(-1,1){0.5}}\put(8,6){\line(-1,1){1}}\put(8,6){\line(2,1){1}}
\put(9,6.5){\line(0,1){1}}\put(9,6.5){\line(1,0){1}}

\put(7.92,4.9){$\bullet$}
\put(1.4,5.88){$\circ$}\put(8.3,6.1){$\bullet$}\put(8.76,6.9){
$\bullet$}
\put(6.55,1.2){$\bullet$}\put(7.92,2.4){$\bullet$}\put(0.42,0.82){$\overset{
\bullet}{}$}\put(7.96,5.84){$\overset{\bullet}{}$}
\put(2.05,4.9){$\ast$}\put(1.8,6.1){$\ast$}\put(2.65,6.9){$\ast$}\put(1.05,1.25){
$\ast$} \put(2.66,2.4){$\ast$}

\put(1.2,5.7){$x$}\put(2.3,6){$V_i$}\put(7.1,5.7){$h(x)$}\put(8.9,6){$h(V_i)$}
\put(2.3,1.5){$V_i$}\put(0.4,0.5){$V_i^a$}\put(1.4,1.9){$E_i^a$}
\put(4.5,5.8){$\longrightarrow$}\put(4.5,1.5){$\longrightarrow$}
\put(4.7,6.1){$h$}\put(4.7,1.9){$h_x$}\put(2,3.5){$\Big\Downarrow$}
\put(2.4,3.6){\rm cutting}\put(4.5,0){\rm (b)}

\thinlines \put(0.5,9){\line(2,1){2}}\put(2.5,10){\line(2,-1){1}}
\put(2.5,10){\line(1,2){0.5}}\put(2.15,13){\line(-2,3){0.6}}
\put(1.65,12.5){\line(1,1){0.5}} \put(2.15,13){\line(1,0){0.5}}
\put(1.5,14){\line(-1,0){1}}\put(1.5,14){\line(2,1){1}}\put(2.5,14.5){\line(2,
-1){1}} \put(2.5,14.5){\line(1,2){0.5}}\put(6,9){\line(2,1){2}}
\thicklines\put(8,10){\line(1,0){1}}
\put(8,10){\line(0,1){1}}\put(8,13){\line(0,1){1}}\put(7.5,12.5){\line(1,1){
0.5}}
\put(8.5,12.5){\line(-1,1){0.5}}\put(8,14){\line(-1,1){1}}\put(8,14){\line(2,1){
1}} \put(9,14.5){\line(0,1){1}}\put(9,14.5){\line(1,0){1}}

\put(7.7,12.7){$\bullet$}
\put(1.42,13.88){$\ast$}\put(1.41,13.88){$\circ$}\put(7.9,13.9){$\bullet$}\put(8.76,14.9){
$\bullet$}
\put(6.55,9.2){$\bullet$}\put(7.92,10.4){$\bullet$}\put(0.42,8.82){$\overset{
\bullet}{}$} \put(1.85,12.7){$\ast$}
\put(2.65,14.9){$\ast$}\put(1.05, 9.2){$\ast$}
\put(2.66,10.4){$\ast$}

\put(1.2,13.7){$x$}\put(2.3,14){$V_i$}\put(7.1,13.7){$h(x)$}\put(8.9,14){
$h(V_i)$}
\put(2.3,9.5){$V_i$}\put(0.4,8.5){$V_i^a$}\put(1.4,9.9){$E_i^a$}
\put(4.5,13.8){$\longrightarrow$}\put(4.5,9.5){$\longrightarrow$}
\put(4.7,14.1){$h$}\put(4.7,9.9){$h_x$}\put(2,11.5){$\Big\Downarrow$}
\put(2.4,11.6){\rm cutting}\put(4.5,8){\rm (a)}
\end{picture}
\caption{Cutting, IV}\label{new-picture4b}
\end{figure}

The marked $\CL$-curve obtained via this procedure is called a {\it
marked cut} of $(\overline\Gam, {\cal V}, h, \bpp)$ at~$x$.

Let $(\overline\Gam, {\cal V}, h, \bpp)$ be a marked $\CL$-curve,
and $x_1$, $\ldots$, $x_s$ a finite sequence of pairwise distinct
points of $\Gam$. We say that the set ${\cal X} = \{x_1, \ldots,
x_s\}$ is {\it sparse}, if no two distinct points of $\cal X$ belong
to the same edge of~$\Gam$ or to incident elements of~$\Gam$. If
$\cal X$ is sparse, define inductively a marked cut of
$(\overline\Gam, {\cal V}, h, \bpp)$ at $x_1$, $\ldots$, $x_s$ by
applying the cutting procedure successively at $x_1$, $\ldots$,
$x_s$. The marked $\CL$-curve obtained is also a marked cut of
$(\overline\Gam, {\cal V}, h, \bpp)$ at $x_{\sigma(1)}$, $\ldots$,
$x_{\sigma(s)}$, where $\sigma$ is an arbitrary permutation of the
set $\{1, \ldots, s\}$. The resulting curve is called a marked cut
of $(\overline\Gam, {\cal V}, h, \bpp)$ at the set ${\cal X}$ and is
denoted by $(\overline\Gam_{\cal X}, {\cal V}_{\cal X}, h_{\cal X},
\bpp_{\cal X})$. Notice that any marked cut of a non-degenerate
marked $\CL$-curve is non-degenerate, and any marked cut of a
pseudo-simple marked $\CL$-curve is pseudo-simple.

\begin{figure}
\setlength{\unitlength}{1cm}
\begin{picture}(10,7.5)(-2,0)
\thinlines \put(0.5,1){\line(2,1){2}}\put(2.5,2){\line(2,-1){1}}
\put(2.5,2){\line(1,2){0.5}}\put(2.15,5){\line(-2,3){0.6}}
\put(1.65,4.5){\line(1,1){0.5}} \put(2.2,5){\line(1,0){0.5}}
\put(1.4,6){\line(-1,0){1}}\put(1.55,6.02){\line(2,1){0.95}}\put(2.5,6.5){\line(2,-1){1
}} \put(2.5,6.5){\line(1,2){0.5}}\put(6,1){\line(2,1){2}}
\thicklines\put(8,2){\line(1,0){1}}
\put(8,2){\line(0,1){1}}\put(8,5){\line(0,1){1}}\put(7.5,4.5){\line(1,1){0.5
}}
\put(8.5,4.5){\line(-1,1){0.5}}\put(8,6){\line(-1,1){1}}\put(8,6){\line(2,1){1}}
\put(9,6.5){\line(0,1){1}}\put(9,6.5){\line(1,0){1}}

\put(7.92,4.9){$\bullet$}
\put(1.4,5.86){$\circ$}\put(7.4,6.4){$\bullet$}\put(8.76,6.9){
$\bullet$}\put(7.96,5.85){$\overset{\bullet}{}$}
\put(6.55,1.2){$\bullet$}\put(7.92,2.4){$\bullet$}\put(0.42,0.82){$\overset{
\bullet}{}$}
\put(2.05,4.9){$\ast$}\put(0.92,5.88){$\ast$}\put(2.65,6.9){$\ast$}\put(1.05,
1.2){$\ast$} \put(2.66,2.4){$\ast$}

\put(1.2,5.7){$x$}\put(2.3,6){$V_i$}\put(8.1,5.7){$h(x)$}\put(8.9,6){$h(V_i)$}
\put(2.3,1.5){$V_i$}\put(0.4,0.5){$V_i^a$}\put(1.4,1.9){$E_i^a$}
\put(4.5,5.8){$\longrightarrow$}\put(4.5,1.5){$\longrightarrow$}
\put(4.7,6.1){$h$}\put(4.7,1.9){$h_x$}\put(2,3.5){$\Big\Downarrow$}
\put(2.4,3.6){\rm cutting}
\end{picture}
\caption{Cutting, V}\label{new-picture4a}
\end{figure}
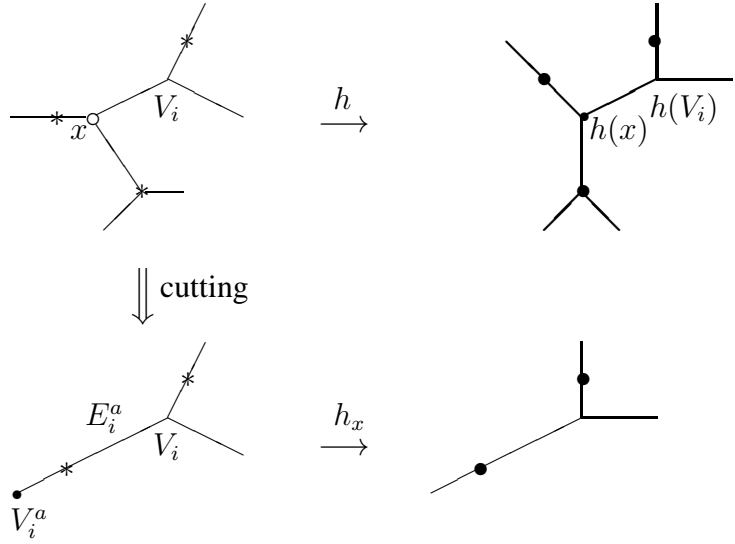

\subsection{Moduli spaces of
marked $\CL$-curves}\label{moduli-spaces}

We say that two marked $\CL$-curves $(\overline\Gam, {\cal V}, h,
\bpp)$ and $(\overline\Gam', {\cal V}', h',\bpp')$ have the same
{\it combinatorial type}, if there is a homeomorphism $\varphi:
\widehat\Gam \to \widehat\Gam'$ such that \begin{itemize}
\item $\varphi$ bijectively takes $\bpp$ onto $\bpp'$, respecting
their ordered $5$-tuple structures,
\item for any $V \in \overline\Gam$ and any edge~$E$ incident to~$V$,
the vectors $u_V(E)$ and $u_{\varphi(V)}(\varphi(E))$ coincide,
\item $w(E) = w(\varphi(E))$ for any edge $E\in\Gam^1$.
\end{itemize} If, in addition,
$h' \circ \varphi = h$, then $\varphi$ is called an {\it
isomorphism} and the curves $(\overline\Gam, \CV, h, \bpp)$ and
$(\overline\Gam', \CV', h', \bpp')$ are said to be {\it isomorphic}.
Note that, in this case, $\varphi$ defines an isometry of $\Gam$ and
$\Gam'$.

Let~$\Del$ be a finite multi-set of vectors in $\Z^2$ such that the
sum of all vectors in~$\Del$ is equal to~$0$, and let $k_ 1$, $k_2$,
$g$, $l$, and $r$ be non-negative integers such that $l \leq r$.
Denote by ${\cal M}(\Del, k_1, k_2, g, l, r)$ the set of isomorphism
classes of non-degenerate irreducible marked $\CL$-curves
$(\overline\Gam, {\cal V}, h, \bpp)$ of degree~$\Del$ and genus~$g$
such that
\begin{itemize}
\item the number of points in $\bpp$ is equal to $r + 1$,
\item the number of points in $\bpp^\flat$ is equal to~$l$,
\item the number of points in $\bpp^1 \cup \bpp^\nu$ is equal to $k_1$,
\item the number of points in $\bpp^2 \cup \bpp^\nu$ is equal to $k_2$.
\end{itemize}

Let $\Lambda(\Del, k_1, k_2, g, l, r)$ be the set of all possible
combinatorial types of marked $\CL$-curves whose isomorphism class
belongs to ${\cal M}(\Del, k_1, k_2, g, l, r)$. The set
$\Lambda(\Del, k_1, k_2, g, l, r)$ is clearly finite ({\it cf}., for
example, \cite[Proposition 3.7]{GM1} and \cite{Mi}). For a given
combinatorial type~$\lambda \in \Lambda(\Del, k_1, k_2, g, l, r)$,
denote by ${\cal M}^\lambda(\Del, k_1, k_2, g, l, r)$ the subset of
$\qquad$ ${\cal M}(\Del, k_1, k_2, g, l, r)$ formed by the
isomorphism classes of curves of type~$\lambda$.

There is a natural evaluation map $ev: {\cal M}(\Del, k_1, k_2, g,
l, r) \to (L_{-\infty})^l \times (\R^2)^{r + 1 - l}$ which
associates to any isomorphism class $[(\overline\Gam, {\cal V}, h,
\bpp)] \in {\cal M}(\Del, k_1, k_2, g, l, r)$ the sequence
${\widehat h}(\bpp)$. For a given combinatorial type $\lambda \in
\Lambda(\Del, k_1, k_2, g, l, r)$, denote by $ev^\lambda$ the
restriction of $ev$ on ${\cal M}^\lambda(\Del, k_1, k_2, g, l, r)$.

One can encode the elements $[(\overline\Gam, \CV, h, \bpp)]$ of
${\cal M}^\lambda(\Del, k_1, k_2, g, l, r)$ by
\begin{itemize}
\item the lengths
of all bounded edges of $\Gam$,
\item the image
$h(x) \in \R^2$ of some point $x \in \Gam$,
\item the coordinates of
the points of ${\widehat h}(\bpp^\flat)$ on $L_{-\infty}$,
\item the distances between~$P$ and $V_P$,
where~$P$ runs over all points in $\bpp^\sharp \cup \bpp^1 \cup
\bpp^2$, and $V_P$ is one of the vertices of the edge
containing~$P$.
\end{itemize}
This gives an identification of ${\cal M}^\lambda(\Del, k_1, k_2, g,
l, r)$ with the relative interior of a convex polyhedron in an
affine space, and, under this identification, the map $ev^\lambda$
becomes affine.

\begin{proposition}\label{finiteness}
For any element $\bx$ in $(L_{-\infty})^l \times (\R^2)^{r + 1 -
l}$, the inverse image $\qquad$ \mbox{$ev^{-1}(\bx) \subset {\cal
M}(\Del, k_1, k_2, g, l, r)$} is finite.
\end{proposition}

\begin{proof} Since the set $\Lambda(\Del, k_1, k_2, g, l, r)$ is
finite, it is enough to show that, for any $\lambda \in
\Lambda(\Del, k_1, k_2, g, l, r)$ and any element $\bx$ in
$(L_{-\infty})^l \times (\R^2)^{r + 1 - l}$, the inverse image
$\qquad$ $(ev^\lambda)^{-1}(\bx) \subset {\cal M}^\lambda(\Del, k_1,
k_2, g, l, r)$ is a finite set. Furthermore, since we can replace
any marked $\CL$-curve $(\overline\Gam, {\cal V}, h, \bpp)$ by the
collection of the irreducible components of a marked cut of
$(\overline\Gam, {\cal V}, h, \bpp)$ at $\bpp^\sharp \cup \bpp^1
\cup \bpp^2 \cup \bpp^\nu$
 (notice that the set $\bpp^\sharp \cup
\bpp^1 \cup \bpp^2 \cup \bpp^\nu$ is sparse), it is sufficient to
prove the statement in the situation, where $\lambda$ is a
combinatorial type of rational non-degenerate end-marked
$\CL$-curves.

For such a combinatorial type~$\lambda$ the statement is easily
proved by induction. Namely, let $(\overline\Gam, {\cal V}, h,
\bpp)$ be a marked $\CL$-curve of combinatorial type~$\lambda$, and
assume that~$\overline\Gam$ has at least two non-univalent vertices
(the other cases are evident). Since the curve $(\overline\Gam,
{\cal V}, h, \bpp)$ is rational, the graph $\overline\Gam$ contains
a non-univalent vertex~$V$ incident to exactly one bounded edge and
not incident to the non-rigid end. Denote by $E_1$, $\ldots$, $E_k$
the (rigid) ends incident to~$V$. Since the curve $(\overline\Gam,
{\cal V}, h, \bpp)$ is non-degenerate, among the ends $E_1$,
$\ldots$, $E_k$, we can find two ends $E_i$ and $E_j$ such that
$u_V(E_i) \ne u_V(E_j)$. Thus, the images under $\widehat h$ of the
rigidity points of $E_i$ and $E_j$ allow one to reconstruct the
image of $V$ under $\widehat h$. Modify the sequence $\bpp$ removing
from it the rigidity points of $E_1$, $\ldots$, $E_k$ and adding to
$\bpp^\nu$ a marked point~$P$ at~$V$ (as the last term of
$\bpp^\nu$). Denote the new sequence of marked points by $\bpp'$.
The graph of a marked cut of $(\overline\Gam, {\cal V}, h, \bpp')$
at~$P$ has fewer non-univalent vertices than $\overline\Gam$, and
one can apply the induction hypothesis. \end{proof}

\subsection{Symmetric ${\cal L}$-curves}

Let $(\overline\Gam,{\cal V},h,\bpp)$ be a marked ${\cal L}$-curve.
A homeomorphism $\xi:\widehat\Gam\to\widehat\Gam$ is called an {\it
involution}, if
\begin{itemize}
\item $\xi^2$ is the identity,
\item the restriction of $\xi$ on $\Gam$ is an isometry,
\item $\widehat h\circ\xi=\widehat h$,
\item $\xi$ is identical on
$\bpp^\flat \cup \bpp^\sharp \cup \bpp^\nu$, and takes $\bpp^1$
bijectively onto $\bpp^2$.
\end{itemize}
A marked $\CL$-curve $(\overline\Gam,{\cal V},h,\bpp)$ equipped with
an involution $\xi: \widehat\Gam \to \widehat\Gam$ is called {\it
symmetric}. A symmetric marked ${\cal L}$-curve is said to be {\it
irreducible} if it does not decompose into proper symmetric marked
${\cal L}$-subcurves. The {\it genus} of an irreducible symmetric
marked ${\cal L}$-curve $(\overline\Gam,{\cal V},h,\bpp,\xi)$ is
defined to be $b_1(\overline\Gam)-b_0(\overline\Gam)+1$.

We define a {\it combinatorial type} and an {\it isomorphism} of
symmetric marked ${\cal L}$-curves in the same way as in section
\ref{moduli-spaces} with an extra requirement that $\varphi$
commutes with the involutions. Given data $(\Del, k_1, k_2, g, l,
r)$ as in section~\ref{moduli-spaces}, we similarly introduce the
set ${\cal M}^{\text{\rm sym}}(\Del, k_1, k_2, g, l, r)$ of
isomorphism classes of irreducible symmetric marked ${\cal
L}$-curves of genus $g$, and consider the evaluation map
$ev^{\text{\rm sym}}: {\cal M}^{\text{\rm sym}}(\Del, k_1, k_2, g,
l, r) \to (L_{-\infty}) \times (\R^2)^{r + 1 - l}$. Clearly, ${\cal
M}^{\text{\rm sym}}(\Del, k_1, k_2, g, l, r)$ is nonempty only if
$k_1 = k_2$. The following statement can be deduced from
Proposition~\ref{finiteness}.

\begin{proposition}\label{sym}
(1) The set of combinatorial types of symmetric marked ${\cal
L}$-curves whose isomorphism classes belong to ${\cal M}^{\text{\rm
sym}}(\Del, k_1, k_2, g, l, r)$ is finite.

(2) For any element $\bx$ in $(L_{-\infty})^l \times (\R^2)^{r + 1 -
l}$, the inverse image of $\bx$ under the evaluation map
$ev^{\text{\rm sym}}: {\cal M}^{\text{\rm sym}}(\Del, k_1, k_2, g,
l, r) \to (L_{-\infty})^l \times (\R^2)^{r + 1 - l}$ is finite.
\end{proposition}

\section{Algebraic Caporaso-Harris type
formulas}\label{algebraic-formulas}

\subsection{Families of curves on Del Pezzo surfaces}\label{secn1}
Here, we establish some properties of generic points of "generalized
Severi varieties". These auxiliary results are close to similar
statements in \cite{CH,Va} used in the proof of the recursive
formulas, but our setting is slightly different and we need a more
detailed information on generic elements of the Severi varieties
considered.

Let $\Sig=\PP^2_q$ be the complex projective plane blown up at $0\le
q\le 5$ generic points, $E_i$, $1\le i\le q$, the exceptional curves
of the blow up, $L\in\Pic(\Sig)$ the pull-back of a line in $\PP^2$,
and $E$ a smooth rational curve linearly equivalent to
\begin{equation}\begin{cases}L,\quad & \text{if}\,\, q\le2,\\ L-E_3,\quad &
\text{if}\,\, q=3, \\ L-E_3-E_4,\quad & \text{if}\,\, q\ge4
.\end{cases}\label{en61}\end{equation} Denote by $\Pic_+(\Sig,E)$
and $\Pic(\Sig,E)$ the semigroups generated by effective irreducible
divisors $D\in\Pic(\Sig)$ such that $DE> 0$ or $DE\ge 0$,
respectively.

Following \cite[Section 2]{Va}, for a given effective divisor
$D\in\Pic(\Sig)$ and nonnegative integers $g,n$, we denote by
$\overline{\cal M}_{g,n}(\Sig,D)$ the moduli space of triples $(\hat
C, \hat\bp, \nu)$, where $\hat C$ is a genus~$g$ connected nodal
curve, $\hat\bp$ is a collection of $n$ marked points of $\hat C$,
and $\nu:\hat C\to\Sigma$ is a stable map such that $\nu_*\hat
C\in|D|$. The triples with a smooth curve $\hat C$ form an open
dense subset in $\overline{\cal M}_{g,n}(\Sig,D)$.

Let $\Z^\infty_+$ be the direct sum of countably many additive
semigroups \mbox{$\Z_+=\{m\in\Z\ |\ m\ge 0\}$}. Denote by $\theta_k$
the generator of the $k$-th summand, $k\in\N$. For
$\alp=(\alp_1,\alp_2,...)\in\Z^\infty_+$ put
$$
\displaylines{ \|\alp\|=\sum_{k=1}^\infty\alp_k,\quad
I\alp=\sum_{k=1}^\infty k\alp_k\ , \cr \mt(\alp,i)=k,\quad
\text{if}\quad \sum_{j<k}\alp_j< i\le\sum_{j\le k}\alp_j,\
i=1,...,\|\alp\|\ . }
$$
Let a divisor $D\in \Pic(\Sig,E)$, an integer $g$, and two elements
$\alp,\bet\in\Z^\infty_+$ satisfy
\begin{equation}
g\le g(\Sig,D)=\frac{D^2+DK_\Sig}{2}+1,\quad I\alp+I\bet=DE\ .
\label{en1}
\end{equation}
Given a sequence $\bp^\flat=(p_i)_{1\le i\le \|\alp\|}$ of
$\|\alp\|$ distinct points of $E$, define ${\cal V}_\Sig(D, g, \alp,
\bet, \bp^\flat)\subset\overline{\cal M}_{g,\|\alp\|}(\Sig,D)$ to be
the closure of the set of elements $(\hat C, \hat\bp, \nu)$ subject
to the following restrictions:
\begin{itemize}
\item $\hat C$ is smooth, $\hat\bp=(\hat p_i)_{1\le i\le\|\alp\|}$, $\nu(\hat p_i)=p_i$
for all $i=1,...,\|\alp\|$,
\item the divisor
$\bd=\nu^*(E)$ is of the form
\begin{equation}\bd=\sum_{i=1}^{\|\alp\|}\mt(\alp,i)\cdot\hat
p_i+\sum_{i=1}^{\|\bet\|}\mt(\bet,i)\cdot\hat q_i\ ,\label{ediv}
\end{equation}
where $(\hat q_i)_{1\le i\le\|\bet\|}$ is a sequence of $\|\bet\|$
distinct points of $\hat C\backslash\hat\bp$.
\end{itemize}
Put $$R_\Sig(D,g,\bet)=-D(E+K_\Sig)+\|\bet\|+g-1\ .$$ For a
component $V$ of ${\cal V}_\Sig(D,g,\alp,\bet,\bp^\flat)$, define
the {\it intersection dimension} $\Idim V$ to be the dimension of
the image of $V$ in the linear system $|D|$. Denote by $\Idim{\cal
V}_\Sig(D,g,\alp,\bet,\bp^\flat)$ the maximum of intersection
dimensions if the components.

\begin{lem}\label{ln1}
Consider the surface $\Sig=\PP^2_5$, the smooth rational curve $E$
linearly equivalent to $L-E_3-E_4$, and an irreducible curve
$C'\subset\Sig$ different from $E$. Let $D\in
\Pic(\Sig,E)\backslash\{0\}$, an integer~$g$, and
$\alp,\bet\in\Z^\infty_+$ satisfy (\ref{en1}), and let
$\bp^\flat=(p_i)_{1\le i\le\|\alp\|}$ be a generic sequence of
points of $E$.

If ${\cal V}_\Sig(D,g,\alp,\bet,\bp^\flat)\ne\emptyset$, then
$R_\Sig(D,g,\bet)\ge0$ and each component $V$ of $\qquad$ ${\cal
V}_\Sig(D,g,\alp,\bet,\bp^\flat)$ satisfies
\begin{equation}
\Idim V\le R_\Sig(D,g,\bet)\ .\label{enov1}
\end{equation}
If (\ref{enov1}) turns into equality and $D$ is not of the form
$mD_0$, where $m\ge 2$ and $D_0=E_3,E_4$, or $L-E_i-E_j$,
$\{i,j\}\subset\{1,2,5\}$, then for a generic element $(\hat C,
\hat\bp, \nu) \in V$ the map~$\nu$ is an immersion birational onto
$C=\nu(\hat C)$, where $C$ is a nodal curve nonsingular along $E$.
Furthermore, if in addition $R_\Sig(D,g,\bet)>0$, then the family
$V$ has no base points outside $\bp^\flat$, and $C$ crosses $C'$
transversally.
\end{lem}

\begin{proof}
First, we prove the lemma under an additional assumption that for a
generic element $(\hat C,\hat\bp,\nu)$ of the component $V\subset
{\cal V}_\Sig(D,g,\alp,\bet,\bp^\flat)$, the curve $\hat C$ is
mapped by $\nu$ birationally to its image $C=\nu(\hat C)$.

The relation (\ref{enov1}) is evident if $D$ is a $(-1)$-curve as
well as if $D=dL-k_1E_1-...-k_5E_5$ with $d\le 2$ or with $k_i<0$
for some $i$. So, we assume that $D=dL-k_1E_1-...-k_5E_5$ with $d\ge
3$, $k_1,...,k_5\ge 0$.

Due to generic position of $\bp^\flat$, one has
$$
\Idim V\le \Idim W-\|\alp\|\
$$
for at least one of the components $W$ of ${\cal
V}_\Sig(D,g,0,\alp+\bet,\emptyset)$. Hence, to prove (\ref{enov1}),
it is sufficient to assume that $\alp=0$ and $\bp^\flat=\emptyset$,
and to check that the intersection dimension of any component of
${\cal V}_\Sig(D,g,0,\bet,\emptyset)$ is at most $R_\Sig(D,g,\bet)$.
To shorten the notation, we write (within the present proof) ${\cal
V}_\Sig(D,g,\bet)$ for ${\cal V}_\Sig(D,g,0,\bet,\emptyset)$.

The remaining part of the proof of (\ref{enov1}) literally follows
the lines of \cite{CH} and is based on the following numerical
observations.
\begin{enumerate}
\item[(E1)] The conclusion of \cite[Corollary 2.4]{CH}
reads in our situation as $$\Idim {\cal V}_\Sig(D,g,(DE)\theta_1)\le
-DK_\Sig+g-1\ ,$$ and it holds since the hypothesis of
\cite[Corollary 2.4]{CH}, which is equivalent to $DK_\Sig<0$, is
true for any effective divisor $D\in\Pic(\Sig)$.
\item[(E2)] The inequality
$\deg(\nu^*{\cal O}_{\PP^2}(1)(-\bd))\ge0$ in \cite[Page 363]{CH}
($\bd$ is defined by (\ref{ediv})) reads in our situation as $DE\ge
I\bet-\|\bet\|$, and it holds true since $I\bet=DE$ under our
assumptions.
\item[(E3)] The
inequality $\deg(c_1({\cal N}(-\bd))\otimes\omega^{-1}_{\hat C})>0$
in \cite[The last paragraph in page 363]{CH} (${\cal N}$ is the
normal sheaf on $\hat C$ and $\omega_{\hat C}$ is the dualizing
bundle) reads in our setting as
$$\displaylines{
-DK_\Sig+ 2g-2 -\deg\bd + 2 - 2g \cr
=-D(K_\Sig+E)+\|\bet\|=(2d-k_1-k_2-k_5)+\|\bet\| > 0, }
$$
and it holds true since $2d-k_1-k_2-k_5\ge 2$ (the latter inequality
follows from B\'ezout's bound applied to the intersection of
$C=\nu(\hat C)$ with an appropriate curve from the linear system
$|2L-E_1-E_2-E_5|$).\end{enumerate} As in the end of the proof of
\cite[Proposition 2.1]{CH}, it implies that the intersection
dimension of any component of ${\cal V}_\Sig(D,g,\bet)$ does not
exceed
$$\deg(c_1({\cal
N}(-\bd))-g+1=R_\Sig(D,g,\bet)\ ,
$$
which completes the proof of (\ref{enov1}).

A stronger inequality $\deg(c_1({\cal
N}(-\bd))\otimes\omega^{-1}_{\hat C})\ge 2$, which is indeed
established in (E3), implies that $\nu:\hat C\to \Sig$ is an
immersion away from $\nu^{-1}( E)$ (see \cite[First paragraph of the
proof of Proposition 2.2]{CH}). Similarly, $\nu$ is an immersion at
$\nu^{-1}( E)$ as soon as $\deg(c_1({\cal
N}(-\bd))\otimes\omega^{-1}_{\hat C})\ge 4$, or, equivalently,
\begin{equation}(2d-k_1-k_2-k_5)+\|\bet\|\ge 4\
.\label{ejan6_3}\end{equation} On the other hand, assuming that
$C=\nu(\hat C)$ has a singular local branch centered at a point
$z\in E$, have $\|\bet\|>0$ and $2d-k_1-k_2-k_5\ge 3$, where the
latter inequality follows from B\'ezout's bound applied to the
intersection of $C$ with a curve belonging to $|2L-E_1-E_2-E_5|$ and
passing through $z$ and another point of $C$.

From now, on we suppose that  $\nu$ is an immersion. To show that
$C$ is nonsingular along $E$ and nodal, we may assume that $C\cap
E\ne\emptyset$, {\it i.e.}, $\|\bet\|>0$. Indeed, otherwise, we
either replace $E$ by another $(-1)$-curve, or, if $C$ does not meet
any $(-1)$-curve, we blow them down and reduce the problem to the
planar case.

To check the remaining statements, we argue by contradiction.
Namely, assuming that one of them fails, we derive that necessarily
$\Idim {\cal V}_\Sig(D,g,\bet)<R_\Sig(D,g,\bet)$.

Suppose that $\nu$ takes the points $q_1,...,q_s$ ($s\ge2$) of the
divisor $\bd$ to the same point $z\in E$. Fixing the position of $z$
in $E$, we obtain a subvariety $U\subset {\cal V}_\Sig(D,g,\bet)$ of
dimension $\Idim U\ge \Idim {\cal V}_\Sig(D,g,\bet)-1$. On the other
hand, the same argument as in the proof of \cite[Proposition
2.1]{CH} gives
$$\Idim U\le h^0(\hat C,
{\cal N}(-\bd-\bd')),$$ where $\bd'=q_1+...+q_s$. Verifying that
$c_1({\cal N}(-\bd-\bd'))\otimes\omega^{-1}_{\hat C}$ is positive on
$\hat C$ and applying \cite[Observation 2.5]{CH}, we get (cf.
\cite[Page 364]{CH}) $$\Idim {\cal V}_\Sig(D,g,\bet)\le\Idim U+1\le
h^0(\hat C,{\cal N}(-\bd-\bd'))+1$$
$$=\deg(c_1({\cal
N}(-\bd-\bd'))-g+2=-DK_\Sig+2g-2-\deg(\bd+\bd')-g+2$$
$$=-D(K_\Sig+E)+g+\|\bet\|-s<R_\Sig(D,g,\bet)\ .$$ The above
positivity is equivalent to
\begin{equation}-DK_\Sig+2g-2-\deg(\bd+\bd')-2g+2=(2d-k_1-k_2-k_5-s)
+\|\bet\| > 0,\label{ejan6_4}\end{equation} and it holds due to
$\|\bet\|>0$ and $2d-k_1-k_2-k_5-s>0$, where the latter inequality
follows from B\'ezout's bound applied to the intersection of $C$
with a curve belonging to \mbox{$|2L-E_1-E_2-E_5|$} and passing
through $z$ and another point of $C$.

Suppose that for some point $z\in\Sig\backslash E$ the set
$\nu^{-1}(z)$ consists of $s\ge 3$ points. Fixing the position of
the point $z$, we obtain a subvariety $U\subset {\cal
V}_\Sig(D,g,\bet)$ of dimension $\Idim U\ge\Idim {\cal
V}_\Sig(D,g,\bet)-2$. On the other hand, by the same arguments as
above we have again the inequality (\ref{ejan6_4}) and an upper
bound
$$\Idim V\le h^0(
\hat C,{\cal N}(-\bd-\bd'))\ ,$$ where $\bd'=\nu^{-1}(z)$. Due to
(\ref{ejan6_4}), the bundle $c_1({\cal
N}(-\bd-\bd'))\otimes\omega^{-1}_{\hat C}$ is positive on $\hat C$,
and thus applying \cite[Observation 2.5]{CH} we get
$$\Idim
{\cal V}_\Sig(D,g,\bet)\le\Idim U+2\le h^0( \hat C,{\cal
N}(-\bd-\bd'))+2$$
$$=\deg(c_1({\cal
N}(-\bd-\bd'))-g+3=-DK_\Sig+2g-2-\deg(\bd+\bd')-g+3$$
$$=-D(K_\Sig+E)+g+\|\bet\|-s+1<R_\Sig(D,g,\bet)\ .$$

Suppose that $\nu^{-1}(z)=w_1+w_2$, $w_1\ne w_2\in\hat C$ for some
point $z\in\Sig\backslash E$, and the two local branches of
$C=\nu(\hat C)$ at $z$ are tangent. Fixing the position of $z$ and
the direction of the tangent, we obtain a subvariety $U\subset {\cal
V}_\Sig(D,g,\bet)$ of dimension $\Idim U\ge\Idim {\cal
V}_\Sig(D,g,\bet)-3$. As in the two preceding computations, we have
$$\Idim U\le h^0(\hat C,{\cal N}(-\bd-\bd'))\ ,$$ where $\bd'=2w_1+2w_2$. The
inequality
$$-DK_\Sig+2g-2-\deg(\bd+\bd')-2g+2=(2d-k_1-k_2-k_5-4)+\|\bet\|>0\
,$$ which we derive from $\|\bet\|>0$ and B\'ezout's bound applied
to the intersection of $C$ with a curve belonging to
$|2L-E_1-E_2-E_5|$, passing through $z$ and tangent to $C$, results
in
$$\Idim {\cal V}_\Sig(D,g,\bet)\le\Idim U+3\le h^0(
\hat C,{\cal N}(-\bd-\bd'))+3$$
$$=\deg(c_1({\cal
N}(-\bd-\bd'))-g+4=-DK_\Sig+2g-2-\deg(\bd+\bd')-g+4$$
$$=-D(K_\Sig+E)+g+\|\bet\|-s-2<R_\Sig(D,g,\bet)\ .$$

Assuming that $R_\Sig(D,g,\bet)>0$ and ${\cal
V}_\Sig(D,g,\alp,\bet,\bp^\flat)$ has a base point
$z\in\Sig\backslash\bp^\flat$, we obtain $$\Idim {\cal
V}_\Sig(D,g,\alp,\bet,\bp^\flat)\le h^0(\hat C,{\cal N}(-\bd-\bd'))\
,$$ where $\bd'=\nu^{-1}(z)$. The positivity of the bundle
$c_1({\cal N}(-\bd-\bd'))\otimes\omega^{-1}_{\hat C}$ on $\hat C$
reduces to the inequality
$$(2d-k_1-k_2-k_5)+\|\bet\|-1>0,$$
which holds by B\'ezout's bound. It allows us to conclude that
$$\Idim
{\cal V}_\Sig(D,g,\alp,\bet,\bp^\flat)\le h^0(\hat C,{\cal
N}(-\bd-\bd'))$$
$$=-DK_\Sig+2g-2-\deg(\bd+\bd')-g+1$$
$$=-D(K_\Sig+E)+g+\|\bet\|-2<R_\Sig(D,g,\bet).
$$

Finally, if $R_\Sig(D,g,\bet)>0$ then according to the same
arguments as above the existence of a tangency point $z$ with $C'$
would lead to the following computation:
$$\Idim {\cal V}_\Sig(D,g,\bet)\le
h^0(\hat C,{\cal N}(-\bd-\bd'))$$
$$=\deg(c_1({\cal
N}(-\bd-\bd'))-g+1=-DK_\Sig+2g-2-\deg(\bd+\bd')-g+1$$
$$=-D(K_\Sig+E)+g+\|\bet\|-s-2<R_\Sig(D,g,\bet)\ ,$$ where $\bd'=\nu^{-1}(z)$.

Thus, the proof of Lemma \ref{ln1} is completed under the assumption
that $\nu$ is birational on its image. Hence, to end the proof it
remains to assume instead that for a generic element $(\hat C, \nu,
\hat\bp)$ of ${\cal V}_\Sig(D,g,\alp,\bet,\bp^\flat)$ the map $\nu$
is an $m$-fold covering to its image, $m\ge2$. But, in such a case,
we would have $D=mD_0$, $2-2g=m(2-2g')-r$, where $g'$ is the
geometric genus of $C=\nu(\hat C)$, $r$ is the total ramification
multiplicity of the covering, and $|(C\cap
E)\backslash\bp^\flat|\le\|\bet\|$. Therefore, using (\ref{enov1})
for the case of birational $\nu$, we get (cf. \cite[Page 62]{Va})
$$\Idim {\cal V}_\Sig(D,g,\alp,\bet,\bp^\flat)\le-(K_\Sig+E)D_0+g'-1+|(C\cap
E)\backslash\bp^\flat|$$
$$\le-\frac{(K_\Sig+E)D}{m}+\frac{g-1-r/2}{m}+\|\bet\|\le-(K_\Sig+E)D+g-1+\|\bet\|=R_\Sig(D,g,\bet),$$
the latter inequality coming from the nef property of $-(K_\Sig+E)$.
The case of equality $\Idim {\cal
V}_\Sig(D,g,\alp,\bet,\bp^\flat)=R_\Sig(D,g,\bet)$ leaves the only
possibility $(K_\Sig+E)D_0=0$. The latter equality holds only for
$D_0=0$ or $D_0\in\{E_3,E_4,L-E_i-E_j,\ \{i,j\}\subset\{1,2,5\}\}$.
\end{proof}

\begin{lem}\label{ltn2}
Consider the surface $\Sig=\PP^2_5$ and the smooth rational curve
$E$ linearly equivalent to $L-E_3-E_4$. Let $D\in \Pic_+(\Sig,E)$,
an integer~$g$, and $\alp,\bet\in\Z^\infty_+$ satisfy (\ref{en1}),
and let $\bp^\flat=(p_i)_{1\le i\le\|\alp\|}$ be a generic sequence
of points of $E$. Assume that $R_\Sig(D,g,\bet)=0$, and $D$ is
represented by a reduced irreducible curve. Then, ${\cal
V}_\Sig(D,g,\alp,\bet,\bp^\flat)$ consists of one element for the
following quadruples $(D,g,\alp,\bet)$:
\begin{enumerate}
\item[(i)] $(E_i,0,0,\theta_1)$ for $i=3,4$,
\item[(ii)] $(L-E_i-E_j,0,0,\theta_1)$
for $i, j = 1, 2, 5$, $i \ne j$,
\item[(iii)] $(L-E_i,0,\theta_1,0)$ for $i=1,2,5$,
\item[(iv)] $(2L-E_1-E_2-E_i-E_5,0,\theta_1,0)$ for $i=3,4$,
\item[(v)] $(2L-E_1-E_2-E_5,0,\alp,0)$ as long as $I\alp=2$,
\end{enumerate}
and ${\cal V}_\Sig(D,g,\alp,\bet,\bp^\flat)=\emptyset$ in all the
other cases.
\end{lem}

\begin{proof} Clearly, ${\cal V}_\Sig(D,g,\alp,\bet,\bp^\flat)$
consists of one element in the cases (i)-(v). Assuming that ${\cal
V}_\Sig(D,g,\alp,\bet,\bp^\flat)\ne\emptyset$ and
$R_\Sig(D,g,\bet)=0$, we will derive that $(D,g,\alp,\bet)$
necessarily belongs to the list (i)-(v). Let $D$ be equal to
$dL-k_1E_1-...-k_5E_5$, where $d, k_1,\dots, k_5$ are integers and
$d \geq 0$.

Suppose that $d=0$. Since $DE=d - k_3 - k_4 > 0$, we obtain $D=E_i$,
$i=3,4$; this is the case (i).

If $d>0$ then $k_1,...,k_5\ge 0$. Thus, if $d=1$, then
$DE=d-k_3-k_4>0$ yields $k_3=k_4=0$. Taking additionally into
account that
$$R_\Sig(D, g, \bet) = 2d-k_1-k_2-k_5+|\bet|+g-1=
2-k_1-k_2-k_5+|\bet|+g-1=0\ ,$$ we derive that $k_1+k_2+k_5\ge 1$.
On the other hand, the B\'ezout upper bound applied to the
intersection of a curve $C=\nu(\hat C)$ as $\{\nu:\hat C\to\Sig\}\in
{\cal V}_\Sig(D,g,\alp,\bet,\bp^\flat)$ with an irreducible curve
belonging to $|2L - E_1 - E_2 - E_5|$ implies $k_1 + k_2 + k_5 \leq
2$. Hence, for $D=L-E_i-E_j$, $i, j = 1, 2, 5$, $i \ne j$, we
necessarily obtain $\alp=0$, $\bet=\theta_1$, $g=0$; this is the
case (ii). Similarly, for $D=L-E_i$, $i=1,2,5$, we necessarily
obtain $\bet=0$, $g=0$, and $\alp=\theta_1$; this is the case (iii).

Assume now that $d = 2$. In view of the relations
$$DE=2-k_3-k_4>0,\quad R_\Sig(D, g, \bet) =
4-k_1-k_2-k_5+|\bet|+g-1=0\ ,$$ and the irreducibility of $D$, we
obtain
$$\bet = 0, \;\;\; g = 0, \;\;\;
k_1 = k_2 = k_5 = 1, \;\;\; 0 \le k_3 + k_4 \le 1,$$ coming to the
cases (iv) and (v).

Finally, if $d\ge 3$, then intersecting $D$ with the lines from
$|L-E_i-E_j|$, we derive that $k_i+k_j\le d$, which implies
$$R_\Sig(D, g,  \bet)=2d-k_1-k_2-k_5+|\bet|+g-1\ge
2d-\frac{3}{2}d+|\bet|+g-1\ge\frac{d}{2}-1>0\qedhere$$\end{proof}

\begin{lem}\label{lnov1}
Consider the surface $\Sig=\PP^2_5$ and the smooth rational curve
$E$ linearly equivalent to $L-E_3-E_4$. Let $D\in\Pic(\Sig,E)$, an
integer $g$, and $\alp,\bet\in\Z^\infty_+$ satisfy (\ref{en1}) and
$R_\Sig(D,g,\bet)>0$, and let $\bp^\flat=(p_i)_{1\le i\le\|\alp\|}$
be a generic sequence of points of $E$. Let $V$ be a component of
${\cal V}_\Sig(D,g,\alp,\bet,\bp^\flat)$ of intersection dimension
$R_\Sig(D,g,\bet)$, and let $\{( \hat C_t,\nu_t,
\hat\bp_t)\}_{t\in(\C,0)}\subset V$ be a generic one-parameter
family such that $\hat C_t$ is smooth connected as $t\ne0$, and
$\nu_0(\hat C_0)\supset E$.

Then $\hat C_0=\widetilde E\cup C^{(1)}\cup...\cup C^{(m)}\cup Z$,
where \begin{enumerate}\item[(1)] $\nu_0$ takes $\widetilde E$
isomorphically onto $E$, takes each component $C^{(i)}$, $1\le i\le
m$, to a curve different from $E$ and crossing $E$, and contracts
the components of $Z$ to points, \item[(2)] $\nu_0$ takes
$C^{(1)}\cup...\cup C^{(m)}$ birationally onto its image $C\in|D-E|$
which is a reduced nodal curve, nonsingular along its intersection
with $E$, \item[(3)] for each $i=1,...,m$, the map
$\nu_0:C^{(i)}\to\Sig$ represents a generic element in a component
of some ${\cal
V}_\Sig(D^{(i)},g^{(i)},\alp^{(i)},\bet^{(i)},(\bp^\flat)^{(i)})$ of
dimension $R_\Sig(D^{(i)},g^{(i)},\bet^{(i)})$ such that
\begin{enumerate}\item[(a)] $\sum_{i=1}^mD^{(i)}=D-E$,
\item[(b)] $(\bp^\flat)^{(i)}$, $i=1,...,m$, are disjoint subsets of
$\bp^\flat$, \item[(c)]
$\sum_{i=1}^mR_\Sig(D^{(i)},g^{(i)},\bet^{(i)})=R_\Sig(D,g,\bet)-1$,
\item[(d)] each quadruple $(D^{(i)},g^{(i)},\alp^{(i)},\bet^{(i)})$ with $n_i=0$ appears in the list
$\qquad$ $(D^{(i)},g^{(i)},\alp^{(i)},\bet^{(i)})$, $i=1,...,m$, at
most once.\end{enumerate}
\end{enumerate}
\end{lem}

\begin{proof} The central element $(\hat C_0,\nu_0, \hat\bp_0)$ must belong to a substratum of
intersection dimension $R_\Sig(D,g,\bet)-1$. Therefore, the fact
that there is only one component of $\hat C_0$ taken to $E$ and the
corresponding map is an isomorphism can be proven as in \cite[Proof
of Theorem 5.1, Case II]{Va}, where the argument is based on the
inequality (\ref{enov1}) and the equality $-(K_\Sig+E)E=2$.

In the further steps we argue as in \cite[Section 3]{CH}. Namely, we
replace the given family $\nu_t:\hat C_t\to\Sig$ by a family with
the same generic fibres and a semi-stable central fibre so that (cf.
conditions (a), (b), (c), (d) in \cite[Section 3.1]{CH}, and
assumptions (b), (c) in \cite[Section 3.2]{CH}):
\begin{itemize}\item the family is represented by a smooth surface
$Y$ and two morphisms $\pi_\Sig:Y\to\Sig$, $\pi_\C:Y\to(\C,0)$, and
for each $t\ne 0$, $\pi_\Sig:Y_t= \pi_\C^{-1}(t)\to\Sig$ is
isomorphic to $\nu_t:\hat C_t\to\Sig$,
\item the central fibre $Y_0$ is a connected nodal
curve splitting into the union of the following parts: the component
$\widetilde E$ isomorphically mapped by $\pi_\Sig$ onto $E$, the
components $C^{(1)}$, ..., $C^{(m)}$ mapped by $\pi_\Sig$ to curves,
and $\widetilde Z$, the union of the components contracted by
$\pi_\Sig$ to points, \item the components of $Z$ are rational and
form disjoint chains joining $\widetilde E$ and $C^{(1)}\cup...\cup
C^{(m)}$, \item the sections $t\in(\C,0)\backslash\{0\}\mapsto\hat
p_{i,t}$, $1\le i\le\|\alp\|$, and
$t\in(\C,0)\backslash\{0\}\mapsto\hat q_{i,t}$, $1\le i\le\|\bet\|$,
defined by (cf. (\ref{ediv}))
$$\bd_t=\pi_\Sig^*(E\cap\pi_\Sig(Y_t))=\sum_{i=1}^{\|\alp\|}\mt(\alp,i)\cdot\hat
p_{i,t}+\sum_{i=1}^{\|\bet\|}\mt(\bet,i)\cdot\hat q_{i,t}\subset
Y_t\ ,$$ $$\pi_\Sig(\hat p_{i,t})=p_i\in\bp^\flat,\ 1\le
i\le\|\alp\|\ ,$$ close up at $t=0$ into disjoint global sections
avoiding singularities of $Y_0$ and the components of $Z$, and such
that $\hat q_{i,t}\in C^{(1)}\cup...\cup C^{(m)}$ for all
$i=1,...,\|\bet\|$, \item for each $i=1,...,m$, the triple $
(C^{(i)},\pi_\Sig, \hat\bp^{i)})$ with $\hat\bp^{(i)}=\hat\bp_0\cap
C^{(i)}$ represents a generic element of some ${\cal
V}_\Sig(D^{(i)},g^{(i)},\alp^{(i)},\bet^{(i)},(\bp^\flat)^{(i)})$,
where $\sum D^{(i)}=D-E$, the sequences $(\bp^\flat)^{(i)}$,
$i=1,...,m$, are disjoint subsequences of $\bp^\flat$, and
\begin{equation}\sum_{i=1}^mR_\Sig(D^{(i)},g^{(i)},\bet^{(i)})=R_\Sig(D,g,\bet)-1\
.\label{edim}\end{equation}
\end{itemize} The proof literally follows the argument of
\cite[Section 3]{CH}, whose main ingredient is the inequality
(\ref{enov1}).

Blow down all the components of $Z$ and obverse now that
(\ref{edim}) can be rewritten in the form
\begin{equation}
\chi(\hat C_t)=
\sum_{i=1}^m\chi(C^{(i)})+\chi(E)-2\|\sum_{i=1}^m\bet^{(i)}-\bet\|,\,
\quad t\ne 0\ . \label{enov6}\end{equation} Since at least
$\|\sum_{i=1}^m\bet^{(i)}-\bet\|$ intersection points of $\widetilde
E$ with $C^{(1)}\cup...\cup C^{(m)}$ smooth up when deforming
$Y'_0=\widetilde E\cup C^{(1)}\cup...\cup C^{(m)}$ (the blown-down
$Y_0$) to $Y_t=\hat C_t$, $t\ne 0$, we derive that they are the only
smoothed up intersection points. In particular, each component
$C^{(i)}$ of $Y'_0$ intersects with $\widetilde E$.

It follows from Lemmas \ref{ln1} and \ref{ltn2} that if $\pi_\Sig$
maps $C^{(i)}$ multiply onto its image, or if $\pi_\Sig$ maps
$C^{(i)},C^{(j)}$ onto the same curve, then this image curve must be
a $(-1)$-curve crossing $E$ at one point. We can assume that this
image curve is $E_3$. Let $E_3$ have multiplicity $s\ge2$ in
$C=(\pi_\Sig)_*(C^{(1)}\cup...\cup C^{(m)})$, and let
$C^{(k+1)},...,C^{(m)}$ be all the components of $Y_0$ mapped onto
$E_3$. Since $R_\Sig(D,g,\bet)>0$, we have $D=dL-k_1E_1-...-k_5E_5$,
$d\ge 1$, $k_1,...,k_5\ge 0$. Thus,
$C'=(\pi_\Sig)_*(C^{(1)}\cup...\cup C^{(k)})\in
|(d-1)L-k_1E_1-k_2E_2-(k_3+s-1)E_3-(k_4-1)E_4+k_5E_5|$. So, $C'$
crosses $E_3\backslash E$ with multiplicity $k_3+s-1$, and as
explained above these intersection points persist in the deformation
$Y'_0\to Y_t$, $t\ne0$. Hence, $(\pi_\Sig)_*(Y_t)$ must cross $E_3$
with multiplicity $\ge k_3+s-1>k_3$, which gives a contradiction.

So, the map $\pi_\Sig:C^{(1)}\cup...\cup C^{(m)}\to\Sig$ is
birational on its image. Furthermore, the genericity of $
(C^{(i)},\pi_\Sig, \hat\bp^{(i)})$ in ${\cal
V}_\Sig(D^{(i)},g^{(i)},\alp^{(i)},\bet^{(i)},(\bp^\flat)^{(i)})$
implies that the above image $C$ is a nodal curve, nonsingular along
$E$. \end{proof}

\subsection{Vakil recursive formula}\label{general}
For any variety ${\cal V}_\Sig(D,g,\alp,\bet,\bp^\flat)$, denote by
${\mathfrak V}_\Sig(D,g,\alp,\bet,\bp^\flat)$ the union of the
components of dimension $R_\Sig(D,g,\bet)$ of the natural image of
${\cal V}_\Sig(D,g,\alp,\bet,\bp^\flat)$ in the linear system $|D|$
on $\Sig$. Introduce the numbers
\begin{equation}N_\Sig(D, g, \alp, \bet)=\begin{cases}0,\quad
&\text{if}\,\, {\mathfrak V}_\Sig(D, g, \alp, \bet, \bp^\flat)=\emptyset,\\
\deg{\mathfrak V}_\Sig(D, g, \alp, \bet, \bp^\flat),\quad &
\text{if}\,\, {\mathfrak V}_\Sig(D, g, \alp, \bet,
\bp^\flat)\ne\emptyset .\end{cases} \label{esym6}
\end{equation}
These numbers do not depend on the choice of $\bp^\flat$ and are
enumerative: they count the irreducible nodal curves in $|D|$ which
pass through $R_\Sig(D, g,  \bet)$ generic points in $\Sig\setminus
E$ and belong to ${\mathfrak V}_\Sig(D, g, \alp, \bet, \bp^\flat)$
(cf. \cite[Section 2.4.2]{Va}).

To formulate a recursive formula for the numbers
$N_\Sig(D,g,\alp,\bet)$, we use the following conventions and
notation:
\begin{itemize}
\item the relation $\alp \geq \alp'$ means that
$\alp - \alp' \in\Z^\infty_+$,
\item $I^\alp=\prod_{k\ge 1}k^{\alp_k}$,
\item if
$\alp\ge\alp^{(1)}+...+\alp^{(s)}$, then $$\left(\begin{matrix}\alp \\
\alp^{(1)}...\alp^{(s)}\end{matrix}\right)=\prod_{k\ge
1}\frac{\alp_k!}{\alp^{(1)}_k!...\alp^{(s)}_k!(\alp_k-\alp^{(1)}_k-...
-\alp^{(s)}_k)!}\ .
$$
\end{itemize}
Let us introduce also the semigroup
$$
\displaylines {
A(\Sig,E)=\{(D,g,\alp,\bet)\in\Pic_+(\Sig,E)\times\Z\times
\Z^\infty_+\times\Z^\infty_+\ | \cr D,g,\alp,\bet\ \text{satisfy
(\ref{en1}) and} \; R_\Sig(D, g,  \bet) \geq 0 \} }
$$ with the operation
$$(D^{(1)},g^{(1)},\alp^{(1)},\bet^{(1)})+(D^{(2)},g^{(2)},
\alp^{(2)},\bet^{(2)})$$
$$=(D^{(1)}+D^{(2)},g^{(1)}+g^{(2)}-1,\alp^{(1)}+\alp^{(2)},
\bet^{(1)}+\bet^{(2)})\ .$$ Notice that a quadruple $(D, g, \alp,
\bet)$ in $A(\Sig, E)$ may have negative $g$. Put $N_\Sig(D, g,
\alp, \beta) = 0$ whenever $g < 0$.

\begin{thm}\text{\rm(cf.~\cite{Va})}.\label{tn1}
Consider a divisor $D \in \Pic_+(\Sig,E)$, an integer $g$, and two
elements $\alp,\bet\in\Z^\infty_+$ such that
$$
(D, g, \alp, \beta) \in A(\Sig, E), \quad D \ne E, \quad R_\Sig(D,
g, \bet) > 0.
$$
Then,
$$N_\Sig(D,g,\alp,\bet)=\sum_{j\ge 1,
\ \bet_j>0}jN_\Sig(D,g,\alp+\theta_j,\bet-\theta_j)$$
\begin{equation}+\sum
\left(\begin{matrix}\alp\\
\alp^{(1)}...\alp^{(m)}\end{matrix}\right)
\frac{(n-1)!}{n_1!...n_m!}\prod_{i=1}^m\left(\left(
\begin{matrix}\bet^{(i)}\\
\widetilde\bet^{(i)}\end{matrix}\right) I^{\widetilde\beta^{(i)}}
N_\Sig(D^{(i)},g^{(i)},\alp^{(i)},\bet^{(i)})\right),\label{en5}
\end{equation}
where
$$
n = R_\Sig(D, g,  \bet), \quad n_i = R_\Sig(D^{(i)}, g^{(i)},
\bet^{(i)}) \; for \; any \; i = 1, \ldots, m,
$$
and the second sum in~(\ref{en5}) is taken
\begin{itemize}
\item over all
splittings
\begin{equation}(D-E,g',\alp',\bet')=\sum_{i=1}^m(D^{(i)},g^{(i)},
\alp^{(i)},\bet^{(i)})\ ,\label{en4}
\end{equation}
in $A(\Sig, E)$ of all possible collections $(D-E,g',\alp',\bet')
\in A(\Sig, E)$ such that
\begin{enumerate}
\item[(a)]
$\alp'\le\alp,\quad \bet\le\bet', \quad g - g' = \|\beta' - \beta\|
- 1$,
\item[(b)] each summand
$(D^{(i)},g^{(i)},\alp^{(i)},\bet^{(i)})$ with $n_i=0$ appears
in~(\ref{en4}) at most once,
\end{enumerate}
\item over all splittings
\begin{equation}
\bet'=\bet+\sum_{i=1}^m\widetilde\bet^{(i)},\quad
\|\widetilde\bet^{(i)}\|>0,\ i=1,...,m\ ,\label{en96}
\end{equation}
satisfying the restriction $\bet^{(i)}\ge\widetilde\bet^{(i)},\
i=1,...,m\ $,
\end{itemize}
and factorized by simultaneous permutations in the both splittings
(\ref{en4}) and (\ref{en96}).
\end{thm}

\begin{rem}\label{rn1}
\begin{enumerate}
\item[(1)]
The second sum in the right-hand side of (\ref{en5}) becomes empty
if $D - E$ is not effective. Notice also that, under the hypotheses
of Theorem~\ref{tn1}, one has $(D-E)E>0$. Indeed, with our choice of
$\Sig$ and $E$, the inequalities $DE>0$ and $(D-E)E\le0$ may occur
only in the case $\Sig=\PP^2_2$, $E^2=1$, and $DE=1$, but then $D =
E$ since $D \in \Pic_+(\Sig, E)$.
\item[(2)] The divisors $D^{(i)}$, $i=1,...,m$, in (\ref{en5}) satisfy
$D^{(i)}E\ge I\widetilde\beta^{(i)}>0$.
\end{enumerate}
\end{rem}

{\bf Proof of Theorem 1}. The cases $\Sig=\PP^2_q$, $0\le q\le 4$,
can be reduced to the case of $\Sig=\PP^2_5$ by means of blowing up
$\PP^2_q$ at appropriately chosen $5 - q$ points. Indeed, let $\pi:
\Sig^* \to \Sig$ be the blow-up under consideration, $E^*$ the
strict transform of $E$, and $E_1$, $\ldots$, $E_b$ the exceptional
divisors of~$\pi$ whose images belong to~$E$. According to the
pull-back formula $\pi^* D - E^*=\pi^*(D-E) + \sum_{j = 1}^b E_j$
and due to the fact that the numbers $N_\Sig(D,g,\alp,\bet)$ are
enumerative, the map which sends a decomposition
$$
(D-E,g',\alp',\bet')=\sum_{i=1}^m(D^{(i)},g^{(i)},
\alp^{(i)},\bet^{(i)})
$$
to the decomposition
$$(
\pi^*D-E^*,g',\alp',\bet' + b\theta_1) =
\sum_{i=1}^m(\pi^*D^{(i)},g^{(i)}, \alp^{(i)},\bet^{(i)}) + \sum_{j
= 1}^b(E_j, 0, 0, \theta_1).
$$
gives rise to a 1-to-1 correspondence between the summands in the
right-hand side of formula (\ref{en5}) for $\Sig$ and $\Sig^*$.

The proof of formula (\ref{en5}) for $\Sig=\PP^2_5$ follows the
scheme of \cite{CH,Va}. Lemma \ref{ln1} provides the (expected)
upper bound to the dimension of the considered families of curves
and ensures required properties of generic elements in the families
of expected dimension. The fact that multiple components do not
appear in degenerations follows from Lemmas \ref{ln1} and
\ref{lnov1}. Finally, the condition that each summand
$(D^{(i)},g^{(i)},\alp^{(i)},\bet^{(i)})$ with $n_i=0$ may appear in
(\ref{en4}) at most once follows from Lemmas \ref{ltn2} and
\ref{lnov1} and from the fact that the second sum in the right-hand
side of (\ref{en4}) corresponds to the degenerations described in
Lemma \ref{lnov1} (cf., \cite[Section 3]{CH} and \cite[Section
5]{Va}). \proofend

\subsection{Initial conditions}\label{initial}

\begin{thm}\label{tn2}
All the numbers $N_\Sig(D,g,\alp,\bet)$ with $(D,g,\alp,\bet)\in
A(\Sig,E)$ are determined recursively by formula (\ref{en5}) and the
following list of initial values.

(1) In the cases $\Sig=\PP^2_q$, $q\le 2$, one has
\begin{enumerate}
\item[(i)] $N_\Sig(L, 0, \alp, \bet)=1$ as long
as $I\alp+I\bet=1$,
\item[(ii)]
if $1 \leq q \leq 2$, then $N_\Sig(L-E_i,0,\theta_1,0)=1$ for each
$1\le i\le q$,
\item[(iii)] if $q = 2$, then $N_\Sig(L-E_1-E_2, 0, 0, \theta_1) = 1$.
\end{enumerate}

(2) In the case $\Sig=\PP^2_3$ one has
\begin{enumerate}
\item[(i)]
$N_\Sig(E_3,0,0,\theta_1)=1$,
\item[(ii)] $N_\Sig(L-E_i,0,\theta_1,0)=1$ for $i=1,2$,
\item[(iii)]
$N_\Sig(L-E_1-E_2,0,0,\theta_1)=1$.
\end{enumerate}

(3) In the case $\Sig=\PP^2_4$ one has
\begin{enumerate}
\item[(i)] $N_\Sig(E_i,0,0,\theta_1)=1$ for $i=3,4$,
\item[(ii)]
$N_\Sig(L-E_i,0,\theta_1,0)=1$ for $i=1,2$,
\item[(iii)] $N_\Sig(L-E_1-E_2,0,0,\theta_1)=1$.
\end{enumerate}

(4) In the case $\Sig=\PP^2_5$ one has
\begin{enumerate}
\item[(i)] $N_\Sig(E_i,0,0,\theta_1)=1$ for $i=3,4$,
\item[(ii)] $N_\Sig(L-E_i-E_j,0,0,\theta_1)=1$
for $i, j = 1, 2, 5$, $i \ne j$,
\item[(iii)] $N_\Sig(L-E_i,0,\theta_1,0)=1$ for $i=1,2,5$,
\item[(iv)] $N_\Sig(2L-E_1-E_2-E_i-E_5,0,\theta_1,0)=1$ for $i=3,4$,
\item[(v)] $N_\Sig(2L-E_1-E_2-E_5,0,\alp,0)=1$ as long as $I\alp=2$.
\end{enumerate}

(5) $N_\Sig(D,g,\alp,\bet)=0$ for all other tuples
$(D,g,\alp,\bet)\in A(\Sig,E)$ such that either $D=E$, or $R_\Sig(D,
g,  \bet) \leq 0$.
\end{thm}

\begin{proof} Straightforward from Lemmas \ref{ln1} and \ref{ltn2}.
\end{proof}

\subsection{Modified recursive formula}\label{modified}
For further purposes, we switch the ground field~$\C$ to the
(algebraically closed) field of complex locally convergent Puiseux
series $\K=\bigcup_{m\ge0}\C\{t^{1/m}\}$; this does not affect the
enumerative invariants under consideration. In addition, we rewrite
Vakil's recursive formula~(\ref{en5}) in a slightly different way.
We specialize the formula to the case $\Sig=\PP^2_5$\ and
$E=L-E_3-E_4$; the other cases can be reduced to this one in the
same way as in the proof of Theorem~\ref{tn1}.

Define the sub-semigroup $A^{tr}(\Sig,E)\subset A(\Sig,E)$ by
$$A^{tr}(\Sig,E)=\{(D,g,\alp,\bet)\in A(\Sig,E)\ |\
DE_i\ge0,\ i=1,...,5\}\ .$$ Notice that the condition $DE_i\ge 0$,
$i=1,...,5$, in the definition of $A^{tr}(\Sig,E)$ means that $E_3$
and $E_4$ are excluded from the semigroup generators.

\begin{proposition}\label{tn3}
(1) If $(D,g,\alp,\bet)\in A^{tr}(\Sig,E)$ and $R_\Sig(D, g, \bet) >
0$, then
$$N_\Sig(D,g,\alp,\bet)=\sum_{j\ge 1,\ \bet_j>0}jN_\Sig(D,g,\alp+\theta_j,
\bet-\theta_j)$$
\begin{equation}
+ \sum\left(\begin{matrix}\alp\\
\alp^{(1)}...\alp^{(m)}\end{matrix}\right)
\frac{(n-1)!}{n_1!...n_m!}\prod_{i=1}^m\left(\left(
\begin{matrix}\bet^{(i)}\\
\widetilde\bet^{(i)}\end{matrix}\right)I^{\widetilde\bet^{(i)}}
N_\Sig(D^{(i)},g^{(i)},\alp^{(i)} ,\bet^{(i)})\right),\label{en8}
\end{equation} where
$$
n = R_\Sig(D, g,  \bet), \quad n_i = R_\Sig(D^{(i)}, g^{(i)},
 \bet^{(i)}) \; for \; any \; i = 1, \ldots, m,
$$
and the second sum in~(\ref{en8}) is taken
\begin{itemize}
\item over all elements~$\daleth \in \{0, E_3, E_4, E_3 + E_4\}$
such that $D - E - \daleth \in \Pic_+(\Sig,E)$,
\item over all
splittings
\begin{equation}(D - E - \daleth, g', \alp', \bet') =
\sum_{i=1}^m(D^{(i)},g^{(i)}, \alp^{(i)},\bet^{(i)})\
,\label{en9}\end{equation} of all possible collections $(D - E -
\daleth, g', \alp', \bet') \in A^{tr}(\Sig,E)$ such that
\begin{enumerate}
\item[(a)]
$\alp'\le\alp,\quad \bet\le\bet', \quad g - g' = \|\beta' - \beta\|
- 1$,
\item[(b)] each summand
$(D^{(i)},g^{(i)},\alp^{(i)},\bet^{(i)})$ with $n_i=0$ appears in
(\ref{en9}) at most once,
\end{enumerate}
\item over all splittings
\begin{equation}
\bet'=\bet+\sum_{i=1}^m\widetilde\bet^{(i)},\quad
\|\widetilde\bet^{(i)}\|>0,\ i=1,...,m\ ,\label{en98}
\end{equation}
satisfying the restriction $\bet^{(i)}\ge\widetilde\bet^{(i)},\
i=1,...,m\ $,
\end{itemize}
and factorized by simultaneous permutations in the both splittings
(\ref{en9}) and (\ref{en98}).

(2)  Formula (\ref{en8}) recursively determines all the numbers
$N_\Sig(D,g,\alp,\bet)$, $(D,g,\alp,\bet)\in A^{tr}(\Sig,E)$, from
the data listed in Theorem \ref{tn2} {\rm (}4{\rm (}ii{\rm )}- 4{\rm
(}v{\rm )},5{\rm )}.
\end{proposition}

{\bf Proof.} Due to condition~(b) in Theorem \ref{tn1}, each
splitting~(\ref{en4}) contains at most one summand with $D^{(i)} =
E_3$ and at most one summand with $D^{(i)} = E_4$. The second sum in
the right-hand side of formula (\ref{en8}) is obtained by
subdividing the second sum of the right-hand side of (\ref{en5}) in
four sums according to the presence of summands with $D^{(i)} = E_3$
and $D^{(i)} = E_4$ in~(\ref{en4}).

The list of initial conditions is obtained from the list given in
Theorem~\ref{tn2}(4, 5) by removing the cases $D = E_3$ and $D =
E_4$. \proofend

\section{Tropical Caporaso-Harris type formulas}\label{tropical-formulas}

\subsection{$\CH$-configurations}\label{CH-config}

Let~$l$ and~$r$ be non-negative integers such that $l \leq r$.
Introduce the space $\qquad\qquad$ \mbox{${\cal P}(l, r) \subset
(L_{-\infty})^l\times (\R^2)^{r + 1 - l}$} formed by the (ordered)
{\it configurations} \mbox{$\bx=(\bx^\flat,\bx^\sharp, \bx_{r +
1})$} of $r + 1$ points in $\widehat\R^2$ such that
\begin{itemize}
\item $\bx^\flat = (p_1, \ldots, p_l)$
is a sequence of $l$ points on $L_{-\infty}$,
\item $\bx^\sharp = (p_{l + 1}, \ldots, p_r)$ is
a sequence of $r - l$ points in $\R^2$,
\item $\bx_{r + 1}$ is a point in $\R^2$,
\item for any indices $i$ and $j$
such that $l + 1 \leq i < j \leq r + 1$, the first coordinate of
$p_i$ is less than the first coordinate of $p_j$,
\item for
any index $i$ such that $1 \leq i \leq r$, the second coordinate of
$p_{r + 1}$ is less than the second coordinate of $p_i$.
\end{itemize}
Consider a finite multi-set~$\Del$ of vectors in $\Z^2$, two
non-negative integers $k_1$ and $k_2$, an integer~$g$, and two
elements $\alp$ and $\bet$ in $\Z^{\infty}_+$. We say that the
collection $(\Del, k_1, k_2, g, \alp, \bet)$ is {\it $(l,
r)$-admissible} if
\begin{itemize}
\item $\|\alp\| = l$ and $|\Del| - I\alp - I\beta
+ \|\alp\| + \|\beta\| + g - 1 - k_1 - k_2 = r$, where $|\Del|$ is
the number of vectors in the multi-set~$\Del$,
\item the sum of the vectors in~$\Del$
is equal to~$0$, and each vector in~$\Del$ belongs to the list $(0,
1)$, $(0, -1)$, $(1, 1)$, $(-1, -1)$, $(1, 0)$, $(-1, 0)$,
\item the number of vectors $(-1, 0)$ in~$\Del$
is non-zero and equal to $I\alp + I\beta$,
\item $k_3 + k_4 < d$, where $d$ is the number of
vectors in~$\Del$ which have non-negative coordinates, $k_3$ is the
number of vectors $(0, 1)$ in~$\Del$, and $k_4$ is the number of
vectors $(-1, -1)$ in~$\Del$.
\item $g \leq \frac{(d - 1)(d - 2)}{2} -
\sum_{i = 1}^5\frac{k_i(k_i - 1)}{2}$, where $k_5$ is the number of
vectors $(1, 0)$ in~$\Del$.
\end{itemize}

Since the sum of the vectors in~$\Del$ is equal to~$0$, there exists
a convex lattice polygon $\Pi(\Del)$, possibly reduced to a vertical
segment, such that
\begin{itemize}
\item each vector in $\Del$ is an outgoing normal vector
of a certain side of $\Pi(\Del)$,
\item for each side~$\sig$ of $\Pi(\Del)$,
the integer length of~$\sig$ ({\it i.e.}, its number of integer
points diminished by~$1$) is equal to the multiplicity of the
outgoing normal vector of~$\sig$ in $\Del$.
\end{itemize}
Such a polygon $\Pi(\Del)$ is unique up to translation by a vector
with integer coordinates.

The geometric meaning of the number~$d$ appearing in the definition
of an $(l, r)$-admissible collection is as follows: $d$ is the
smallest positive integer such that $\Pi(\Del)$ can be shifted into
the triangle with vertices $(0, 0)$, $(d, 0)$, $(0, d)$.

The multi-set of vectors $B - A$, where $(A, B)$ runs over all
couples of points of $\Pi(\Del)$ which have integer coordinates, is
denoted by $\widetilde\Del$. We say that a finite multi-set $\Theta$
of vectors in $\Z^2$ {\it dominates} $\Del$, if $\widetilde\Del
\subset \Theta$ (since $\widetilde\Del$ is finite, there always
exists a finite multi-set~$\Theta$ which dominates~$\Del$).

\begin{lem}\label{finite}
For any integers $0 \leq l \leq r$ and any finite multi-set $\Theta$
of vectors in $\Z^2$, the set of $(l, r)$-admissible collections
$(\Del, k_1, k_2, g, \alp, \beta)$ such that $\Theta$ dominates
$\Del$ is finite.
\end{lem}

\begin{proof} Straightforward. \end{proof}

For any $\bx \in {\cal P}(l, r)$ and any $(l, r)$-admissible
collection $(\Del, k_1, k_2, g, \alp, \beta)$, the points of
$\bx^\flat$ are naturally divided into groups: the first group
consists of the first $\alp_1$ points of $\bx^\flat$, the second
group consists of the next $\alp_2$ points of $\bx^\flat$, and so
on.

For any $\bx \in {\cal P}(l, r)$ and any $(l, r)$-admissible
collection $(\Del, k_1, k_2, g, \alp, \beta)$, introduce the set
${\cal T}(\Del, k_1, k_2, g, \alp, \bet, \bx)$ (respectively, ${\cal
T}^{\text{\rm sym}}(\Del, k_1, k_2, g, \alp, \bet, \bx)$) of
isomorphism classes of irreducible marked pseudo-simple
$\CL$-curves~$Q = (\overline\Gam, {\cal V}, h, \bpp)$ (respectively,
irreducible symmetric marked pseudo-simple ${\cal L}$-curves $Q =
(\overline\Gam, {\cal V}, h, \bpp,\xi)$) satisfying the following
conditions:
\begin{itemize}
\item $Q$
is of genus~$g$ and degree $\Del^\circ$, where $\Del^\circ$ is
obtained from $\Del$ by replacing $i(\alp_i + \beta_i)$ vectors
$(-1, 0)$ with $\alp_i + \beta_i$ vectors $(-i, 0)$ for each
positive integer~$i$;
\item ${\widehat h}(\bpp^\flat) = \bx^\flat$, $h(\bpp^\sharp) = \bx^\sharp$;
\item if $\bpp^1 \cup \bpp^2 \cup \bpp^\nu\ne\emptyset$,
then $h(\bpp^1 \cup \bpp^2 \cup \bpp^\nu) = p_{r + 1}$;
\item the number of points in $\bpp^\aleph \cup \bpp^\nu$
is equal to $k_\aleph$, $\aleph = 1, 2$,
\item any point
$p_m \in \bx^\flat$ is contained in the image of a left end of $Q$
of weight $\mt(\alp, m)$.
\end{itemize}

\begin{rem}\label{rem-sym}
(1) The set ${\cal T}^{\text{\rm sym}}(\Del, k_1, k_2, g, \alp,
\bet, \bx)$ is nonempty only if $k_1 = k_2$.

(2) If $l>0$ or $r>l$, then any symmetric marked ${\cal L}$-curve
$(\overline\Gam, {\cal V}, h, \bpp, \xi)$ whose isomorphism class
belongs to ${\cal T}^{\text{\rm sym}}(\Del, k_1, k_2, g, \alp, \bet,
\bx)$ has a connected graph $\overline\Gam$, {\it i.e.}, the marked
${\cal L}$-curve $(\overline\Gam, {\cal V}, h, \bpp)$ is
irreducible.

(3) Assume that $l = r = 0$ and the isomorphism class of an
irreducible symmetric marked $\CL$-curve $(\overline\Gam, \CV, h,
\bpp, \xi)$ with disconnected graph $\overline\Gam$ belongs to
$\qquad\qquad\qquad$ ${\cal T}^{\text{\rm sym}}(\Del, k_1, k_2, g,
\alp, \bet, \bx)$. Then, $\bx=\{p_1\}\subset\R^2$, all the points of
$\bpp$ are mapped to $p_1$, and as it follows from the last
condition in the definition of marked $\CL$-curves and the last
condition in the definition of $(l,r)$-admissible collections,
\begin{itemize}
\item
$\Del=\{(-1,0),(-1,0),(1,0),(1,0)\}$, $\alp=0$, $\bet=(2)$, $g=-1$,
$k_1=k_2=1$,
\item $\Gam$
consists of two edges interchanged by~$\xi$ and mapped by~$h$ onto
the horizontal straight line passing through $p_1$ (see Figure
\ref{lines}).
\end{itemize}
\end{rem}

\begin{figure}
\setlength{\unitlength}{1cm}
\begin{picture}(11,2)(-3,0)
\thinlines\put(0,1){\line(1,0){4}}\put(3,0.87){$\bullet$}
\put(5,0.9){$2L-E_1-E_2-2E_5$}\put(2.9,0.5){$p_1$}
\put(1.4,1.1){$2$}\put(3.5,1.1){$2$}
\end{picture}
\caption{Exceptional class $T^{sp}(p_1)$}\label{lines}
\end{figure}
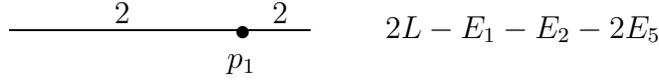

The collection $(\{(-1, 0), (-1, 0), (1, 0), (1, 0)\}, 1, 1, -1,
(0), (2))$ is denoted by $K^{sp}$, and the isomorphism class
described is denoted by $T^{sp}(p_1)$. If $(\Del, k_1, k_2, g, \alp,
\beta) \ne K^{sp}$, there is a well-defined forgetful map
\begin{equation}
\Psi:{\cal T}^{\text{\rm sym}} (\Del, k_1, k_2, g, \alp, \bet, \bx)
\to {\cal T}(\Del, k_1, k_2, g, \alp, \bet, \bx)\ ,\label{esym1}
\end{equation}
$$\Psi[(\overline\Gam,{\cal V},h,\bpp,\xi)]=
[(\overline\Gam,{\cal V},h,\bpp)]\ .$$

Let $\Theta$ be a finite multi-set of vectors in $\Z^2$. A
configuration $\bx \in {\cal P}(l, r)$ is called a {\it weak
$\CH_\Theta$-configuration of type $(l, r)$} if there exist a
positive real number~$\varepsilon$ and real numbers $\delta_l <
\ldots < \delta_{r + 1}$ such that
\begin{itemize}
\item for each integer $i = l + 1$, $\ldots$, $r + 1$,
the point $p_i$ belongs to the rectangle $(\delta_{i-1}, \delta_i)
\times (-\eps,\eps)$,
\item for any $(l, r)$-admissible collection
$(\Del, k_1, k_2, g, \alp, \beta)$ having $g \geq 0$ and dominated
by $\Theta$, and for each marked $\CL$-curve $(\overline\Gam, {\cal
V}, h, \bpp)$ whose isomorphism class belongs to ${\cal T}(\Del,
k_1, k_2, g, \alp, \bet, \bx)$, the following properties hold:
\begin{enumerate}
\item[(1)]
the image under~$h$ of any vertex of $\overline\Gam$ has the first
coordinate different from $\delta_i$, $i = l$, $\ldots$, $r + 1$.
\item[(2)]
any edge~$E$ of $\overline\Gam$ such that the image $h(E)$ of~$E$
intersects one of the vertical segments $I_i$ with endpoints
$(\delta_i, -\varepsilon)$ and $(\delta_i, \varepsilon)$, $i = l$,
$\ldots$, $r + 1$, is horizontal,
\item[(3)]
for each integer $i=l,...,r$ and each irreducible component
$\qquad\qquad\qquad$ $(\overline\Gam_{{\cal X}_i, \CX'_{i+1}, j},
{\cal V}_{{\cal X}_i, \CX'_{i+1}, j}, h_{{\cal X}_i, \CX'_{i+1}, j},
\bpp_{{\cal X}_i, \CX'_{i+1}, j})$ of a marked cut of
$(\overline\Gam_{\CX_i}, {\cal V}_{\CX_i}, h_{\CX_i}, \bpp_{\CX_i})$
at ${\cal X}'_{i+1}$ (where ${\cal X}_i$ is the inverse image
under~$h$ of the segment $I_i$, the curve $(\overline\Gam_{\CX_i},
\CV_{\CX_i}, h_{\CX_i}, \bpp_{\CX_i})$ is a marked cut of
$(\overline\Gam, \CV, h, \bpp)$ at $\CX_i$, and $\CX'_{i+1} =
h^{-1}_{\CX_i}(I_{i+1})$; notice that $\CX_i$ and $\CX'_{i+1}$ are
sparse), non-emptiness of
$$
h_{{\cal X}_i, \CX'_{i+1}, j}(\Gam_{{\cal X}_i, \CX'_{i+1}, j} \cap
\Gam) \cap \{(x, y) \in \R^2 \ | \ \delta_i \leq x \leq \delta_{i +
1}, \; -\varepsilon \leq y \leq \varepsilon\}
$$
implies that $(\overline\Gam_{{\cal X}_i, \CX'_{i+1}, j}, {\cal
V}_{{\cal X}_i, \CX'_{i+1}, j}, h_{{\cal X}, \CX'_{i+1}, j})$ is
either horizontal or one-sheeted (see Figure~\ref{new-picture5}),
\item[(4)]
for each irreducible component $(\overline\Gam_{{\cal X}_l, j},
{\cal V}_{{\cal X}_l, j}, h_{{\cal X}_l, j}, \bpp_{{\cal X}_l, j})$
of a marked cut of $(\overline\Gam, {\cal V}, h, \bpp)$ at ${\cal
X}_l$, the non-emptiness of
$$h_{{\cal X}_l, j}(\Gam_{{\cal X}_l, j} \cap \Gam)
\cap \{(x, y) \in \R^2 \ | \ x \leq \delta_l, \; -\varepsilon \leq y
\leq \varepsilon\}$$ implies that $(\overline\Gam_{{\cal X}_l, j},
{\cal V}_{{\cal X}_l, j}, h_{{\cal X}_l, j})$ is horizontal.
\end{enumerate}
\end{itemize}

Let $\bx \in {\cal P}(l, r)$ be a weak $\CH_\Theta$-configuration.
Numbers $\varepsilon$, $\delta_l$, $\ldots$, $\delta_{r + 1}$
certifying that $\bx$ is a weak $\CH_\Theta$-configuration are said
to be {\it parameters} of~$\bx$ (of course, these parameters are far
from being unique). The parameter~$\varepsilon$ is called {\it
vertical}, and the parameters $\delta_l$, $\ldots$, $\delta_{r + 1}$
are called {\it horizontal}. Pick a marked $\CL$-curve
$(\overline\Gam, {\cal V}, h, \bpp)$ whose isomorphism class belongs
to ${\cal T}(\Del, k_1, k_2, g, \alp, \bet, \bx)$. For any integer
$i=l,...,r$, put $\CX_i = h^{-1}(I_i)$, where $I_i$ is defined
in~(2) above, and consider a marked cut $(\overline\Gam_{{\cal
X}_i}, {\cal V}_{{\cal X}_i}, h_{{\cal X}_i}, \bpp_{{\cal X}_i})$ of
$(\overline\Gam, \CV, h, \bpp)$ at~${\cal X}_i$. An irreducible
component $(\overline\Gam_{{\cal X}_i, j}, {\cal V}_{{\cal X}_i, j},
h_{{\cal X}_i, j}, \bpp_{{\cal X}_i, j})$ of $(\overline\Gam_{{\cal
X}_i}, {\cal V}_{{\cal X}_i}, h_{{\cal X}_i}, \bpp_{{\cal X}_i})$ is
called {\it left} (respectively, {\it right}) if
$$h_{{\cal X}_i, j}(\Gam_{{\cal X}_i, j} \cap \Gam)
\cap \{(x, y) \in \R^2 \ | \ x \leq \delta_i, \; -\varepsilon \leq y
\leq \varepsilon\} \ne \varnothing$$ (respectively, $h_{{\cal X}_i,
j}(\Gam_{{\cal X}_i, j} \cap \Gam) \cap \{(x, y) \in \R^2 \ | \ x
\geq \delta_i, \; -\varepsilon \leq y \leq \varepsilon\} \ne
\varnothing$).

\begin{figure}
\setlength{\unitlength}{0.9cm}
\begin{picture}(15,7)(-1,0)
\thicklines
\put(0.5,3){\line(1,0){3}}\put(0.5,4){\line(1,0){5.5}}\put(0.5,5.5){\line(1,0){3}}
\put(4,2.5){\line(1,0){2}}\put(6.5,2){\line(1,0){2.5}}\put(3.5,3){\line(1,-1){0.5}}
\put(6,2.5){\line(1,-1){0.5}}
\put(3.5,3){\line(0,1){2.5}}\put(4,0.5){\line(0,1){2}}\put(6,2.5){\line(0,1){1.5}}
\put(6.5,0.5){\line(0,1){1.5}}\put(9,0.5){\line(0,1){1.5}}
\put(3.5,5.5){\line(1,1){1}}\put(6,4){\line(1,1){2.5}}\put(9,2){\line(1,1){4.5}}
\thinlines
\put(0.5,0.5){\line(0,1){6.5}}\dottedline{0.1}(0.5,1)(14.5,1)\dottedline{0.1}(0.5,6)(14.5,6)
\dashline{0.2}(1,1)(1,6)\dashline{0.2}(2.2,1)(2.2,6)\dashline{0.2}(3,1)(3,6)
\dashline{0.2}(4.5,1)(4.5,6)\dashline{0.2}(5.4,1)(5.4,6)\dashline{0.2}(8.5,1)(8.5,6)
\dashline{0.2}(13.5,1)(13.5,6)
\put(0.4,5.35){$\bullet$}\put(1.5,3.85){$\bullet$}\put(2.5,2.85){$\bullet$}
\put(3.65,2.65){$\bullet$}\put(4.8,2.35){$\bullet$}\put(5.93,3.2){$\bullet$}
\put(8.9,1.85){$\bullet$}
\put(0.2,0.9){$\eps$}\put(-0.1,5.9){$-\eps$}\put(0.6,6.6){$L_{-\infty}$}
\put(0.9,0.5){$I_l$}\put(1.3,4.25){$\bx_{l+1}$}\put(6.2,3.2){$\bx_r$}\put(8.4,0.5){$I_r$}\put(9.2,1.8){$\bx_{r+1}$}
\put(13.3,0.5){$I_{r+1}$}
\end{picture}
\caption{Cutting, VI}\label{new-picture5}
\end{figure}

For any configuration $\bx \in {\cal P}(l, r)$ and any integer
number $u=l+1,...,r$, the configuration, which belongs to ${\cal
P}(u, r)$ and is obtained from~$\bx$ by removing the points $p_{l +
1}$, $\ldots$, $p_u$ from $\bx^\sharp$ and inserting their
horizontal projections to $L_{-\infty}$ into $\bx^\flat$ at
arbitrary places, is called the {\it $u$-projection} of~$\bx$.

A configuration $\bx \in {\cal P}(l, r)$ is called a {\it
$\CH_\Theta$-configuration of type $(l, r)$} if there exist a
positive real number~$\varepsilon$ and real numbers $\delta_l$,
$\ldots$, $\delta_{r + 1}$ such that
\begin{itemize}
\item for any non-negative
integers~$l'$ and~$r'$ satisfying the inequalities $l' \leq l$ and
$l' \leq r' \leq r$, any subconfiguration $\bx' \subset \bx$ such
that $\bx' \in {\cal P}(l', r')$ is a weak
$\CH_\Theta$-configuration of type $(l', r')$ and has parameters
$\varepsilon$, $\delta_{s(l')}$, $\ldots$, $\delta_{s(r')}$,
$\delta_{r + 1}$, where $(p_{s(l')}, \ldots,
p_{s(r')})=(\bx')^\sharp$;
\item for any integer number $u=l+1,...,r$,
any $u$-projection of~$\bx$ is a weak $\CH_\Theta$-configuration of
type $(u, r)$ having $\varepsilon$, $\delta_u$, $\ldots$,
$\delta_r$, $\delta_{r + 1}$ as parameters.
\end{itemize}

\begin{proposition}\label{CH-existence}
Let~$l$ and~$r$ be non-negative integer numbers such that $l \leq
r$, and $\Theta$ a finite multi-set of vectors in $\Z^2$. Then, the
set of $\CH_\Theta$-configurations in ${\cal P}(l, r)$ contains a
non-empty subset which is open in ${\cal P}(l, r)$.
\end{proposition}

The proof of Proposition~\ref{CH-existence} is based on the
following lemmas.

\begin{lem}\label{splitting}\text{\rm(cf.~\cite[proof of Theorem
4.3]{GM2} and~\cite[Lemma~28]{IKS3})}. Let~$l$ and~$r$ be
non-negative integer numbers such that $l \leq r$, and let $(\Del,
k_1, k_2, g, \alp, \beta)$ be an $(l, r)$-admissible collection such
that $g \geq 0$. Fix a positive real number~$\varepsilon$ and two
real numbers $N_1$ and $N_2$ such that $N_1 < N_2$. Consider a
configuration $\bx \in {\cal P}(l, r)$ such that
\begin{itemize}
\item the second coordinates of all points in~$\bx$
belong to the interval $(-\varepsilon, \varepsilon)$,
\item no first coordinate of a point in $\bx^\sharp \cup \bx_{r + 1}$
belongs to the interval $[N_1, N_2]$.
\end{itemize}
Then, for each marked $\CL$-curve $(\overline\Gam, {\cal V}, h,
\bpp)$ whose isomorphism class belongs to $\qquad$ ${\cal T}(\Del,
k_1, k_2, g, \alp, \bet, \bx)$, the second coordinate of the image
under~$h$ of any vertex of $\Gam$ belongs to the interval
$(-\varepsilon, \varepsilon)$. Furthermore, if the length of the
interval $[N_1, N_2]$ is sufficiently large with respect
to~$\varepsilon$, then there exist real numbers~$a$ and~$b$ such
that $N_1 < a < b < N_2$ and satisfying the following condition: for
each marked $\CL$-curve $(\overline\Gam, {\cal V}, h, \bpp)$ whose
isomorphism class belongs to ${\cal T}(\Del, k_1, k_2, g, \alp,
\bet, \bx)$, the intersection of~$h(\Gam)$ with the rectangle $\{(x,
y) \in \R^2: \; a \leq x \leq b \; \text{\rm and} \; -\varepsilon
\leq y \leq \varepsilon\}$ consists of horizontal segments.
\end{lem}

\begin{proof} The proof is completely similar to the proof of Lemma~28
in~\cite{IKS3}, but since our present setting is slightly different
from the one in~\cite{IKS3}, we repeat the proof here.

Consider a marked pseudo-simple $\CL$-curve $(\overline\Gam, {\cal
V}, h, \bpp)$ whose isomorphism class belongs to ${\cal T}(\Del,
k_1, k_2, g, \alp, \bet, \bx)$. Among the non-univalent vertices
of~$\overline\Gam$, choose a vertex~$V$ whose image $h(V) = (v_1,
v_2)$ has the maximal second coordinate. The curve $\qquad$
$(\overline\Gam, {\cal V}, h, \bpp)$ has an end~$E$ such that~$E$ is
incident to~$V$ and the second coordinate of the vector $u_V(E)$ is
positive. This end is of weight~$1$ and $u_V(E)$ is equal either to
$(0, 1)$, or to $(1, 1)$. Hence, $(\overline\Gam, {\cal V}, h,
\bpp)$ should have another edge~$E'$ such that~$E'$ is incident
to~$V$ and the second coordinate of the vector $u_V(E')$ is
non-negative. If $v_2 > \varepsilon$, the connected component of
$\overline\Gam \setminus \bpp$ containing $V$ has at least two
non-rigid ends, which is impossible by the definition of marked
${\cal L}$-curves. In the same way one shows that $\overline\Gam$
has no non-univalent vertex whose image under~$h$ is below the line
$y = -\varepsilon$. This proves the first statement of the lemma.

Denote by~$R$ the rectangle $\{(x, y) \in \R^2: \; N_1 \leq x \leq
N_2 \; \text{\rm and} \; -\varepsilon \leq y \leq \varepsilon\}$. It
follows from the first statement of the lemma, that the image
under~$h$ of any path~$\gamma \subset \Gamma \setminus \bpp$ does
not intersect at least one of the two horizontal edges of~$R$.
Let~$z \in \Gam$ be a point such that $h(z) = (x_1, y_1)$ belongs to
the interior of~$R$, and~$z$ belongs to a non-horizontal edge of
$\overline\Gam$. Then, there exists a path $\gamma \subset \Gam
\setminus \bpp$ having~$z$ as an extreme point and such that
$h(\gamma)$ is the graph of a strictly monotone function~$f$ defined
either on the interval $[N_1, x_1]$, or on the interval $[x_1,
N_2]$. Since there are only finitely many slopes that can be
realized by the images of edges of a parameterized plane tropical
curve of degree~$\Del^\circ$, the length of the definition interval
of~$f$ is bounded from above by a constant depending only on~$\Del$
and~$\varepsilon$. This proves the second statement of the lemma.
\end{proof}

\begin{lem}\label{floors}
Let~$l$ and~$r$ be non-negative integer numbers such that $l \leq
r$, and let $(\Del, k_1, k_2, g, \alp, \beta)$ be an $(l,
r)$-admissible collection such that $g \geq 0$. Fix a positive real
number~$\varepsilon$ and two real numbers $M_1$ and $M_2$ such that
$M_1 < M_2$. Consider a configuration $\bx \in {\cal P}(l, r)$ such
that
\begin{itemize}
\item the second coordinates of all points in~$\bx$
belong to the interval $(-\varepsilon, \varepsilon)$,
\item there exists a point in $\bx^\sharp$
such that the first coordinate of this point belongs to the interval
$(M_1, M_2)$, and no other point in $\bx^\sharp \cup \bx_{r + 1}$
has the first coordinate in the interval $[M_1, M_2]$.
\end{itemize}
Pick an irreducible marked pseudo-simple $\CL$-curve
$(\overline\Gam, {\cal V}, h, \bpp)$ satisfying the following
properties: its isomorphism class belongs to ${\cal T}(\Del, k_1,
k_2, g, \alp, \bet, \bx)$, and each edge~$E$ of $\overline\Gam$ such
that the image $h(E)$ of~$E$ intersects one of the two vertical
segments $J_i$ with endpoints $(M_i, -\varepsilon)$ and $(M_i,
\varepsilon)$, $i = 1, 2$, is horizontal. Consider a marked cut
$(\overline\Gam_{\cal X}, {\cal V}_{\cal X}, h_{\cal X}, \bpp_{\cal
X})$ of $(\overline\Gam, {\cal V}, h, \bpp)$ at $\CX = h^{-1}(J_1)$.
Then, any irreducible component $\qquad\qquad\qquad\qquad$
$(\overline\Gam_{{\cal X}, \CX', j}, {\cal V}_{{\cal X}, \CX', j},
h_{{\cal X}, \CX', j}, \bpp_{{\cal X}, \CX, j})$ of a marked cut of
$(\overline\Gam_{\cal X}, {\cal V}_{\cal X}, h_{\cal X}, \bpp_{\cal
X})$ at ${\cal X}' = h^{-1}_\CX(J_2)$ such that
$$h_{{\cal X}, \CX', j}(\Gam_{{\cal X}, \CX', j} \cap \Gam)
\cap \{(x, y) \in \R^2 \ : \ M_1 < x < M_2, \; -\varepsilon \leq y
\leq \varepsilon\} \ne \varnothing,$$ is either horizontal or
one-sheeted.
\end{lem}

\begin{proof} Assume that an irreducible component
$(\overline\Gam_{{\cal X}, {\cal X}', j}, {\cal V}_{{\cal X}, {\cal
X}', j}, h_{{\cal X}, {\cal X}', j}, \bpp_{{\cal X}, {\cal X}', j})$
is neither horizontal, nor one-sheeted. Then, $\overline\Gam_{{\cal
X}, {\cal X}', j}$ has at least four non-horizontal ends $E_1$,
$E_2$, $E'_1$, and $E'_2$. Lemma~\ref{splitting} implies that the
images under~$h_{{\cal X}, {\cal X}', j}$ of all non-univalent
vertices of $\overline\Gam_{{\cal X}, {\cal X}', j}$ belong to the
rectangle
$$\{(x, y) \in \R^2 \ : \ M_1 < x < M_2,\ -\varepsilon < y < \varepsilon\},$$
and at least three of the ends $E_1$, $E_2$, $E'_1$, and $E'_2$ are
non-rigid. This contradicts the fact that $\overline\Gam_{{\cal X},
{\cal X}', j} \setminus \bpp_{{\cal X}, {\cal X}', j}$ has at most
two connected components whose images under~$h_{{\cal X}, {\cal X}',
j}$ intersect the strip $\{(x, y) \in \R^2 \ : \ M_1 < x < M_2\}$.
\end{proof}

{\bf Proof of Proposition~\ref{CH-existence}}. The statement follows
from Lemmas~\ref{splitting},~\ref{floors}, and the fact that the
number of $(l', r')$-admissible collections $(\Del, k_1, k_2, g,
\alp, \beta)$ such that $l' \leq r' \leq r$, and~$\Theta$
dominates~$\Del$, is finite. \proofend

Let~$l$ and~$r$ be non-negative integer numbers such that $l \leq
r$, and $\Theta$ a finite multi-set of vectors in $\Z^2$. A
configuration $\bx \in {\cal P}(l, r)$ is called {\it
$\Theta$-generic}, if the following condition is satisfied: for any
end-marked rational marked $\CL$-curve $(\overline\Gam, \CV, h,
\bpp)$ such that its degree is contained in~$\Theta$ and $h(\bpp)
\subset \bx$, the image under~$h$ of the non-rigid end of
$(\overline\Gam, \CV, h, \bpp)$ does not contain any point of~$\bx$.
Let ${\cal P}^{\text{\rm gen}}_\Theta(l, r) \subset {\cal P}(l, r)$
be the subset formed by the $\Theta$-generic configurations.

The following lemma is an immediate consequence of ~\cite[Lemma
2]{Sh09}.

\begin{lem}\label{genericity}
Let~$l$ and~$r$ be non-negative integer numbers such that $l \leq
r$, and $\Theta$ a finite multi-set of vectors in $\Z^2$. Then, the
subset ${\cal P}^{\text{\rm gen}}_\Theta(l, r)$ is dense in ${\cal
P}(l, r)$.\proofend
\end{lem}

\subsection{Tropical complex recursive formula}\label{sec42}

Introduce the set ${\cal S}$ of the admissible $6$-tuples $(\Del,
k_1, k_2, g, \alp, \beta)$, each $6$-tuple $\qquad\qquad$ $(\Del,
k_1, k_2, g, \alp, \beta)$ being $(l, r)$-admissible for certain
non-negative integer numbers~$l$ and~$r$ such that $l \leq r$
(recall that by the definition of $(l,r)$-admissibility, $l =
\|\alp\|$ and $r = |\Del| - I\alp - I\beta + \|\alp\| + \|\beta\| +
g - 1 - k_1 - k_2$). Define in ${\cal S}$ the following operation:
$$
\displaylines{ (\Del, k_1, k_2, g, \alp, \bet) + (\Del', k'_1, k'_2,
g', \alp', \bet') \cr =(\Del \cup \Del', k_1 + k'_1, k_2 + k'_2, g +
g' - 1, \alp + \alp', \bet + \bet',)\ . }
$$

Let~$l$ and~$r$ be non-negative integer numbers such that $l \leq
r$. Fix an $(l, r)$-admissible collection $(\Del, k_1, k_2, g, \alp,
\beta) \in {\cal S}$, choose a multi-set $\Theta$ dominating $\Del$,
and consider a $\CH_\Theta$-configuration~$\bx$ of type $(l, r)$.
Let $\varepsilon$, $\delta_l$, $\ldots$, $\delta_{r + 1}$ be
parameters of~$\bx$. For any $i = l$, $\ldots$, $r$, denote by $I_i$
the vertical segment with the endpoints $(\delta_i, -\varepsilon)$
and $(\delta_i, \varepsilon)$.

Denote by ${\cal T}^c(\Del, k_1, k_2, g, \alp, \bet, \bx) \subset
{\cal T}(\Del, k_1, k_2, g, \alp, \bet, \bx)$ the set formed by the
isomorphism classes of irreducible marked pseudo-simple $\CL$-curves
$(\overline\Gam, {\cal V}, h, \bpp)$ such that
\begin{enumerate}
\item[(i)] for any integer $i=l,...,r-1$
and any irreducible component $(\overline\Gam_{{\cal X}_i, j}, {\cal
V}_{{\cal X}_i, j}, h_{{\cal X}_i, j}, \bpp_{{\cal X}_i, j})$ of a
marked cut of $(\overline\Gam, {\cal V}, h, \bpp)$ at ${\cal X}_i =
h^{-1}(I_i)$, no two right irreducible components of a marked cut of
$(\overline\Gam_{{\cal X}_i, j}, {\cal V}_{{\cal X}_i, j}, h_{{\cal
X}_i, j}, \bpp_{{\cal X}_i, j})$ at ${\cal X}'_{i + 1} =
h^{-1}_{\CX_i}(I_{i + 1})$ are isomorphic,
\item[(ii)] any right irreducible component
$(\overline\Gam_{{\cal X}_r, j}, {\cal V}_{{\cal X}_r, j}, h_{{\cal
X}_r, j}, \bpp_{{\cal X}_r, j})$ of a marked cut of $\qquad$
$(\overline\Gam, {\cal V}, h, \bpp)$ at ${\cal X}_r = h^{-1}(I_r)$
has one of the combinatorial types presented in Figure~\ref{fn2}
(the collection $(\bpp_{{\cal X}_r, j})^\sharp$ is empty, and the
symbol~$\bullet$ which does not coincide with a univalent vertex
represents the only element in $(\bpp_{{\cal X}_r, j})^1 \cup
(\bpp_{{\cal X}_r, j})^2 \cup (\bpp_{{\cal X}_r, j})^\nu$; this
symbol~$\bullet$ is equipped with an index~$i$ if and only if the
corresponding point belongs to $(\bpp_{{\cal X}_r, j})^i$, $i = 1,
2$).
\end{enumerate}

\begin{figure}
\setlength{\unitlength}{1cm}
\begin{picture}(11,14)(-1,-1)
\thinlines\put(0,2){\line(1,0){3}}\put(3,2){\line(0,-1){1}}
\put(3,2){\line(1,1){1}}\put(0,6){\line(1,0){3}}
\put(3,6){\line(0,-1){1}}\put(3,6){\line(1,1){0.5}}\put(0,9){\line(1,0){4}}
\put(0,12.5){\line(1,0){4}}
\put(5.5,2){\line(1,0){2}}\put(5.5,2.5){\line(1,0){2}}\put(7.5,2){\line(0,11){
0.5}}
\put(7.5,2.5){\line(1,1){0.5}}\put(7.5,2){\line(1,-1){0.5}}\put(8,1.5){\line(0,
-1){0.5}}
\put(8,1.5){\line(1,0){2}}\put(5.5,6){\line(1,0){2}}\put(7.5,6){\line(1,1){0.5}}
\put(7.5,6){\line(1,-1){0.5}}\put(8,5.5){\line(0,-1){0.5}}\put(8,5.5){\line(1,0)
{2}}
\put(5.5,9.5){\line(1,0){2}}\put(8,9){\line(0,-1){0.5}}\put(7.5,9.5){\line(1,-1){
0.5}}
\put(8,9){\line(1,0){2}}\put(7.5,9.5){\line(0,1){0.5}}\put(5.5,13){\line(1,0){
2.5}}
\put(8,13){\line(1,1){0.5}}\put(8,13){\line(0,-1){0.5}}\put(8,12.5){\line(-1,-1)
{0.5}} \put(8,12.5){\line(1,0){2}}\put(2.91,5.37){$\bullet$}
\put(2.9,1.89){$\bullet$}\put(2.9,8.89){$\bullet$}
\put(7.9,1.38){$\bullet$}\put(7.9,5.4){$\bullet$}\put(7.9,8.87){$\bullet$}
\put(7.9,12.37){$\bullet$}
\put(-0.16,5.88){$\bullet$}\put(-0.16,12.38){$\bullet$}\put(5.35,1.89){$\bullet$}
\put(5.35,2.38){$\bullet$}
\put(5.35,5.88){$\bullet$}\put(5.35,9.38){$\bullet$}\put(5.35,12.87){$\bullet$}
\put(0,0){$\text{\rm (d)}\ L-E_1-E_2$}\put(5.5,0){$\text{\rm (h)}\
2L-E_1-E_2-E_5,\ \alp=2\theta_1$}\put(0,4){$\text{\rm (c)}\ L-E_i,\
i=1,2,\ \alp=\theta_1$}\put(5.5,4){$\text{\rm (g)}\ 2L-E_1-E_2-E_5,\
\alp=\theta_2$}\put(0,7.5){$\text{\rm (b)}\ L-E_i-E_5,\
i=1,2$}\put(5.5,7.5){$\text{\rm (f)}\ 2L-E_1-E_2-E_3-E_5,\
\alp=\theta_1$}\put(0,11){$\text{\rm (a)}\ L-E_5,\
\alp=\theta_1$}\put(5.5,11){$\text{\rm (e)}\ 2L-E_1-E_2-E_4-E_5,\
\alp=\theta_1$}\put(3.2,5.4){$i$}\put(3,9.2){$i$}\put(6.5,6.1){$2$}
\put(-0.5,-1){$\text{\rm All edges are of weight~$1$, except for the
left end of weight~$2$ in~(g).}$}
\end{picture}
\caption{Tropical initial conditions}\label{fn2}
\end{figure}

Notice that the set ${\cal T}^c(\Del, k_1, k_2, g, \alp, \bet, \bx)$
does not depend on the choice of parameters $\varepsilon$,
$\delta_l$, $\ldots$, $\delta_{r + 1}$, of~$\bx$.

For each $T=[(\overline\Gam, {\cal V}, h, \bpp)] \in{\cal T}^c(\Del,
k_1, k_2, g, \alp, \bet, \bx)$, define the {\it complex
multiplicity} $N(T)=N(\overline\Gam, {\cal V}, h, \bpp)$ as $AB$,
where~$A$ is the product of weights of all left ends of
$\overline\Gam$ which are not of $\alp$-type, and~$B$ is the product
of squares of weights of all bounded edges of $\overline\Gam$. Put
$$
{\cal N}(\Del, k_1, k_2, g, \alp, \bet, \bx) = \sum_{T\in {\cal
T}^c(\Del, k_1, k_2, g, \alp, \bet, \bx)} N(T)\ .
$$

\begin{rem}\label{rn56}
The complex multiplicity $N(T)$ defined above coincides with the
complex weight $M(T)$ introduced in \cite[Section 2.6]{Sh09}. This
can easily be checked applying \cite[Formula (7)]{Sh09} to the
classes $T\in{\cal T}^c(\Del, k_1, k_2, g, \alp, \bet, \bx)$.
\end{rem}

Assume that~$\Del$ has a subset~$\daleth$ formed by two vectors: one
with a positive second coordinate and the other with a negative
second coordinate. The sum of the vectors of $\daleth$ is either
$0$, or $(-1,0)$, or $(1,0)$. In the first case, put $\Del_\daleth$
to be the multi-set~$\Del\backslash\daleth$, and, in the two other
cases obtain $\Del_\daleth$ from ~$\Del\backslash\daleth$ by adding
or removing vector $(-1, 0)$ in such a way that the sum of the
vectors in $\Del_\daleth$ becomes~$0$.

\begin{proposition}\label{inv1}
Let~$l \leq r$ be non-negative integer numbers. Fix an $(l,
r)$-admissible collection $(\Del, k_1, k_2, g, \alp, \beta)$ and
choose a multi-set~$\Theta$ dominating~$\Del$. Let~$\bx$ be a
$\Theta$-generic $\CH_\Theta$-configuration of type $(l, r)$. Then,
the number ${\cal N}(\Del, k_1, k_2, g, \alp, \bet, \bx)$ does not
depend on the choices of~$\Theta$ and~$\bx$.
\end{proposition}

Such an independence allows us to write simply ${\cal N}(\Del, k_1,
k_2, g, \alp, \beta)$ for $\qquad\qquad$ ${\cal N}(\Del, k_1, k_2,
g, \alp, \bet, \bx)$ as soon as $(\Del, k_1, k_2, g, \alp, \beta)$
and~$\bx$ are as in Proposition~\ref{inv1}. Notice that according to
our definitions, ${\cal N}(\Del, k_1, k_2, g, \alp, \beta) = 0$
whenever $g < 0$.

The proof of Proposition~\ref{inv1} is given below, simultaneously
with the proof of the following theorem.

\begin{thm}\label{CHtrop-for}
Let~$l < r$ be non-negative integer numbers. If $(\Del, k_1, k_2, g,
\alp, \beta)$ is an $(l, r)$-admissible collection, then
$$
{\cal N}(\Del, k_1, k_2, g, \alp, \beta) =\sum_{j\ge 1,\ \beta_j>0}
j {\cal N}(\Del, k_1, k_2, g, \alp + \theta_j, \bet - \theta_j)
$$
\begin{equation}
+\sum\left(\begin{matrix}\alp\\
\alp^{(1)},...,\alp^{(m)}\end{matrix}\right) \frac{(r - l -
1)!}{(r^{(1)} - l^{(1)})!...(r^{(m)} - l^{(m)})!}\label{en21}
\end{equation}
$$\times\prod_{i=1}^m\left(
\left(\begin{matrix}\bet^{(i)}\\
\widetilde\bet^{(i)}\end{matrix}\right) I^{\widetilde\bet^{(i)}}
{\cal N}(\Del^{(i)}, k^{(i)}_1, k^{(i)}_2, g^{(i)}, \alp^{(i)},
\bet^{(i)})\right)\ ,
$$
where
$$
l^{(i)} = \|\alp^{(i)}\|, \;\;\; i = 1, \ldots, m,
$$
$$
r^{(i)} = |\Del^{(i)}| - I\alp^{(i)} - I\beta^{(i)} + \|\alp^{(i)}\|
+ \|\beta^{(i)}\| + g^{(i)} - 1 - k^{(i)}_1 - k^{(i)}_2, \;\;\; i =
1, \ldots, m,
$$
and the second sum in~(\ref{en21}) is taken
\begin{itemize}
\item over all subsets $\daleth$ of~$\Del$
which are formed by two vectors, one with a positive second
coordinate and the other with a negative second coordinate, and such
that the multi-set $\Del_\daleth$ contains at least one vector $(-1,
0)$,

\item over all
splittings
\begin{equation}
(\Del_\daleth, k_1, k_2, g', \alp', \bet') =\sum_{i=1}^m(\Del^{(i)},
k^{(i)}_1, k^{(i)}_2, g^{(i)}, \alp^{(i)}, \bet^{(i)})\label{en25}
\end{equation}
in ${\cal S}$ of all possible collections $(\Del_\daleth, k_1, k_2,
g', \alp', \bet') \in {\cal S}$ with
$$\displaylines{
\alp'\le\alp,\quad \bet\le\bet', \quad g - g' = \|\beta' - \beta\| -
1\ ,}$$ such that each summand $(\Del^{(i)}, k^{(i)}_1, k^{(i)}_2,
g^{(i)}, \alp^{(i)}, \bet^{(i)})$ with $r^{(i)} - l^{(i)}= 0$
appears in (\ref{en25}) at most once,
\item over all splittings
\begin{equation}
\bet'=\bet+\sum_{i=1}^m\widetilde\bet^{(i)},\quad
\|\widetilde\bet^{(i)}\|>0,\ i=1,...,m\ ,\label{en-beta}
\end{equation}
satisfying the restriction $\bet^{(i)}\ge\widetilde\bet^{(i)},\
i=1,...,m\ $,
\end{itemize}
and factorized by simultaneous permutations in the both splittings
(\ref{en25}) and (\ref{en-beta}).
\end{thm}

{\bf Proof of Proposition~\ref{inv1} and Theorem~\ref{CHtrop-for}}
(cf.~\cite{GM2} and~\cite{IKS3}). In the case $r = l$, the statement
of Proposition~\ref{inv1} immediately follows from the property~(4)
in the definition of weak $\CH_\Theta$-configurations and the
property~(ii) of the set ${\cal T}^c(\Del, k_1, k_2, g, \alp, \bet,
\bx)$.

Consider now non-negative integers $l < r$, fix an $(l,
r)$-admissible collection $\qquad$ $(\Del, k_1, k_2, g, \alp,
\beta)$, and choose a multi-set~$\Theta$ dominating~$\Del$. Assume
that for any non-negative integer numbers $l^\checkmark$ and
$r^\checkmark$ satisfying the inequalities $l^\checkmark \leq
r^\checkmark$ and $r^\checkmark - l^\checkmark < r - l$, the fact
that the numbers ${\cal N}(\Del^\checkmark, k^\checkmark_1,
k^\checkmark_2, g^\checkmark, \alp^\checkmark, \bet^\checkmark,
\bx^\checkmark)$, where $(\Del^\checkmark, k^\checkmark_1,
k^\checkmark_2, g^\checkmark, \alp^\checkmark, \bet^\checkmark)$ is
an $(l^\checkmark, r^\checkmark)$-admissible collection, and
$\Theta$ dominates $\Del^\checkmark$, do not depend on the choice of
a $\CH_\Theta$-configuration~$\bx^\checkmark$ of type
$(l^\checkmark, r^\checkmark)$ is already established.

Pick a $\Theta$-generic $\CH_\Theta$-configuration~$\bx$ of type
$(l, r)$ and assume that there exists an irreducible marked
$\CL$-curve $(\overline\Gam, {\cal V}, h, \bpp)$ whose isomorphism
class belongs to $\qquad$ ${\cal T}^c(\Del, k_1, k_2, g, \alp, \bet,
\bx)$. Suppose, first, that the point~$P_{l + 1} \in \bpp^\sharp$
belongs to a left end~$E$ of~$\overline\Gam$. Denote by~$j$ the
weight of~$E$, and by~$V$ the unique univalent vertex incident
to~$E$. Consider the $\CH_\Theta$-configuration $\widehat\bx$ which
is a $1$-projection of $\bx$ such that the horizontal projection of
$p_{l + 1}$ is inserted in the $j$-th group of points in
$\bx^\flat$. Consider also an $(l + 1, r)$-admissible collection
$(\Del, k_1, k_2, g, \alp + \theta_j, \bet - \theta_j)$, and a
marked $\CL$-curve $(\overline\Gam, {\cal V}, h, \bpp')$ such that
$(\bpp')^\flat = \bpp^\flat \cup \{V\}$, $(\bpp')^\sharp =
\bpp^\sharp \setminus \{P_{l + 1}\}$, and $(\bpp')^\aleph =
\bpp^\aleph$ for any $\aleph \in \{1, 2, \nu\}$.

The isomorphism class of the curve $(\overline\Gam, {\cal V}, h,
\bpp')$ belongs to $\qquad\qquad\qquad\qquad\qquad$ \mbox{${\cal
T}^c(\Del, k_1, k_2, g, \alpha + \theta_j, \beta - \theta_j,
\widehat\bx)$}, and
$$
N(\overline\Gam, {\cal V}, h, \bpp') = \frac{1}{j} N(\overline\Gam,
{\cal V}, h, \bpp).
$$
The described procedure establishes a bijection between \mbox{${\cal
T}^c(\Del, k_1, k_2, g, \alpha + \theta_j, \beta - \theta_j,
\widehat\bx)$} and those isomorphisms classes in \mbox{${\cal
T}^c(\Del, k_1, k_2, g, \alpha, \beta, \bx)$} that are realized by
curves $(\overline\Gam, {\cal V}, h, \bpp)$ such that $P_{l + 1}$
belongs to a left end of $\overline\Gam$ of weight~$j$. Thus, by the
induction assumption, the contribution to $N(\Del, k_1, k_2, g,
\alp, \bet, \bx)$ of the latter isomorphism classes is equal to
$$\sum_{j\ge 1,\ \beta_j>0}
jN(\Del, k_1, k_2, g, \alpha + \theta_j, \beta - \theta_j)\ .$$

Suppose now that $P_{l + 1}$ does not belong to any left end of
$\overline\Gam$. Let $\varepsilon$, $\delta_l$, $\ldots$, $\delta_{r
+ 1}$ be parameters of the $\CH_\Theta$-configuration $\bx$.
Consider a marked cut $(\overline\Gam_{\cal X}, {\cal V}_{\cal X},
h_{\cal X}, \bpp_{\cal X})$ of $(\overline\Gam, {\cal V}, h, \bpp)$
at ${\cal X} = h^{-1}(I_{l + 1})$, where $I_{l + 1}$ is the vertical
segment with endpoints $(\delta_{l + 1}, -\varepsilon)$ and
$(\delta_{l + 1}, \varepsilon)$.

Let $(\overline\Gam_{{\cal X}, j_i}, {\cal V}_{{\cal X}, j_i},
h_{{\cal X}, j_i}, \bpp_{{\cal X}, j_i})$, $i = 1$, $\ldots$, $m$,
be the right irreducible components of $(\overline\Gam_{\cal X},
{\cal V}_{\cal X}, h_{\cal X}, \bpp_{\cal X})$. For each $i = 1$,
$\ldots$ $m$, introduce the following numbers:
\begin{itemize}
\item $l^{(i)}$ is the number
of points in $(\bpp_{{\cal X}, j_i})^\flat$, and $r^{(i)}$ is the
number of points in $(\bpp_{{\cal X}, j_i})^\flat \cup (\bpp_{{\cal
X}, j_i})^\sharp$,
\item $\alp^{(i)}_t$ (respectively, $\beta^{(i)}_t$), $t$ being a positive integer,
is the number of those left ends of weight~$t$ in
$\overline\Gam_{{\cal X}, j_i}$ which are of $\alp$-type
(respectively, not of $\alp$-type),
\item $g^{(i)}$ is the genus of
$(\overline\Gam_{{\cal X}, j_i}, {\cal V}_{{\cal X}, j_i}, h_{{\cal
X}, j_i}, \bpp_{{\cal X}, j_i})$,
\item $k^{(i)}_1$ (respectively, $k^{(i)}_2$)
is the number of points in $(\bpp_{{\cal X}, j_i})^1 \cup
(\bpp_{{\cal X}, j_i})^\nu$ (respectively, $(\bpp_{{\cal X}, j_i})^2
\cup (\bpp_{{\cal X}, j_i})^\nu$).\end{itemize} Let $\bx_{j_i}$ be a
configuration in ${\cal P}(l^{(i)}, r^{(i)})$ such that
$$(\bx_{j_i})^\flat =
\widehat{h}_{{\cal X}, j_i}((\bpp_{{\cal X}, j_i})^\flat), \;\;\;
(\bx_{j_i})^\sharp = \widehat{h}_{{\cal X}, j_i}((\bpp_{{\cal X},
j_i})^\sharp),$$ and the $(r^{(i)} + 1)$-th point of $\bx_{j_i}$
coincides with the $(r + 1)$-th point $\bx_{r + 1}$ of~$\bx$. For
each $i=1,...,m$, the collection $(\Del^{(i)}, k^{(i)}_1, k^{(i)}_2,
g^{(i)}, \alp^{(i)}, \beta^{(i)})$ is $(l^{(i)},
r^{(i)})$-admissible, and
$$
[(\overline\Gam_{{\cal X}, j_i}, {\cal V}_{{\cal X}, j_i}, h_{{\cal
X}, j_i}, \bpp_{{\cal X}, j_i})] \in {\cal T}^c(\Del^{(i)},
k^{(i)}_1, k^{(i)}_2, g^{(i)}, \alp^{(i)}, \beta^{(i)}, \bx_{j_i})\
.
$$

Since $P_{l + 1}$ does not belong to any left end of
$\overline\Gam$, among the left irreducible components of
$(\overline\Gam_{\cal X}, {\cal V}_{\cal X}, h_{\cal X}, \bpp_{\cal
X})$ there is one (and exactly one) which is one-sheeted. Denote
this component by $(\overline\Gam_{\CX, \LIC}, \CV_{\CX, \LIC},
h_{\CX, \LIC}, \bpp_{\CX, \LIC})$. The point $P_{l + 1}$ belongs to
$\overline\Gam_{\CX, \LIC}$ and lies on a non-horizontal edge. Any
left end of $\overline\Gam_{\CX, \LIC}$ is of $\alp$-type. The
degree of $(\overline\Gam_{\CX, \LIC}, \CV_{\CX, \LIC}, h_{\CX,
\LIC}, \bpp_{\CX, \LIC})$ contains two vectors having non-zero
second coordinate. Let~$\daleth$ be the set formed by these two
vectors.

The sum
\begin{equation}
\sum_{i=1}^m(\Del^{(i)}, k^{(i)}_1, k^{(i)}_2, g^{(i)}, \alp^{(i)},
\bet^{(i)}),\label{en21-new}
\end{equation}
taken in $\cal S$, gives a collection $(\Del_\daleth, k_1, k_2, g',
\alp', \bet')$ with certain $\alp',\bet',g'$. For each right
irreducible component $(\overline\Gam_{\CX, j_i}, \CV_{\CX, j_i},
h_{\CX, j_i}, \bpp_{\CX, j_i})$, any left end of $\alp$-type of this
component matches the edge of a horizontal left irreducible
component, and the edge-predecessor of the two edges in question is
of $\alp$-type. Thus, $\alp' \leq \alp$. Since any left end of
$\overline\Gam_{\CX, \LIC}$ is of $\alp$-type, we obtain $\beta'
\geq \beta$. For each $i = 1$, $\ldots$, $m$, let the sequence
$\widetilde\beta^{(i)}$ encode the weights of the left ends of
$\overline\Gam_{\CX, j_i}$ which match edges of $\overline\Gam_{\CX,
\LIC}$. Notice that $\|\widetilde\beta^{(i)}\| > 0$ for any $i = 1$,
$\ldots$, $m$. Furthermore, $g - g' = \|\beta' - \beta\| - 1$.
Finally, the fact that each summand $(\Del^{(i)}, k^{(i)}_1,
k^{(i)}_2, g^{(i)}, \alp^{(i)}, \bet^{(i)})$ with $r^{(i)} -
l^{(i)}= 0$ cannot not appear twice in the sum~(\ref{en21-new})
follows from the property~(i) of the set ${\cal T}^c(\Del, k_1, k_2,
g, \alp, \beta, \bx)$. Thus, each marked $\CL$-curve whose
isomorphism class belongs to ${\cal T}^c(\Del, k_1, k_2, g, \alp,
\beta, \bx)$ gives rise to a subset $\daleth \subset \Del$ and a
pair of splittings~(\ref{en25}) and~(\ref{en-beta}) satisfying all
the conditions mentioned in the theorem.

Assume now that we are given a subset $\daleth \subset \Del$, a
collection $(\Del_\daleth, k_1, k_2, g', \alp', \bet')$, and a pair
of splittings~(\ref{en25}) and~(\ref{en-beta}) such that all these
data obey the restrictions listed in the theorem. Put
$$
l^{(i)} = \|\alp^{(i)}\|, \;\;\; i = 1, \ldots, m,
$$
$$
r^{(i)} = |\Del^{(i)}| - I\alp^{(i)} - I\beta^{(i)} + \|\alp^{(i)}\|
+ \|\beta^{(i)}\| + g^{(i)} - 1 - k^{(i)}_1 - k^{(i)}_2, \;\;\; i =
1, \ldots, m.
$$

Choose~$m$ pairwise disjoint subsequences $(\bx^{(1)})^\flat$,
$\ldots$, $(\bx^{(m)})^\flat$ of $\bx^\flat$ such that, for each $i
= 1$, $\ldots$, $m$, and each positive integer number $t$, the
number of points in $(\bx^{(i)})^\flat$ which belong to the $t$-th
group of $\bx^\flat$ is equal to $\alp^{(i)}_t$ (the number of
possible choices is equal to
$\left(\begin{matrix}\alp\\
\alp^{(1)},...,\alp^{(m)}\end{matrix}\right)$). Choose~$m$ pairwise
disjoint subsequences $(\bx^{(1)})^\sharp$, $\ldots$,
$(\bx^{(m)})^\sharp$ of $\bx^\sharp \setminus \{p_{l + 1}\}$ such
that, for each $i = 1$, $\ldots$, $m$, the number of points in
$(\bpp^{(i)})^\sharp$ is equal to $r^{(i)} - l^{(i)}$ (the number of
possible choices is equal to $\frac{(r - l - 1)!}{(r^{(1)} -
l^{(1)})!...(r^{(m)} - l^{(m)})!}$). For each $i = 1$, $\ldots$,
$m$, denote by $\bx^{(i)}$ the sequence of points formed by the
sequence $(\bx^{(i)})^\flat$, the sequence $(\bx^{(i)})^\sharp$, and
the point $p_{r + 1}$. Each of these~$m$ sequences is a
$\Theta$-generic $\CH_\Theta$-configuration. If all the sets ${\cal
T}^c(\Del^{(i)}, \alp^{i}, \beta^{(i)}, g^{(i)}, k^{(i)}_1,
k^{(i)}_2, \bx^{(i)})$ are non-empty, then, for each $i = 1$,
$\ldots$, $m$, pick a marked $\CL$-curve $(\overline\Gam^{(i)},
\CV^{(i)}, h^{(i)}, \bpp^{(i)})$ whose isomorphism class belongs to
${\cal T}^c(\Del^{(i)}, \alp^{i}, \beta^{(i)}, g^{(i)}, k^{(i)}_1,
k^{(i)}_2, \bx^{(i)})$. For each $i = 1$, $\ldots$, $m$, and each
positive integer number~$t$, choose $\widetilde\beta^{(i)}_t$ left
ends of $\overline\Gam^{(i)}$ which are of weight~$t$ and are not of
$\alp$-type (the number of possible
choices is equal to $\left(\begin{matrix}\bet^{(i)}\\
\widetilde\bet^{(i)}\end{matrix}\right)$).

There exists a unique isomorphism class
$$
[(\overline\Gam, \CV, h, \bpp)] \in {\cal T}^c(\Del, k_1, k_2, g,
\alp, \beta, \bx)
$$
such that the point $P_{l + 1}$ does not belong to a left end of
$\overline\Gam$ and the curves $(\overline\Gam^{(i)}, \CV^{(i)},
h^{(i)}, \bpp^{(i)})$, $i = 1$, $\ldots$, $m$, are the right
irreducible components of a marked cut $(\overline\Gam_\CX, \CV_\CX,
h_\CX, \bpp_\CX)$ of $(\overline\Gam, \CV, h, \bpp)$ at $\CX =
h^{-1}(I_{l + 1})$. Indeed, the isomorphism class of a unique left
one-sheeted irreducible component $(\overline\Gam_{\CX, \LIC},
\CV_{\CX, \LIC}, h_{\CX, \LIC}, \bpp_{\CX, \LIC})$ of
$(\overline\Gam_\CX, \CV_\CX, h_\CX, \bpp_\CX)$ is given, up to
composition of~$h$ with a horizontal shift, by the degree of this
component (the degree is the two vectors of $\daleth$ completed by
$\alp_i - \alp'_i$ vectors $(-i, 0)$ and $\beta'_i - \beta_i$
vectors $(i, 0)$ for any positive integer~$i$), the points in
$\bx^\flat \setminus \cup_{i = 1}^m (\bx^{(i)})^\flat$ (together
with the distribution of these points by groups), and the heights
and the weights of the chosen left ends of $(\overline\Gam^{(i)},
\CV^{(i)}, h^{(i)}, \bpp^{(i)})$, $i = 1$, $\ldots$, $m$ (see
\cite[Lemma 29]{IKS3}). The horizontal shift is uniquely determined
by the position of the point $p_{l + 1}$, since this point should
belong to the image of a non-horizontal edge of
$(\overline\Gam_{\CX, \LIC}, \CV_{\CX, \LIC}, h_{\CX, \LIC},
\bpp_{\CX, \LIC})$ due to the assumption that~$\bx$ is a
$\Theta$-generic configuration. Furthermore,
$$
N(\overline\Gam, \CV, h, \bpp) = \prod_{i = 1}^m
\left(I^{\widetilde\beta^{(i)}} N(\overline\Gam^{(i)}, \CV^{(i)},
h^{(i)}, \bpp^{(i)})\right).
$$
Thus, the contribution to $N(\Del, k_1, k_2, g, \alp, \bet, \bx)$ of
the sets $\qquad\qquad\qquad\qquad\qquad\qquad$ ${\cal
T}^c(\Del^{(i)}, k^{(i)}_1, k^{(i)}_2, g^{(i)}, \alp^{(i)},
\beta^{(i)}, \bx^{(i)})$, $i = 1$, $\ldots$, $m$, is equal to
$$
\left(\begin{matrix}\alp\\
\alp^{(1)},...,\alp^{(m)}\end{matrix}\right) \frac{(r - l -
1)!}{(r^{(1)} - l^{(1)})!...(r^{(m)} - l^{(m)})!}
$$
$$\times\prod_{i=1}^m\left(
\left(\begin{matrix}\bet^{(i)}\\
\widetilde\bet^{(i)}\end{matrix}\right) I^{\widetilde\bet^{(i)}}
{\cal N}(\Del^{(i)}, k^{(i)}_1, k^{(i)}_2, g^{(i)}, \alp^{(i)},
\bet^{(i)})\right)\ .
$$
This implies that the number ${\cal N}(\Del, k_1, k_2, g, \alp,
\bet, \bx)$ depends neither on the choice of $\Theta$, nor on the
choice of a $\Theta$-generic $\CH_\Theta$-configuration~$\bx$ of
type $(l, r)$, and proves the formula~(\ref{en21}). \proofend

\begin{proposition}\label{initial-conditions}
All the numbers ${\cal N}(\Del, k_1, k_2, g, \alp, \beta)$, where
$(\Del, k_1, k_2, g, \alp, \beta) \in {\cal S}$, are recursively
determined by formula {\rm (}\ref{en21}{\rm )} of
Theorem~\ref{CHtrop-for} and the following initial values: for any
non-negative integer~$r$ and any $(r, r)$-admissible collection
$(\Del, k_1, k_2, g, \alp, \beta)$, the number ${\cal N}(\Del, k_1,
k_2, g, \alp, \beta)$ is equal to~$0$ or~$1$, and this number is
equal to~$1$ if and only if the collection $(\Del, k_1, k_2, g,
\alp, \beta)$ is of one of the combinatorial types presented on
Figure~\ref{fn2}, {\it i.e.} $(\Del, k_1, k_2, g, \alp, \beta)$
coincides with one of the following collections:
\begin{itemize}
\item $(\{(-1, 0), (1, 0)\}, 0, 0, 0, (1), (0))$,
\item $(\{(-1, 0), (1, 0)\}, 1, 0, 0, (0), (1))$,
\item $(\{(-1, 0), (1, 0)\}, 0, 1, 0, (0), (1))$,
\item $(\{(-1, 0), (0, -1), (1, 1)\}, 1, 0, 0, (1), (0))$,
\item $(\{(-1, 0), (0, -1), (1, 1)\}, 0, 1, 0, (1), (0))$,
\item $(\{(-1, 0), (0, -1), (1, 1)\}, 1, 1, 0, (0), (1))$,
\item $(\{(-1, 0), (-1, -1), (1, 0), (1, 1)\}, 1, 1, 0, (1), (0))$,
\item $(\{(-1, 0), (0, -1), (1, 0), (0, 1)\}, 1, 1, 0, (1), (0))$,
\item $(\{(-1, 0), (-1, 0), (0, -1), (1, 0), (1, 1)\}, 1, 1, 0, (0, 1), (0))$,
\item $(\{(-1, 0), (-1, 0), (0, -1), (1, 0), (1, 1)\}, 1, 1, 0, (2), (0))$,
\end{itemize}
\end{proposition}

\begin{proof} The description of the numbers ${\cal N}(\Del, k_1,
k_2, g, \alp, \beta)$, where $(\Del, k_1, k_2, g, \alp, \beta)$ is
an $(r, r)$-admissible collection and~$r$ a non-negative integer,
follows from the property~(4) in the definition of weak
$\CH_\Theta$-configurations and the property~(ii) of sets $\qquad$
${\cal T}^c(\Del, k_1, k_2, g, \alp, \bet, \bx)$. All the other
numbers ${\cal N}(\Del, k_1, k_2, g, \alp, \beta)$, where
$\qquad\qquad$ $(\Del, k_1, k_2, g, \alp, \beta) \in {\cal S}$, can
be expressed as linear combinations of these values by successive
use of formula~(\ref{en21}). \end{proof}

\subsection{Tropical real recursive formulas}\label{sec43}
Denote by $(\Z^\infty_+)^{\text{\rm odd}}$ the sub-semigroup of
$\Z^\infty_+$ formed by the sequences~$\alp$ such that $\alp_{2i} =
0$ for any positive integer~$i$, and denote by ${\cal S}^{\text{\rm
odd}}$ the sub-semigroup of $\cal S$ formed by the $6$-tuples
$(\Del, k_1, k_2, g, \alp, \beta)$ such that $\alp, \beta \in
(\Z^\infty_+)^{\text{\rm odd}}$.

Denote by ${\cal S}^{\text{\rm odd, sym}}$ the sub-semigroup of
${\cal S}^{\text{\rm odd}}$ formed by the $6$-tuples $(\Del, k_1,
k_2, g, \alp, \beta)$ such that $k_1 = k_2$.

Let~$l$ and~$r$ be non-negative integer numbers such that $l \leq
r$. Choose an $(l, r)$-admissible collection $(\Del, k_1, k_2, g,
\alp, \beta) \in {\cal S}^{\text{\rm odd}}$, a finite multi-set
$\Theta$ dominating $\Del$, and a $\Theta$-generic
$\CH_\Theta$-configuration $\bx$.

For each isomorphism class $T=[(\overline\Gam, {\cal V}, h,
\bpp)]\in {\cal T}^c(\Del, k_1, k_2, g, \alp, \bet, \bx)$, define
the {\it Welschinger multiplicity} $W(T)=W(\overline\Gam, {\cal V},
h, \bpp)$ to be equal to~$1$ if all the edges of $\overline\Gam$
have odd weight, and~$0$ otherwise. Put
$$
{\cal W}(\Del, k_1, k_2, g, \alp, \bet, \bx) = \sum_{T \in {\cal
T}^c(\Del, k_1, k_2, g, \alp, \bet, \bx)} W(T)\ .
$$
For any element $(\Del, k_1, k_2, g, \alp, \bet)\in{\cal
S}^{\text{\rm odd,sym}}$ different from $K^{sp}$ (see
Section~\ref{CH-config}), a finite multi-set $\Theta$ dominating
$\Del$, and a $\Theta$-generic $\CH_\Theta$-configuration $\bx$,
introduce a subset ${\cal T}^{c, \text{\rm sym}}(\Del, k_1, k_2, g,
\alp, \bet, \bx) \subset {\cal T}^{\text{\rm sym}}(\Del, k_1, k_2,
g, \alp, \bet, \bx)$ defined by
$$
{\cal T}^{c, \text{\rm sym}}(\Del, k_1, k_2, g, \alp, \bet, \bx)
=\Psi^{-1}({\cal T}^c(\Del, k_1, k_2, g, \alp, \beta, \bx))\ ,
$$ where $\Psi$ is the forgetful map (\ref{esym1}).
In addition, for any point $p_1 \in \R^2$, we put
$${\cal T}^{c, \text{\rm sym}}(K^{sp}, \{p_1\}) = \{T^{sp}(p_1)\}\ ,$$
where $T^{sp}(p _1)$ is as in Section~\ref{CH-config}. We consider
the following Welschinger multiplicities:
\begin{itemize}
\item $W^{\text{\rm sym}}(T)=
W(\Psi(T))$ for all $T\in{\cal T}^{c, \text{\rm sym}} (\Del, k_1,
k_2, g, \alp, \bet, \bx)$ with $\qquad\qquad$ $(\Del, k_1, k_2, g,
\alp, \beta) \ne K^{sp}$,
\item $W^{\text{\rm sym}}(T^{sp}(p_1)) = 1$ for any point $p_1 \in \R^2$,
\end{itemize}
and then put
$$
{\cal W}^{\text{\rm sym}}(\Del, k_1, k_2, g, \alp, \bet, \bx) =
\sum_{T \in {\cal T}^{c,\text{\rm sym}} (\Del, k_1, k_2, g, \alp,
\bet, \bx)} W^{\text{\rm sym}}(T)\ .
$$

\begin{proposition}\label{inv1r}
Let~$l \leq r$ be non-negative integer numbers. Fix an $(l,
r)$-admissible collection $(\Del, k_1, k_2, g, \alp, \beta)\in{\cal
S}^{\text{\rm odd}}$ {\rm (}respectively, $(\Del, k_1, k_2, g, \alp,
\beta)\in{\cal S}^{\text{\rm odd,sym}}${\rm )}, and choose a
multi-set~$\Theta$ dominating~$\Del$. Let~$\bx$ be a
$\Theta$-generic $\CH_\Theta$-configuration of type $(l, r)$. Then,
the number ${\cal W}(\Del, k_1, k_2, g, \alp, \bet, \bx)$
(respectively, ${\cal W}^{\text{\rm sym}}(\Del, k_1, k_2, g, \alp,
\bet, \bx)$) does not depend on the choices of~$\Theta$ and~$\bx$.
\end{proposition}

\begin{proof} If $(\Del, k_1, k_2, g, \alp, \beta) \ne K^{sp}$, the
proof is completely similar to the proof of Proposition~\ref{inv1}.
Furthermore, ${\cal W}^{\text{\rm sym}}(T^{sp}(p_1)) = 1$ for any
point $p_1 \in \R^2$. \end{proof}

Such an independence allows us to write simply ${\cal W}(\Del, k_1,
k_2, g, \alp, \beta)$ for $\qquad\qquad$ ${\cal W}(\Del, k_1, k_2,
g, \alp, \bet, \bx)$ as soon as $(\Del, k_1, k_2, g, \alp, \beta)$
and~$\bx$ are as in Proposition~\ref{inv1r}. The latter requirement
on $(\Del, k_1, k_2, g, \alp, \beta)$ and~$\bx$ is similar to the
requirement on the constraints related to the enumerative invariants
which appear in the recursive formula presented in~\cite{ABLdM}.

The proof of the following Proposition~\ref{CHtrop-for-real} is
completely similar to the proof of Theorem~\ref{CHtrop-for}. The
only additional observation consists of the fact that, in the case
of symmetric marked ${\cal L}$-curves, the sets $h^{-1}(I_i)$ are
invariant with respect to the involution, and each cut inherits a
uniquely defined involution.

\begin{proposition}\label{CHtrop-for-real}
Let~$l < r$ be non-negative integer numbers. Fix an $(l,
r)$-admissible collection $(\Del, k_1, k_2, g, \alp, \beta)\in{\cal
S}^*$, where ${\cal S}^*$ stands either for ${\cal S}^{\text{\rm
odd}}$, or for ${\cal S}^{\text{\rm odd,sym}}$. Then
$$
{\cal W}^*(\Del, k_1, k_2, g, \alp, \beta) =\sum_{j\ge1,\ \beta_j>0}
{\cal W}^*(\Del, k_1, k_2, g, \alp + \theta_j, \bet - \theta_j)
$$
\begin{equation}
+\sum\left(\begin{matrix}\alp\\
\alp^{(1)},...,\alp^{(m)}\end{matrix}\right) \frac{(r - l -
1)!}{(r^{(1)} - l^{(1)})!...(r^{(m)} - l^{(m)})!}\label{en26}
\end{equation}
$$\times\prod_{i=1}^m\left(
\left(\begin{matrix}\bet^{(i)}\\
\widetilde\bet^{(i)}\end{matrix}\right) {\cal W}^*(\Del^{(i)},
k^{(i)}_1, k^{(i)}_2, g^{(i)}, \alp^{(i)}, \bet^{(i)})\right)\ ,
$$
where ${\cal W}^*$ stands for ${\cal W}$ {\rm (}respectively, ${\cal
W}^{\text{\rm sym}}${\rm )} if ${\cal S}^*={\cal S}^{\text{\rm
odd}}$ {\rm (}respectively, ${\cal S}^*={\cal S}^{\text{\rm
odd,sym}}${\rm )}, and
$$
l^{(i)} = \|\alp^{(i)}\|, \;\;\; i = 1, \ldots, m,
$$
$$
r^{(i)} = |\Del^{(i)}| - I\alp^{(i)} - I\beta^{(i)} + \|\alp^{(i)}\|
+ \|\beta^{(i)}\| + g^{(i)} - 1 - k^{(i)}_1 - k^{(i)}_2, \;\;\; i =
1, \ldots, m,
$$
and the second sum in~(\ref{en26}) is taken
\begin{itemize}
\item over all subsets $\daleth$ of~$\Del$
which are formed by two vectors, one with positive second coordinate
and one with negative second coordinate, and such that the multi-set
$\Del_\daleth$ contains at least one vector $(-1, 0)$,
\item over all
splittings
\begin{equation}
(\Del_\daleth, k_1, k_2, g', \alp', \bet') =\sum_{i=1}^m(\Del^{(i)},
k^{(i)}_1, k^{(i)}_2, g^{(i)}, \alp^{(i)}, \bet^{(i)},)\label{en27}
\end{equation}
in ${\cal S}^*$ of all possible collections $(\Del_\daleth, k_1,
k_2, g', \alp', \bet') \in {\cal S}^*$, such that
$$\displaylines{
\alp'\le\alp,\quad \bet\le\bet', \quad g - g' = \|\beta' - \beta\| -
1\ ,}$$ and each summand $(\Del^{(i)}, k^{(i)}_1, k^{(i)}_2,
g^{(i)}, \alp^{(i)}, \bet^{(i)}, )$ with $r^{(i)} - l^{(i)} = 0$
appears in (\ref{en27}) at most once,
\item over all splittings
\begin{equation}\bet'=\bet+\sum_{i=1}^m\widetilde\bet^{(i)},
\quad\|\widetilde\bet^{(i)}\|>0,\ i=1,...,m\
,\label{en97}\end{equation} satisfying the restriction
$\bet^{(i)}\ge\widetilde\bet^{(i)},\ i=1,...,m\ $,
\end{itemize}
and factorized by simultaneous permutations in the both splittings
(\ref{en27}) and (\ref{en97}).\proofend\end{proposition}

The proof of the next statement coincides with the proof of
Proposition \ref{initial-conditions}.

\begin{proposition}\label{initial-conditionsr}
All the numbers ${\cal W}(\Del, k_1, k_2, g, \alp, \beta)$, where
$(\Del, k_1, k_2, g, \alp, \beta) \in {\cal S}^*$, and ${\cal S}^*$
stands either for ${\cal S}^{\text{\rm odd}}$, or for ${\cal
S}^{\text{\rm odd,sym}}$, are recursively determined by formula {\rm
(}\ref{en26}{\rm )} of Proposition~\ref{CHtrop-for-real} and the
following initial values. For any non-negative integer~$r$ and any
$(r, r)$-admissible collection $(\Del, k_1, k_2, g, \alp, \beta) \in
{\cal S}^*$, the number ${\cal W}^*(\Del, k_1, k_2, g, \alp, \beta)$
is equal either to~$0$, or to~$1$. Furthermore,
\begin{enumerate}
\item[--] if ${\cal S}^*={\cal S}^{\text{\rm odd}}$, then ${\cal
W}(\Del, k_1, k_2, g, \alp, \beta) = 1$ if and only if the
collection $\qquad$ $(\Del, k_1, k_2, g, \alp, \beta)$ coincides
with one in the list presented in
Proposition~\ref{initial-conditions},
\item[--] if
${\cal S}^*={\cal S}^{\text{\rm odd,sym}}$, then ${\cal
W}^{\text{\rm sym}}(\Del, k_1, k_2, g, \alp, \beta) = 1$ if and only
if $(\Del, k_1, k_2, g, \alp, \beta)$ coincides with one of the
following collections {\rm (}corresponding curves are shown in
Figures~\ref{fn2}{\rm (}a,d-h{\rm )} and~\ref{lines}{\rm )}:
\begin{itemize}
\item $(\{(-1, 0), (1, 0)\}, 0, 0, 0, (1), (0))$,
\item $(\{(-1, 0), (0, -1), (1, 1)\}, 1, 1,0, (0), (1))$,
\item $(\{(-1, 0), (-1, -1), (1, 0), (1, 1)\}, 1, 1,0, (1), (0))$,
\item $(\{(-1, 0), (0, -1), (1, 0), (0, 1)\}, 1, 1,0, (1), (0))$,
\item $(\{(-1, 0), (-1, 0), (0, -1), (1, 0), (1, 1)\}, 1, 1,0, (0, 1), (0))$,
\item $(\{(-1, 0), (-1, 0), (0, -1), (1, 0), (1, 1)\}, 1,
1,0, (2), (0))$,
\item $(\{(-1, 0), (-1, 0), (1, 0), (1, 0)\}, 1, 1,-1, (0), (2))$. \proofend
\end{itemize}
\end{enumerate}
\end{proposition}

\section{Correspondence theorem}\label{corresp-thm}

\subsection{Auxiliary statements}\label{auxiliary}
Let $l \leq r$ be  non-negative integers, and $\Theta$ a finite
multi-set of vectors in $\Z^2$. A configuration~$\bx \in {\cal P}(l,
r)$ is called {\it $\Theta$-proper}, if
\begin{itemize}
\item $\bx$ is a $\Theta$-generic $\CH_\Theta$-configuration,
\item
for any integers $0 \leq l' < r'$ and $0 \leq l'' \leq r''$, any
disjoint sub-configurations $\bx'$ and $\bx''$ of~$\bx$ such that
$\bx' \in {\cal P}(l', r')$ and $\bx'' \in {\cal P}(l'', r'')$, any
$(l', r')$-admissible collection $(\Del', k'_1, k'_2, g', \alp',
\bet')$ and any $(l'', r'')$-admissible collection $(\Del'', k''_1,
k''_2, g'', \alp'', \beta'')$ such that~$\Theta$ dominates~$\Del'$
and $\Del''$, one has the following property:
\begin{enumerate}
\item[]
for any elements
$$[(\overline\Gam',{\cal V}',h',\bpp')]\in{\cal T}^c
(\Del', k'_1, k'_2, g', \alp', \bet', \bx')\ ,$$
$$[(\overline\Gam'',{\cal V}'',h'',\bpp'')]\in{\cal T}^c
(\Del'', k''_1, k''_2, g'', \alp'', \bet'', \bx'')\ ,$$ the map
${\widehat h}'\cup {\widehat h}'': ((\Gam')^0_\infty \backslash
{\cal V'}) \cup ((\Gam'')^0_\infty \backslash {\cal V}'') \to
L_{-\infty}$ is injective.
\end{enumerate}
\end{itemize}

\begin{lem}\label{ln10}
For any non-negative integers $l < r$ and any finite
multi-set~$\Theta$ of vectors in $\Z^2$, there exists a
$\Theta$-proper configuration $\bx \in {\cal P}(l, r)$.
\end{lem}

{\bf Proof.} Pick a configuration $\widetilde\bx$ in the interior of
the set of $\Theta$-generic $\CH_\Theta$-configurations of type $(l,
r)$ such that, for any integers $0 \leq l' < r'$, any $(l',
r')$-admissible collection $(\Del', k'_1, k'_2, g', \alp', \bet')$,
and any sub-configuration $\widetilde\bx' \subset \widetilde\bx$ of
type $(l', r')$, one has the following property:
\begin{enumerate}
\item[(*)] $\widetilde\bx'$ does not belong to the image
of any affine map $ev^\lambda: {\cal M}^\lambda(\Del', k'_1, k'_2,
g', \alp', \beta') \to (L_{-\infty})^{l'} \times (\R^2)^{r' + 1 -
l'}$, where ${\cal M}^\lambda(\Del', k'_1, k'_2, g', \alp', \beta')$
is of dimension $< 2r' + 2 - l'$.
\end{enumerate}
Notice that the conditions imposed on $\widetilde\bx$ are open. A
point $p \in L_{-\infty}$ is called {\it bad} for $\widetilde\bx$,
if there exist integers $0 \leq l' \leq r'$, an $(l',
r')$-admissible collection $(\Del', k'_1, k'_2, g', \alp', \beta')$,
a sub-configuration $\widetilde\bx' \subset \widetilde\bx$, a marked
$\CL$-curve $(\overline\Gam', \CV', h', \bpp')$ whose isomorphism
class belongs to ${\cal T}^c(\Del', k'_1, k'_2, g', \alp', \beta',
\widetilde\bx')$, and a vertex $V' \in (\Gam')^0_\infty \backslash
\CV'$ such that
\begin{itemize}
\item ${\widehat h'}(V') = p$,
\item in the case $l' < r'$ and $V' \not\in (\bpp')^\flat$,
there exist integers $0 \leq l'' < r''$, an $(l'', r'')$-admissible
collection $(\Del'', k''_1, k''_2, g'', \alp'', \beta'')$, a
sub-configuration $\widetilde\bx'' \subset \widetilde\bx$ disjoint
from $\widetilde\bx'$, a marked $\CL$-curve $(\overline\Gam'',
\CV'', h'', \bpp'')$ whose isomorphism class belongs to ${\cal
T}^c(\Del'', k''_1, k''_2, g'', \alp'', \beta'', \widetilde\bx'')$,
and a vertex $V'' \in (\Gam'')^0_\infty \backslash (\CV'' \cup
(\bpp'')^\flat)$ such that ${\widehat h''}(V'') = p$.
\end{itemize}
Lemma~\ref{finite} and Proposition~\ref{finiteness} imply that the
set ${\cal B}(\widetilde\bx) \subset L_{-\infty}$ of bad points for
$\widetilde\bx$ is finite.

Choose integers $0 \leq l' \leq r'$, an $(l', r')$-admissible
collection $(\Del', k'_1, k'_2, g', \alp', \bet')$, and a
subconfiguration $\widetilde\bx'$ of~$\widetilde\bx$ of type $(l',
r')$. Let $\lambda \in \Lambda(\Del', k'_1, k'_2, g', \alp',
\beta')$ be a combinatorial type such that
$(ev^\lambda)^{-1}(\widetilde\bx') \ne \varnothing$ (see
section~\ref{moduli-spaces} for notation). If $l' = r'$, then
$\lambda$ is one of the combinatorial types presented in
Figure~\ref{fn2}.

Assume that $l' < r'$. Due to the property~(*) above, the map
$ev^\lambda$ is injective. Thus, there exists a unique isomorphism
class $[(\overline\Gam', {\cal V}', h', \bpp')] \in{\cal T}^c(\Del',
k'_1, k'_2, g', \alp', \beta', \widetilde\bx')$ such that
$(\overline\Gam', \CV', h', \bpp')$ is of combinatorial
type~$\lambda$, and any small variation of $\widetilde\bx'$ uniquely
lifts to a variation of $[(\overline\Gam', {\cal V}', h', \bpp')]$
inside ${\cal M}^\lambda(\Del', k'_1, k'_2, g', \alp', \beta')$.

The map $\widehat h'$ is injective on $(\bpp')^\flat$ by the
definition of marked $\CL$-curves. Let~$V$ be a point in
$(\Gam')^0_\infty \backslash ({\cal V}' \cup (\bpp')^\flat)$, and
$K(V) \subset \overline\Gam'\backslash\bpp'$ be the connected
component containing~$V$. The component $K(V)$ does not contain any
other univalent vertex of $\overline\Gam'$ (see section~\ref{sec4}).
This fact together with the property~(ii) in the definition of
${\cal T}^c(\Del', k'_1, k'_2, g', \alp', \beta', \widetilde\bx')$
implies that the boundary of $K(V)$ contains a point $P \in
(\bpp')^\sharp$. Then, $h'(P) = p \in (\widetilde\bx')^\sharp$.
Let~$E \subset K(V)$ be the edge incident to~$P$. Move slightly the
point~$p$ in the direction orthogonal to $h'(E)$ keeping all the
points of $\bx' \backslash \{p\}$ fixed. Due to \cite[Formula
(4)]{Sh09}, the corresponding variation of $[(\overline\Gam', {\cal
V}', h', \bpp')]$ in ${\cal M}^\lambda(\Del', k'_1, k'_2, g', \alp',
\beta')$ changes the position of ${\widehat h}'(V)$ in
$L_{-\infty}$. This proves that the image of
$(\overline\Gam')^0_\infty \backslash (\CV' \cup (\bpp')^\flat)$
under $\widehat h'$ can be made disjoint from any point in ${\cal
B}(\widetilde\bx)$.

Assume now that ${\widehat h}'(V_1) = {\widehat h}'(V_2)$ for two
distinct points $V_1, V_2 \in (\overline\Gam')^0_\infty \backslash
({\cal V}' \cup (\bpp')^\flat)$. The connected components $K(V_1),
K_2(V_2) \subset \overline\Gam' \backslash \bpp'$ containing $V_1$
and $V_2$, respectively, do not coincide. Since $[(\overline\Gam',
{\cal V}', h', \bpp')] \in{\cal T}^c(\Del', k'_1, k'_2, g', \alp',
\beta', \widetilde\bx')$, the boundary of $K(V_i)$, $i = 1, 2$,
contains a point in $(\bpp')^\sharp$. Suppose that the second
coordinate of ${\widehat h}'(V_1) = {\widehat h}'(V_2)$ does not
exceed the second coordinate of $\widetilde\bx'_{r + 1}$. Among the
points in $(\bpp')^\sharp$ which belong to the boundary of $K(V_1)
\cup K(V_2)$, choose a point $P'$ whose image has the minimal second
coordinate. The point~$P'$ belongs to the boundary of exactly one
component $K(V_i)$, and moving the point $p' = h'(P')$ as above, we
break the equality ${\widehat h}'(V_1) = {\widehat h}'(V_2)$.

Repeating the described procedure for all $\lambda \in
\Lambda(\Del', k'_1, k'_2, g', \alp', \bet')$, we transform
$\widetilde\bx'$ to a certain configuration $\bx' \in {\cal P}(l',
r')$ and modify accordingly the configuration~$\widetilde\bx$. Due
to Lemma~\ref{finite}, we can iterate similar procedures for all
integers $0 \leq l' < r'$, all $(l', r')$-admissible collections
$(\Del', k'_1, k'_2, g', \alp', \beta')$, and all sub-configurations
$\widetilde\bx' \subset \widetilde\bx$ of type $(l', r')$ in order
to transform $\widetilde\bx$ to a $\Theta$-proper configuration~$\bx
\in {\cal P}(l, r)$. \proofend

\begin{lem}\label{new-corollary}
Let $l < r$ be non-negative integers, $(\Del, k_1, k_2, g, \alp,
\beta)$ an $(l, r)$-admissible collection, and $\Theta$ a finite
multi-set of vectors in $\Z^2$ such that~$\Theta$ dominates $\Del$.
Pick a $\Theta$-proper configuration~$\bx$ of type $(l, r)$ and a
class $[(\overline\Gam, \CV, h, \bpp)] \in {\cal T}^c(\Del, k_1,
k_2, g, \alp, \bet, \bx)$. Then, for any vertex~$V$ of
$\overline\Gam$ such that the valency of~$V$ is greater than~$3$,
the following statements hold:
\begin{enumerate}
\item[(1)] all the edges adjacent to~$V$ are of weight~$1$;
\item[(2)] any irreducible component
$(\overline\Gam_{V, j}, {\cal V}_{V, j}, h_{V, j}, \bpp_{V, j})$ of
a marked cut of $(\overline\Gam, {\cal V}, h, \bpp)$ at $V$ has one
of the combinatorial types presented in Figure~\ref{fn2}{\rm (}b,
d{\rm )} provided that $u_{V^a}(E^a) = (1, 0)$ for certain added
vertex $V^a$ of $\overline\Gam_{V, j}$ and the added edge $E^a$
adjacent to $V^a$.
\end{enumerate}
\end{lem}

\begin{proof} If~$E$ is an edge adjacent to~$V$ such that $h(E)$ is
not horizontal, then~$E$ is of weight~$1$ due to the definition of
$\CH_\Theta$-configurations and the second condition in the
definition of $(l, r)$-admissible collections.

Let~$E$ be a horizontal edge adjacent to~$V$. In this case, one has
$u_V(E) = (1, 0)$. Consider the irreducible component
$(\overline\Gam_{V, j}, {\cal V}_{V, j}, h_{V, j}, \bpp_{V, j})$ of
a marked cut of $(\overline\Gam, {\cal V}, h, \bpp)$ at $V$ such
that~$E$ is the predecessor of an added edge $E^a$ of
$\overline\Gam_{V, j}$. Since $\bx$ is $\Theta$-proper, the
combinatorial type of $(\overline\Gam_{V, j}, {\cal V}_{V, j}, h_{V,
j}, \bpp_{V, j})$ belongs to the list presented in Figure~\ref{fn2}.
Moreover, the added vertex $V^a$ adjacent to $E^a$ does not belong
to $(\bpp_{V, j})^\flat$ ({\it cf}. the proof of
Theorem~\ref{CHtrop-for}). Thus, the combinatorial type of
$(\overline\Gam_{V, j}, {\cal V}_{V, j}, h_{V, j}, \bpp_{V, j})$ is
presented in Figure~\ref{fn2}(b, d), and $E$ is of weight~$1$.
\end{proof}

\subsection{From tropical to algebraic}\label{sec6}
Let $l\le r$ be non-negative integers, and $(\Del, k_1, k_2, g,
\alp, \beta)$ an $(l, r)$-admissible collection as defined in
section \ref{CH-config}. Assume that the vectors of $\Del$,
clockwise rotated by $\pi/2$, determine the boundary of a
non-degenerate convex lattice polygon $\Pi$ having one of the shapes
depicted in Figure \ref{fn3} (with slopes $0$, $-1$, or $\infty$).

\begin{figure}
\begin{center}
\epsfxsize 145mm \epsfbox{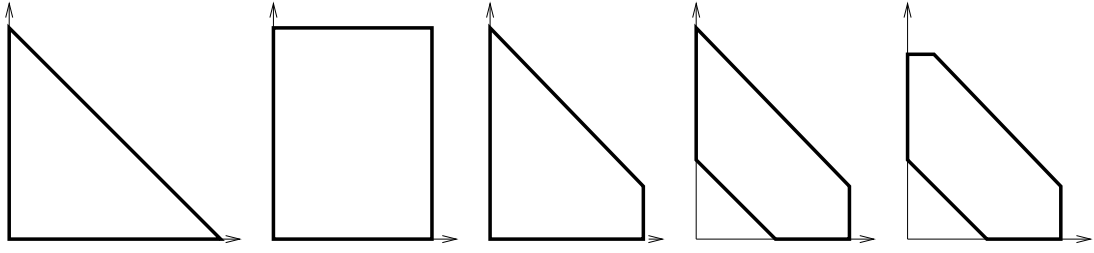}
\end{center}
\caption{Shapes of polygons $\Pi$}\label{fn3}
\end{figure}

The field $\K$ of complex locally convergent Puiseux series is
equipped with the non-Archimedean valuation
$\val(\sum_ia_it^i)=\min\{i\ |\ a_i\ne 0\}$ and the conjugation
involution $\overline{\sum_ia_it^i}=\sum_i\overline{a_i}t^i$.

Let $\Sig'=\Tor_\K(\Pi)$ be the toric surface over $\K$ associated
with $\Pi$, and let \mbox{$\pi:\Sig\to\Sig'$} be the blow-up of
$\Sig'$ at two generic points $z_1,z_2\in(\K^*)^2\subset\Sig'$ such
that $\Val(z_1)=\Val(z_2)$, where $\Val: (\K^*)^2 \to \R^2$ is
defined by $\Val(z^{(1)}, z^{(2)}) = (\val(z^{(1)}),
\val(z^{(2)}))$. Denote by $D'$ the first Chern class of the line
bundle on $\Sig'$, generated by the global sections $z^\omega$,
$\omega\in\Pi\cap\Z^2$, and introduce the divisor class
$D=\pi^*D'-k_1E_1-k_2E_2\in\Pic(\Sig)$, where $E_1$ and $E_2$ are
exceptional divisors of~$\pi$, and $k_1$, $k_2$ are non-negative
integers. Denote also by $E \subset \Sig$ the strict transform of
the toric divisor in $\Sig'$ associated with the left-most vertical
side of $\Pi$ (cf. Figure \ref{fn3}), and denote by ${\breve E}
\subset \Sig$ the union of the strict transforms of the remaining
toric divisors of $\Sig'$. We have a well-defined valuation map
$$\Sig \backslash
{\breve E} \overset{\pi} \longrightarrow \Sig' \backslash
\pi({\breve E}) \overset{-\Val}\longrightarrow\widehat\R^2$$ which
takes $E \backslash {\breve E}$ to $L_{-\infty}$.

We say that the $5$-tuple $(\Sig, D, g, \alp, \beta)$ is {\it
generated} by the collection $(\Del, k_1, k_2, g, \alp, \beta)$.

Let $\Theta$ be a multi-set dominating $\Del$, and $\bx$ a
$\Theta$-generic $\CH_\Theta$-configuration of type $(l, r)$ such
that $\bx_{r+1}=-\Val(z_i)$, $i=1,2$. An ordered configuration
$\bp=(\bp^\flat,\bp^\sharp)$ of $r$ distinct points of $\Sig
\backslash ({\breve E} \cup E_1 \cup E_2)$ is called an {\it
algebraic ${\cal CH}$-configuration over} $\bx$ if
\begin{itemize}
\item the map $(-\Val)\circ\pi$ takes $\bp^\flat\cup\bp^\sharp$ bijectively
onto $\bx^\flat \cup \bx^\sharp$,
\item the configuration $\bp$ is generic among the configurations
satisfying the preceding condition.
\end{itemize}
If $\bp$ is an algebraic $\CH$-configuration over~$\bx$, denote by
$V_\Sig(D, g, \alp, \bet, \bp)$ the subset in $|D|$ represented by
reduced irreducible curves $C\in|D|$ of genus $g$ subject to the
following restrictions:
\begin{itemize}
\item $\bp^\sharp \subset C$,
\item $C$ is nodal and nonsingular along $E$,
\item $C\cap E$
consists of $\|\alp+\bet\|$ distinct points and includes
$\bp^\flat$,
\item the intersection multiplicity $(C\cdot E)_p$
equals $k\ge 1$ for precisely $\alp_k+\bet_k$ points $p\in C\cap E$,
\item $(C\cdot E)_{p_i} = \mt(\alp, i)$.
\end{itemize}

The set $V_\Sig(D, g, \alp, \beta, \bp)$ is finite and the number of
its elements is equal to $N_\Sig(D, g, \alp, \beta)$ (see Lemma
\ref{ln1} and Section~\ref{general} for a discussion of the
corresponding complex enumerative problem and the definition of
$N_\Sig(D, g, \alp, \beta)$).

\begin{proposition}\label{tn11-new}
Let $(\Del, k_1, k_2, g, \alp, \beta)$, $\bx$, $\Sig$, and $D$ be as
above. Then, for any algebraic $\CH$-configuration~$\bp$ over~$\bx$,
there exists a multi-valued map $PW: {\cal T}^c(\Del, k_1, k_2, g,
\alp, \bet, \bx) \to V_\Sig(D, g, \alp, \bet, \bp)$ such that
\begin{itemize}
\item for each $T \in {\cal T}^c(\Del, k_1, k_2, g, \alp, \bet, \bx)$,
the set $PW(T)$ consists of $N(T)$ elements,
\item the sets $PW(T_1)$ and $PW(T_2)$ are disjoint
if $T_1 \ne T_2$.
\end{itemize}
In particular, for any algebraic $\CH$-configuration~$\bp$
over~$\bx$, one has
$${\cal N}(\Del, k_1, k_2, g, \alp, \beta)
\leq N_\Sig(D, g, \alp, \beta).$$
\end{proposition}

\begin{proof} Let~$\bp$ be an algebraic $\CH$-configuration
over~$\bx$, and~$T$ an element of $\qquad\qquad$ ${\cal T}^c(\Del,
k_1, k_2, g, \alp, \bet, \bx)$. Then, the set $PW(T)$ is constructed
using the patchworking procedure described in~\cite{Sh09}. Notice
that~$T$ and~${\widehat\bp} = \pi(\bp) \cup \{z_1, z_2\}$ satisfy
the hypotheses of \cite[Theorem 2]{Sh09}. Indeed,
\begin{itemize}
\item
the condition (T1) in \cite[Section 2.6]{Sh09} immediately follows
from the definition of marked pseudo-simple ${\cal L}$-curves;
\item the condition~(T2) in~\cite[Section 3.1]{Sh09}
follows from the definition of the set $\qquad$ ${\cal T}^c(\Del,
k_1, k_2, g, \alp, \beta, \bx)$;
\item the $\Pi$-genericity of
the configuration $\bx$ (condition (T3) in \cite[Section 2.6]{Sh09})
follows from the definition of $\Theta$-generic ${\cal
CH}_\Theta$-configurations;
\item the condition~(T4) in~\cite[Section 3.1]{Sh09}
follows from the statement~(1) in Lemma~\ref{new-corollary};
\item the condition~(T5) in~\cite[Section 3.1]{Sh09}
follows from the statement~(2) in Lemma~\ref{new-corollary};
\item the condition~(T6) in \cite[Section 3.1]{Sh09}
follows from Lemma~\ref{new-corollary} and the condition (i) in the
definition of the set ${\cal T}^c(\Del, k_1, k_2, g, \alp, \bet,
\bx)$;
\item the condition (T7) in \cite[Section 3.1]{Sh09} follows
from the fact that, for any class $T=[(\overline\Gam, {\cal V}, h,
\bpp)] \in {\cal T}^c(\Del, k_1, k_2, g, \alp, \bet, \bx)$, the
first coordinate of $h(V)$ for any vertex $V$ of $\Gam$ does not
exceed the first coordinate of $\bx_{r+1}$ (cf. section
\ref{sec42}).
\end{itemize}
In addition, ${\widehat\bp} = \pi(\bp) \cup \{z_1, z_2\}$ matches
restrictions (A1)-(A4) in \cite[Section 3.1]{Sh09} by construction,
and meets condition (A5) in \cite[Section 3.1]{Sh09} by \cite[Lemma
8]{Sh09}.

Thus, the required statement follows from \cite[Theorem 2]{Sh09}.
\end{proof}

\subsection{Correspondence}\label{adaptation}
To compare formulas (\ref{en21}) and (\ref{en8}), we construct a map
$\Phi:A^{tr}(\Sig,E)\to{\cal S}$. Any element $(D,g,\alp,\bet)\in
A^{tr}(\Sig,E)$ satisfies \mbox{$D=dL-k_1E_1-...-k_5E_5$} with
$k_1,...,k_5\ge 0$ and
\begin{equation}
k_3 + k_4 < d, \quad \max_{3\le i<j\le 5}(k_i + k_j) \leq d, \quad g
\leq \frac{(d - 1)(d - 2)}{2} - \sum_{i = 1}^5\frac{k_i(k_i -
1)}{2}, \label{en43}
\end{equation}
due to $DE>0$, $D(L-E_i-E_j)\ge 0$ ($i, j = 3, 4, 5$), and
$g\le\frac{D^2+DK_\Sig}{2}+1$. These data give rise to the numbers
$l = \|\alp\|$ and $r = R_\Sig(D, g,  \beta) + l = -D(E + K_\Sigma)
+ \|\beta\| + g - 1$, and the multi-set of vectors~$\Del$ formed by
\begin{itemize}
\item $k_3$ vectors $(0, 1)$,
\item $d - k_4 - k_5$ vectors $(0, -1)$,
\item $d - k_3 - k_5$ vectors $(1, 1)$,
\item $k_4$ vectors $(-1, -1)$,
\item $k_5$ vectors $(1, 0)$,
\item $d - k_3 - k_4$ vectors $(-1, 0)$.
\end{itemize}

\begin{figure}
\setlength{\unitlength}{1cm}
\begin{picture}(5,5.5)(-4,0)
\thinlines\put(1,1){\vector(1,0){5}}\put(1,1){\vector(0,1){4.5}}
\dashline{0.2}(1,5)(3,3)\dashline{0.2}(4,2)(5,1) \thicklines
\put(2,1){\line(-1,1){1}}\put(1,2){\line(0,1){1}}
\put(1,3){\line(1,0){2}}\put(3,3){\line(1,-1){1}}\put(2,1){\line(1,0){
2}} \put(4,1){\line(0,1){1}}
\put(1.8,0.5){$k_4$}\put(3.5,0.5){$d-k_5$}\put(5.1,1.1){$d$}\put(0.5,1.9){$k_4$}
\put(-0.2,2.9){$d - k_3$}\put(0.7,4.9){$d$}
\end{picture}
\caption{Polygon $\Pi$}\label{fn7}
\end{figure}
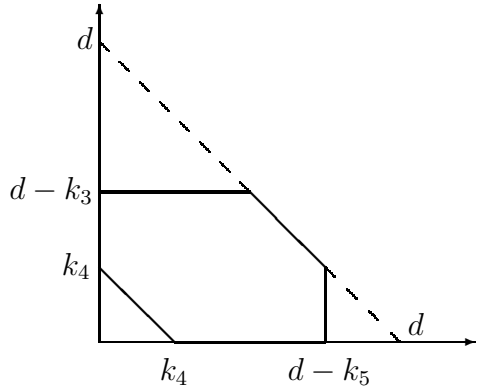

The collection $(\Del, k_1, k_2, g, \alp, \beta)$ is $(l,
r)$-admissible, and $(\Sig, D, g, \alp, \beta)$ is generated by
$(\Del, k_1, k_2, g, \alp, \beta)$; {\it cf}. Section~\ref{sec6}
(the polygon $\Pi$ obtained from $\Del$ is shown in
Figure~\ref{fn7}). Put $\Phi(D, g, \alp, \bet) = (\Del, k_1, k_2, g,
\alp, \bet)$.

\begin{proposition}\label{ln4}
The map $\Phi:A^{tr}(\Sig,E)\to{\cal S}$ is an injective
homomorphism of semi-groups which establishes a bijection between
the initial conditions of (\ref{en21}) {\rm (}which are listed in
Proposition \ref{initial-conditions}{\rm )} and the initial
conditions of (\ref{en8}) {\rm (}which are indicated in Proposition
\ref{tn3}(2){\rm )}. In particular,
\begin{equation}
N_\Sig(D, g, \alp, \bet) \leq {\cal N}(\Del, k_1, k_2, g, \alp,
\bet)\ ,\label{en19}
\end{equation}
for any $(D, g, \alp, \bet) \in A^{tr}(\Sig,E)$ and $(\Del, k_1,
k_2, g, \alp, \bet) = \Phi(D, g, \alp, \bet) \in {\cal S}$.
\end{proposition}

\begin{proof} The fact that $\Phi$ is an injective homomorphism is
straightforward. Furthermore, $\Phi$ allows one to identify each
splitting~(\ref{en9}) in the recursive formula~(\ref{en8}) with a
splitting~(\ref{en25}) in the recursive formula~(\ref{en21}). Thus,
the inequalities $N_\Sig(D, g, \alp, \bet) \leq {\cal N}(\Del, k_1,
k_2, g, \alp, \bet)$ follow from non-negativity of the coefficients
of the aforementioned recursive formulas. \end{proof}

The following statement is an immediate consequence of
Propositions~\ref{tn11-new} and~\ref{ln4}.

\begin{thm}\label{tn11}
Let $l \leq r$ be non-negative integers, $(\Del, k_1, k_2, g, \alp,
\beta)$ an $(l, r)$-admissible collection, and $\Theta$ a finite
multi-set of vectors in $\Z^2$ such that~$\Theta$ dominates $\Del$.
Pick a $\Theta$-proper configuration $\bx \in {\cal P}(l, r)$ and an
algebraic $\CH$-configuration~$\bp$ over~$\bx$. Then,
\begin{equation}
V_\Sig(D, g, \alp, \bet, \bp)= \coprod_{T\in{\cal T}^c(\Del, k_1,
k_2, g, \alp, \bet, \bx)} PW(T)\ ,\label{en18}
\end{equation}
where $(\Sig, D, g, \alp, \beta)$ is generated by $(\Del, k_1, k_2,
g, \alp, \beta)$. \proofend
\end{thm}

\section{Recursive formulas for Welschinger invariants}\label{Wel}

From now on, we switch back the ground field to~$\C$.

\subsection{Welschinger invariants}\label{Welsch-inv}
Let $\Sig$ be a real {\it unnodal} ({\it i.e.}, not containing any
rational $(-n)$-curve, $n\ge 2$) Del Pezzo surface with a connected
real part $\R\Sigma$, and let $D\subset\Sigma$ be a real effective
divisor. Consider a generic set~$\bp$ of $c_1(\Sig)\cdot D-1$ real
points of $\Sig$. The set ${\cal R}(\Sig, D, \bp)$ of real rational
curves $C\in|D|$ passing through the points of~$\bp$ is finite, and
all these curves are nodal and irreducible. Due to the Welschinger
theorem~\cite{W1} (and the genericity of the complex structure on
$\Sig$), the number
$$
W(\Sig,D, \bp)=\sum_{C\in{\cal R}(\Sig,D, \bp)}(-1)^{s(C)}\ ,
$$
where $s(C)$ is the number of {\it solitary nodes} of $C$ ({\it
i.e.}, real points, where a local equation of the curve can be
written over $\R$ in the form $x^2+y^2=0$), does not depend on the
choice of a generic set~$\bp$. We denote this Welschinger invariant
by $W(\Sig,D)$.

\subsection{Welschinger numbers and recursive formula}\label{sec5}
Put $\Sig=\PP^2_{q,s}$, where $0\le q+2s\le 5$, $s\le 1$; see
Introduction. Denote by~$L$ the pull-back of a line in $\PP^2$, and
by $E_1,...,E_{q+2s}$ the exceptional curves of the blow up. In the
case $s=1$, assume that $E_1$ and $E_2$ are conjugate imaginary. Fix
a smooth real rational curve $E$ linearly equivalent to (cf. section
\ref{initial})
\begin{itemize}\item $L$ for $q+2s\le2$, \item $L-E_3$
for $q+2s=3$, \item $L-E_3-E_4$ for $q+2s\ge 4$.
\end{itemize}

Denote by $\Pic^{re}_+(\Sig,E)$ the subset of $\Pic_+(\Sig, E)$
formed by the divisors representable by a real reduced irreducible
over $\C$ curve.

\begin{lem}\label{l6n2}
The set $\Pic^{re}_+(\Sig,E)$ consists of the following divisors:
\begin{itemize}
\item
the divisors represented by real $(-1)$-curves which cross~$E$,
\item
the divisors $D=dL-k_1E_1-...-k_qE_q$ satisfying
$$
\displaylines{ DE>0, \quad d > 0, \quad g(\Sig,D) \ge 0, \cr
k_1,...,k_q\ge 0,\quad \max_{1\le i<j\le q}(k_i+k_j)\le d,\quad k_1+
\ldots +k_q\le 2d }$$ in the case $s = 0$, and the divisors
$D=dL-k_1E_1-...-k_{q+2}E_{q+2}$ satisfying
$$
\displaylines{ DE>0, \quad d > 0, \quad g(\Sig,D) \ge 0, \cr k_1,
\ldots,k_{q+2}\ge 0,\quad k_1 = k_2, \quad \max_{1\le i<j \le
q+2}(k_i+k_j)\le d,\quad k_1+ \ldots +k_{q+2}\le 2d }
$$
in the case $s = 1$.
\end{itemize}
\end{lem}

\begin{proof} The inequalities on the numbers $k_i$ follow from the
B\'ezout theorem. Real reduced irreducible curves representing the
described divisors can be found, for example, in~\cite{GLS}.
\end{proof}

Introduce the set $A^{re}_0(\Sig,E)$ which consists of the triples
$$(D,\alp,\bet)\in\Pic_+^{re}(\Sig,E)\times(\Z_+^\infty)^{\text{\rm
odd}}\times(\Z_+^\infty)^{\text{\rm odd}}
$$
such that $I\alp+I\bet=DE$ and $R_\Sig(D, 0, \bet) \geq 0$. In the
case $\Sig=\PP^2_{3,1}$, include in $A^{re}_0(\Sig, E)$ also the
element $(2L-E_1-E_2-2E_5,0,2\theta_1)$.

The following statement is straightforward.

\begin{lem}\label{l6n3} (1)
For each element $(D,\alp,\bet)\in A^{re}_0(\Sig,E)$ different from
$$
(2L-E_1-E_2-2E_5,0,2\theta_1), \quad (E_3,0,\theta_1), \quad
(E_4,0,\theta_1),
$$
the quadruple $(D,0,\alp,\bet)$ belongs to
$A^{tr}(\Sig,E)\cap\Phi^{-1}({\cal S}^*)$, where ${\cal S}^*$ stands
for ${\cal S}^{\text{\rm odd}}$ {\rm (}respectively, ${\cal
S}^{\text{odd,sym}}${\rm )} if $s=0$ {\rm (}respectively, $s =
1${\rm )}.

(2) If $s = 0$, each triple $(D,\alp,\bet)$ presented in Figure
\ref{fn2} and relevant to the surface $\Sig$ belongs to
$A^{re}_0(\Sig,E)$. If $s = 1$, each triple $(D,\alp,\bet)$
presented in Figure \ref{fn2}(a,d-h) and relevant to the surface
$\Sig$ belongs to $A_0^{re}(\Sig,E)$. \proofend
\end{lem}

Consider the function $W_\Sig: A^{re}_0(\Sig,E)\to\Z$ defined as
follows:
\begin{itemize}
\item $W_\Sig(E_3,0,\theta_1)=1$ if $q+2s\ge 3$,
\item $W_\Sig(E_4,0,\theta_1)=1$ if
$q+2s\ge 4$,
\item $W_\Sig(2L-E_1-E_2-2E_5,0,2\theta_1) = 1$
if $\Sig = \PP^2_{3, 1}$;
\item for any $(D,\alp,\bet)\in A^{re}_0(\Sig,E)$
different from those mentioned above, put $W_\Sig(D, \alp, \bet) =
{\cal W}^*(\Del, k_1, k_2, 0, \alp, \bet)$, where $(\Del, k_1, k_2,
0, \alp, \bet) = \Phi(D, 0, \alp, \bet) \in {\cal S}^*$ (see
Lemma~\ref{l6n3}) and ${\cal W}^*$ stands for ${\cal W}$ or ${\cal
W}^{\text{\rm sym}}$ according to $s=0$ or $1$.
\end{itemize}

\begin{thm}\label{tn12} (1) For any
divisor $D\in\Pic^{re}_+(\Sig,E)$, one has
\begin{equation}W(\Sig,D)=W_\Sig(D,0,(DE)\theta_1)\ .\label{en47}\end{equation}

(2) For any element $(D,\alp,\bet)\in A^{re}_0(\Sig,E)$ such that
$D\in\Pic_+^{re}(\Sig,E)$ and $R_\Sig(D, 0,  \bet) > 0$, the
following formula holds
$$W_\Sig(D,\alp,\bet)=\sum_{j\ge 1,\ \bet_j>0}W_\Sig(D,\alp+\theta_j,\bet-\theta_j)$$
\begin{equation}+\sum
\left(\begin{matrix}\alp\\
\alp^{(1)}...\alp^{(m)}\end{matrix}\right)
\frac{(n-1)!}{n_1!...n_m!}\prod_{i=1}^m\left(\left(
\begin{matrix}\bet^{(i)}\\
\widetilde\bet^{(i)}\end{matrix}\right)W_\Sig(D^{(i)},\alp^{(i)}
,\bet^{(i)})\right),\label{en60}\end{equation} where $n = R_\Sig(D,
0,  \bet)$, the number $n_i$ is equal to
\begin{equation}
\begin{cases}
0 \; & if  \; \Sig=\PP^2_{3,1}, D^{(i)}=2L-E_1-E_2-2E_5, \cr & \quad
\alp^{(i)}=0, \bet^{(i)}=(2), \cr R_\Sig(D^{(i)}, 0,  \bet^{(i)}) \;
& otherwise,
\end{cases}
\end{equation}
and the second sum in~(\ref{en60}) is taken
\begin{itemize}
\item over all
sequences
\begin{equation}(D^{(1)},\alp^{(1)},\bet^{(1)}),\ ...,\ (D^{(m)},\alp^{(m)},
\bet^{(m)})\ ,\label{en70}
\end{equation}
of elements of $A^{re}_0(\Sig,E)$ such that
\begin{enumerate}
\item[(a)] $\sum_{i=1}^mD^{(i)}=D-E$,
\item[(b)] $\alp'\le\alp$ and $\bet\le\bet'$,
where $\alp' = \sum_{i=1}^m\alp^{(i)}$ and $\beta' =
\sum_{i=1}^m\bet^{(i)}$,
\item[(c)] each
triple $(D^{(i)},\alp^{(i)},\bet^{(i)})$ with $n_i=0$ appears in
(\ref{en70}) at most once,
\end{enumerate}
\item over all splittings in $(\Z_+)^{\text{\rm odd}}$
\begin{equation}
\bet'=\bet+\sum_{i=1}^m\widetilde\bet^{(i)} \ ,\label{en99}
\end{equation}
satisfying the restriction $\bet^{(i)}\ge\widetilde\bet^{(i)}$ and
$$\|\widetilde\bet^{(i)}\|=\begin{cases}2,\quad&\text{if}\quad
\Sig=\PP^2_{3,1},\ D^{(i)}=2L-E_1-E_2-2E_5,\\ &\quad \alp^{(i)} =0,\
\bet^{(i)}=(2),\\ 1,\quad&\text{otherwise}
                           \end{cases}$$
for all $i=1,...,m$,
\end{itemize}
and the second sum in (\ref{en60}) is factorized by simultaneous
permutations in the sequence (\ref{en70}) and in the splitting
(\ref{en99}).

(3) All the numbers $W_\Sig(D,\alp,\bet)$, where $(D,\alp,\bet)\in
A^{re}_0(\Sig,E)$, $D\in\Pic^{re}_+(\Sig,E)$, and $R_\Sig(D, 0,
\bet) > 0$, are recursively determined by the formula {\rm
(}\ref{en60}{\rm )}, the value
$W_{\PP^2_{3,1}}(2L-E_1-E_2-2E_5,0,(2))=1$ and the values
$W_\Sig(D,\alp,\bet)$ for the elements $(D,\alp,\bet)\in
A^{re}_0(\Sig,E)$ with $R_\Sig(D, 0,  \bet) = 0$. The latter initial
values are equal to $1$ in the cases listed in Figure~\ref{fn2} for
$s=0$, the cases listed in Figure~\ref{fn2}(a,d-h) for $s=1$, and
vanish in all the remaining cases.
\end{thm}

{\bf Proof.} To prove the first statement of the theorem, put
$(\Del, k_1, k_2, 0, 0, (DE)\theta_1) = \Phi(D, 0, 0, (DE)\theta_1)$
and
$$
{\cal T}^*(\Del, k_1, k_2, 0, 0, (DE)\theta_1, \bx) =
\begin{cases}
{\cal T}^c(\Del, k_1, k_2, 0, 0, (DE)\theta_1, \bx)
\quad & \text{if} \; s = 0, \\
{\cal T}^{c, \text{sym}}(\Del, k_1, k_2, 0, 0, (DE)\theta_1, \bx)
\quad & \text{if} \; s = 1,
\end{cases}
$$
where $\bx$ is an appropriate configuration of points. Formula
(\ref{en47}) follows from Theorem \ref{tn11} and~\cite[Theorem
3]{Sh09}. The latter theorem states that, for any
$\qquad\qquad\qquad$ $T \in {\cal T}^*(\Del, k_1, k_2, 0, 0,
(DE)\theta_1, \bx)$ and the set $PW(T)$ constructed in~\cite{Sh09}
for an algebraic $\CH$-configuration~$\bp$ over~$\bx$ ({\it cf}.
Proposition~\ref{tn11-new}), the sum of the Welschinger signs
$(-1)^{s(C)}$ of the real rational curves~$C$ in $PW(T)$ is equal to
the Welschinger multiplicity $W(T)$ of~$T$. (The number $W(T)$
appears in \cite[Section 2.6]{Sh09} under the name of {\it real
weight}; our definition of $W(T)$ is a specialization of the
definition of the real weight given there.)

To prove the second statement of the theorem, it is sufficient to
establish formula (\ref{en60}) for $\Sig=\PP^2_{5,0}$ or
$\PP^2_{3,1}$ ({\it cf}. proof of Theorem \ref{tn1}). Observe that
by Lemma \ref{l6n3}, one has $(D, 0, \alp, \bet) \in A^{tr}(\Sig,
E)$ and the collection $\Phi(D, 0, \alp, \bet) = (\Del, k_1, k_2, 0,
\alp, \bet)$ belongs to ${\cal S}^*$. This collection is
$(l,r)$-admissible, where $l=\|\alp\|$ and $r = l + R_\Sig(D, 0,
\bet) > l$. Applying formula (\ref{en26}) to ${\cal W}^*(\Del, k_1,
k_2, 0, \alp, \bet) = W_\Sig(D,\alp,\bet)$ in the left-hand side
(${\cal W}^*$ stands for ${\cal W}$ or ${\cal W}^{\text{\rm sym}}$
according to $s=0$ or $1$), we intend to equate the right-hand sides
of (\ref{en26}) and (\ref{en60}). Clearly, the first sum in the
right-hand side of (\ref{en26}) coincides with the first sum in the
right-hand side of (\ref{en60}). So, it remains to compare the
second sums in the right-hand side of (\ref{en26}) and (\ref{en60}).
In view of Remark \ref{rem-sym}, a non-zero value ${\cal
W}^*(\Del^{(i)}, k_1^{(i)}, k_2^{(i)}, g^{(i)}, \alp^{(i)},
\bet^{(i)})$ with $g^{(i)}<0$ is possible only if $\Sig=\PP^2_{3,1}$
and $(\Del^{(i)}, k_1^{(i)}, k_2^{(i)}, g^{(i)}, \alp^{(i)},
\bet^{(i)}) = K^{sp}$. Hence, the restrictions
$$g'=\sum_{i=1}^mg^{(i)}-m+1=-\sum_{i=1}^m|\widetilde\bet^{(i)}|
+1\quad\text{and} \quad\|\widetilde\bet^{(i)}\|>0,\ i=1,...,m\ ,$$
for the second sum in the right-hand side of (\ref{en26}) yield that
each factor $\qquad\qquad\qquad$ ${\cal W}^*(\Del^{(i)}, k^{(i)}_1,
k^{(i)}_2, g^{(i)}, \alp^{(i)}, \beta^{(i)})$ in a non-zero summand
either coincides with ${\cal W}^*(K^{sp})$ (in this case
$g^{(i)}=-1$ and $\|\bet^{(i)}\|=2$; such a situation can occur only
if $\Sig=\PP^2_{3,1}$), or satisfy $g^{(i)}=0$ and
$\|\widetilde\bet^{(i)}\|=1$. Moreover, by Theorem \ref{tn11} and
\cite[Theorem 3]{Sh09}, if ${\cal W}^*(\Del^{(i)}, k^{(i)}_1,
k^{(i)}_2, 0, \alp^{(i)}, \bet^{(i)}) \ne 0$ then
$$(\Del^{(i)},
k^{(i)}_1, k^{(i)}_2, 0, \alp^{(i)}, \bet^{(i)},) = \Phi(\hat D, 0,
\hat\alp, \hat\bet)$$ for some $(\hat D, \hat\alp, \hat\bet) \in
A^{re}_0(\Sig,E)$, $\hat D\in\Pic_+^{re}(\Sig,E)$. We complete the
comparison of formulas~(\ref{en26}) and~(\ref{en60}) noticing that
the summation over $\daleth$ in~(\ref{en26}) is equivalent to the
subdivision of the second sum of the right-hand side of~(\ref{en60})
in four sums according to the presence of sequences $(E_3, 0,
\theta_1)$ and $(E_4, 0, \theta_1)$ in~(\ref{en70}) ({\it cf}.
Lemma~\ref{l6n3}(1) and the proof of Proposition~\ref{tn3})).

The last statement of the theorem immediately follows from
Proposition \ref{initial-conditionsr} and Lemma \ref{l6n3}(2).
\proofend

Theorem~\ref{tn12} gives a possibility to calculate Welschinger
invariants of the surfaces $\PP^2_{q,s}$ with $1\le q\le5$, $0\le
s\le1$, $d+2s\le5$, and their blow-downs. Here are some of the
values.

\begin{itemize}
\item The case $\Sig = \PP^2_{5,0}$, $D = -K$ or $D = -2K$, where
$-K=3L-E_1-...-E_5$:
$$
W(\Sig, -K) = 8, \;\;\; W(\Sig, -2K) = 4160.
$$
\item The case $\Sig = \PP^2_{3,1}$, $D = -K$ or $D = -2K$, where
$-K=3L-E_1-...-E_5$:
$$
W(\Sig, -K) = 6, \;\;\; W(\Sig, -2K) = 2004.
$$
\item The case $\Sig = \PP^2_{4,0}$, $D = -K$ or $D = -2K$, where
$-K=3L-E_1-...-E_4$:
$$
W(\Sig, -K) = 8, \;\;\; W(\Sig, -2K) = 16312.
$$
\item The case $\Sig = \PP^2_{2,1}$, $D = -K$ or $D = -2K$, where
$-K=3L-E_1-...-E_4$:
$$
W(\Sig, -K) = 6, \;\;\; W(\Sig, -2K) = 7222.
$$
\item The case $\Sig = \PP^2_{3,0}$:
$$
\displaylines{ W(\Sig, L) = 1, \;\;\; W(\Sig, 2L) = 1, \;\;\;
W(\Sig, 3L)=8, \;\;\; \cr W(\Sig, 4L) = 240, \;\;\; W(\Sig, 4L -
2E_1) = 48, \;\;\; W(\Sig, 5L - 2E_1 - 2E_2) = 798. }
$$
\item The case $\Sig = \PP^2_{1,1}$
(the divisors $E_1$, $E_2$ are imaginary, and the divisor $E_3$ is
real):
$$
\displaylines{ W(\Sig, L) = 1, \;\;\; W(\Sig, 2L) = 1, \;\;\;
W(\Sig, 3L) = 6, \;\;\; \cr W(\Sig, 4L) = 144, \;\;\; W(\Sig, 4L -
2E_3) = 32, \;\;\; W(5L - 2E_1 - 2E_2) = 432. }
$$
\item The case $\Sig = \PP^1 \times \PP^1$,
the real structure on $\Sig$ is standard, that is, given by $(z, w)
\mapsto (\overline z, \overline w)$, and $D$ is a curve of bi-degree
$(1, 1)$:
$$
W(\Sig, D) = 1, \;\;\; W(\Sig, 2D) = 8, \;\;\; W(\Sig, 3D) = 798.
$$
\item The case $\Sig = (\PP^1 \times \PP^1)_{0, 1}$
(we denote by $(\PP^1 \times \PP^1)_{q, s}$ the real surface
considered in the previous item and blown up at a generic collection
of $q$ real points and $s$ pairs of conjugate imaginary points):
$$
W(\Sig, D - E_1 - E_2) = 1, \;\;\; W(2D - E_1 - E_2) = 6, \;\;\;
W(3D - 2E_1 - 2E_2) = 432.
$$
\item The case $\Sig = \PP^1 \times \PP^1$,
the real structure on $\Sig$ is given by $(z, w) \mapsto (\overline
w, \overline z)$, and $D$ is a curve of bi-degree $(1, 1)$:
$$
W(\Sig, D) = 1, \;\;\; W(\Sig, 2D) = 6, \;\;\; W(\Sig, 3D) = 432.
$$
\item The case $\Sig = \SSS_{1,0}$ (we denote by $\SSS_{q, s}$
the real surface considered in the previous item and blown up at a
generic collection of $q$ real points and $s$ pairs of conjugate
imaginary points):
$$
\displaylines{ W(\Sig, D - kE) = 1, \; k = 0, 1, \cr W(\Sig, 2D -
kE) = 6, \; k = 0, 1, \cr W(\Sig, 2D - 2E) = 1, \cr W(\Sig, 3D - kE)
= 432, \; k = 0, 1, \cr W(\Sig, 3D - 2E) = 144, \;\;\; W(\Sig, 3D -
3E) = 6. }
$$
\item The case $\Sig = \SSS_{2,0}$:
$$
W(\Sig, D - E_1 - E_2) = 1,\;\;\; W(\Sig, 2D - E_1 - E_2) = 6,
\;\;\; W(3D - 3E_1 - 2E_2) = 32.
$$
\end{itemize}

\section{Properties of Welschinger invariants}\label{properties}

\subsection{Positivity and asymptotics}\label{sec-pa}

As is usual, we call a divisor $D$ on a surface $\Sig$ {\it nef} if
$D$ non-negatively intersects any algebraic curve on $\Sig$. When
$\Sig$ is an unnodal Del Pezzo surface, $D$ is nef if and only if
its intersection with any $(-1)$-curve is non-negative. A nef
divisor $D$ is called {\it big} if $D^2>0$.

\begin{thm}\label{tn5}
Let $\Sig=(\PP^1)^2_{0,1}$ or $\PP^2_{q,s}$, $4\le q+2s\le 5$, $s\le
1$. Then, for any real nef and big divisor $D$ on $\Sig$, the
invariant $W(\Sig,D)$ is positive, and the following asymptotic
relation holds:
\begin{equation}\
\log W(\Sig,nD)=(-DK_\Sig) n\log n+O(n),\quad n\to \infty\
.\label{en12}
\end{equation}
In particular,
$$\lim_{n\to\infty}\frac{\log W(\Sig,nD)}{\log GW_0(\Sig,nD)}=1\ .$$
\end{thm}

\begin{rem}\label{rn2}
(1) The positivity and asymptotic behavior as in Theorem \ref{tn5}
were established before for all real toric unnodal Del Pezzo
surfaces with a nonempty real part, except for $(\PP^1)^2_{0,1}$
{\rm (}see \cite{IKS,IKS2,IKS4,Sh2}{\rm )}.

(2) If $D$ is not nef or not big, then $W(\Sig,D)=1$ or $0$
depending on whether the linear system $|D|$ contains an irreducible
curve or not (for the existence of rational irreducible
representatives see, for instance, \cite{GLS}).

(3) The Gromov-Witten and Welschinger invariants do not depend on
variation of tamed almost complex structures; hence Theorem
\ref{tn5} is valid for blow-ups at arbitrary (not necessarily
generic) configurations of points and for any homology class $D\in
H_2(\Sig)$ with $D^2>0$ which non-negatively intersects each class
$e\in H_2(\Sig)$ such that $eK_\Sig=e^2=-1$.
\end{rem}

{\bf Proof of Theorem~\ref{tn5}}. Since the coefficients in the
recursive formula (\ref{en60}) are positive, and its initial values
are nonnegative (see Theorem \ref{tn12}(3)), to prove the positivity
it is enough to find at least one tropical curve which matches a
given configuration of fixed points and has a positive Welschinger
multiplicity. Due to the upper bound $W(\Sig,D)\le GW_0(\Sig,D)$ and
the asymptotics \mbox{$\log GW_0(\Sig,nD)=(-DK_\Sig)n\log n+O(n)$}
(see~\cite{IKS5}), to prove the asymptotic relation~(\ref{en12}) it
is enough to find tropical curves with sufficiently large total
Welschinger multiplicity.

As in the proof of Theorem \ref{tn1}, we may consider only two
cases, $\Sig=\PP^2_{5,0}$ or $\PP^2_{3,1}$.

\begin{lem}\label{cremona}
If $\Sig=\PP^2_{5,0}$ and $D$ is a nef and big divisor on $\Sig$,
then there exists a collection $L,E_1,...,E_5$ of disjoint real
smooth rational curves on $\Sig$ such that $L^2=-E_1^2 = \ldots =
-E_5^2=1$, $D=dL-k_1E_1- \ldots -k_5E_5$, and
\begin{equation}
k_3\ge k_5\ge k_1\ge k_2\ge k_4\ge 0,\quad d\ge k_1+k_3+k_5\
.\label{en13}
\end{equation}
If $\Sig=\PP^2_{3,1}$ and $D$ is a real nef and big divisor on
$\Sig$, then there exist a collection $L,E_3,E_4,E_5$ of real smooth
rational curves and a pair $E_1,E_2$ of conjugate imaginary smooth
rational curves on $\Sig$ which all are pairwise disjoint and such
that $L^2=-E_1^2= \ldots =-E_5^2=1$, $D=dL-k_1E_1-...-k_5E_5$, and
\begin{equation}
k_3\ge k_5\ge k_4\ge 0,\quad k_1=k_2\ge0,\quad
d\ge\max\{k_3+k_4+k_5,k_1+k_2+k_3\}\ . \label{en14}\end{equation}
\end{lem}

\begin{proof} (cf. \cite{Hir} or \cite[Section 5]{GLS}) The last
inequalities in (\ref{en13}) and (\ref{en14}) can be achieved by
means of finitely many basis changes (standard quadratic Cremona
transformations)
$$(L,E_i,E_j,E_\ell)\mapsto(2L-E_i-E_j-E_\ell,L-E_j-E_\ell,L-E_i-
E_\ell,L-E_i-E_j)\ ,$$ where $E_i,E_j,E_\ell$ either are real, or
two of them are conjugate imaginary. Such a transformation
diminishes $d=DL$ as soon as $d<k_i+k_j+k_\ell$.\end{proof}

In addition to (\ref{en13}) and (\ref{en14}), we may suppose that
$k_2>0$, since otherwise we can blow down two exceptional curves and
thus deduce the theorem from \cite{IKS2}.

Let $E$ be the real smooth rational $(-1)$-curve linearly equivalent
to $L-E_3-E_4$. Define a multi-set $\Del$ of vectors in $\Z^2$ by
relation
$$(\Del, k_1, k_2, 0, 0, (DE)\theta_1) = \Phi(D, 0, 0, (DE)\theta_1)\ .$$
Pick a sequence of finite multi-sets $\Theta_n$ dominating $n\Del$,
$n\ge 1$. Put $p=(0,0)$ and, for any $n\ge 1$, pick a configuration
$\bx_n$ of $r_n=-nDK_\Sig-1$ points in the strip $\{x < 0, \ 0 < y <
\eps\}$ such that $(\bx_n, p)$ is a $\Theta_n$-generic ${\cal
CH}_{\Theta_n}$-configuration of type $(0,r_n)$ and vertical
parameter $\eps$. Denote by $\del_0, \ldots, \del_{r_n + 1}$
horizontal parameters of $(\bx_n, p)$.

{\bf (1)} Assume that $k_5\ge k_1$. Consider the divisor $$D'=
(d-k_2)L-(k_3-k_2)E_3-k_4E_4-(k_1-k_2+k_5)E_5$$ on the real toric
surface $\Sig'=\PP^2_{3,0}$. The divisor $D'$ is nef in view of
(\ref{en13}), (\ref{en14}) and the relations
$$D'E_i\ge0,\ i=3,4,5,\quad D'(L-E_3-E_4)=d-k_3-k_4\ge0\ ,$$
$$D'(L-E_4-E_5)=
d-k_1-k_4-k_5\ge0\ ,$$ $$D'(L-E_3-E_5)=d-k_1+k_2-k_3-k_5\ge0$$ (as
is well-known, the divisors $E_i,L-E_j-E_\ell$, $i,j,\ell=3,4,5$,
$j\ne \ell$, generate the effective cone of $\Sig'$). The linear
system $|D'|$ is naturally associated with the convex lattice
polygon $\Pi'$ depicted in Figure \ref{fn1}(a). Though some sides
may collapse, the polygon is always nondegenerate, which means, in
particular, that $D$ is big. Notice also that
\mbox{$DK_\Sig=D'K_{\Sig'}$}.

Denote by $\Del'_n$ the multi-set of the primitive integral exterior
normal vectors to $n\Pi'$, where the multiplicity of each normal
equals the lattice length of the corresponding side. The subset
${\cal T}_n \subset {\cal T}^c(\Del'_n, 0, 0, 0, 0, n(D'E)\theta_1,
\bx_n)$ formed by the isomorphism classes~$T$ such that $W(T)=1$ is
non-empty (see, for instance, \cite{IKS}). Furthermore, by
\cite[Theorem 3]{IKS2},
$$\log\sum_{T\in{\cal T}_n} W(T)=(-D'K_{\Sig'})n\log n+O(n)
=(-DK_\Sig)n\log n+O(n),\quad n\to\infty\ .$$ Any tropical curve
$(\overline\Gam, \CV, h, \bpp)$ representing a class $T \in {\cal
T}_n$ has $n(k_1-k_2+k_5)$ ends directed by vector $(1,0)$.
Lemma~\ref{splitting} implies that each of these ends crosses the
segment $I_{r_n} = \{x = \del_{r_n}, \ 0 < y <\eps\}$. Among these
ends, select $nk_1\le n(k_1-k_2+k_5)$ ends and consider a marked
tropical $\CL$-curve $(\overline\Gam', \CV', h', \bpp')$ such that
$$
T' = [(\overline\Gam', \CV', h', \bpp')] \in {\cal T}^c(n\Del, nk_1,
nk_2, 0, 0, n(d - k_3 - k_4)\theta_1, (\bx_n, p)),
$$
$(\bpp')^2 = \varnothing$, and a marked cut of this curve at ${\cal
X}_{r_n} = (h')^{-1}(I_{r_n})$ satisfies the following properties:
\begin{itemize}
\item the only left component of the cut is isomorphic to
$(\overline\Gam, \CV, h, \bpp)$,
\item the number of right components
is $n(k_1 - k_2 + k_5)$,
\item exactly $n(k_1 - k_2)$ right components are isomorphic
to those shown on Figure~\ref{fn1}(b) and match $n(k_1 - k_2)$
selected ends of $(\overline\Gam, \CV, h, \bpp)$,
\item exactly $nk_2$
right components are isomorphic to those shown on
Figure~\ref{fn1}(c) and match $nk_2$ selected ends of
$(\overline\Gam, \CV, h, \bpp)$,
\item the remaining right components are horizontal.
\end{itemize}
If $\Sig=\PP^2_{3,1}$, we define an involution $\xi=\Id$ on
$\overline\Gam'$ and thus obtain a class in $\qquad$ ${\cal
T}^{c,\text{\rm sym}} (n\Del, nk_1, nk_2, 0, 0,
n(d-k_3-k_4)\theta_1, (\bx_n, p))$. Finally, notice that $W(T') =
W(T)$, which completes the proof in the case under consideration.

\begin{figure}
\setlength{\unitlength}{1cm}
\begin{picture}(13,19)(0,0.5)
\thinlines\put(1,3){\vector(1,0){3.5}}\put(1,3){\vector(0,1){2.5}}
\thinlines\put(1,9){\vector(1,0){3.5}}\put(1,9){\vector(0,1){3.5}}
\thinlines\put(1,16){\vector(1,0){3.5}}\put(1,16){\vector(0,1){3.5}}
\dashline{0.2}(3,16)(3,19)\dashline{0.2}(1,18)(4,18)
\dashline{0.2}(3,9)(3,12)\dashline{0.2}(1,11)(4,11)
\dashline{0.2}(3,3)(3,5)\dashline{0.2}(1,4)(4,4)

\thicklines \put(1,3){\line(1,0){3}}\put(1,3){\line(0,1){2}}
\put(1,5){\line(1,0){2}}\put(3,5){\line(1,-1){1}}\put(4,3){\line(0,1){1}}
\put(2,9){\line(-1,1){1}}\put(1,10){\line(0,1){2}}\put(2,9){\line(1,0){2}}\put(1
,12){\line(1,0){2}}
\put(3,12){\line(1,-1){1}}\put(4,9){\line(0,1){2}}\put(2,16){\line(-1,1){1}}
\put(1,17){\line(0,1){2}}\put(1,19){\line(1,0){2}}\put(2,16){\line(1,0){2}}
\put(3,19){\line(1,-1){1}}\put(4,16){\line(0,1){2}}\put(6.5,1.5){\line(0,1){3.5}
}
\put(7.5,8){\line(0,1){4}}\put(9,4){\line(1,0){1}}\put(10,4){\line(1,-1){1}}
\put(10,4){\line(0,1){1}}\put(11,3){\line(1,0){1}}\put(11,1.5){\line(0,1){1.5}}
\put(12,1.5){\line(0,1){1.5}}\put(12,3){\line(1,1){1}}\put(9.5,9.5){\line(1,1){
0.5}}
\put(10,10){\line(0,1){2}}\put(10,10){\line(1,0){1}}\put(11,8){\line(0,1){2}}
\put(11,10){\line(1,1){1}}\put(6.5,16){\line(1,0){1}}\put(7.5,15){\line(0,1){1}}
\put(7.5,16){\line(1,1){1}}\put(10.5,15){\line(0,1){1}}\put(10.5,16){\line(-1,1)
{0.5}}
\put(10,16.5){\line(0,1){0.5}}\put(10,16.5){\line(-1,0){0.5}}\put(6.5,19){
\line(1,0){2}}
\put(9.5,19){\line(1,0){2}}\put(10.5,16){\line(1,0){1}}

\put(7.4,18.4){$\bullet$}\put(10.4,18.4){$\bullet$}
\put(7.41,15.4){$\bullet$}\put(10.4,15.9){$\bullet$}
\put(7.9,9.9){$\bullet$}\put(10.9,9.9){$\bullet$}
\put(6,3.9){$\bullet$}\put(7.4,2.88){$\bullet$}
\put(9.4,3.9){$\bullet$}\put(11.9,2.88){$\bullet$}

\put(7.4,2.6){$\bx$}\put(5.9,3.5){$\bx''$}\put(9.4,3.5){$\bx''$}\put(12,2.7){
$\bx$}
\put(7.9,9.6){$\bx$}\put(11.2,9.7){$\bx$}\put(7.7,15.4){$\bx$}\put(10.7,15.6){
$\bx$} \put(7.7,18.4){$\bx$}\put(10.7,18.4){$\bx$}

\put(7.35,17.5){$\Big\Downarrow$}\put(10.35,17.5){$\Big\Downarrow$}
\put(8.5,9.9){$\Longrightarrow$}\put(8,3.4){$\Longrightarrow$}

\put(0.5,16.9){$k_4$}\put(0.6,17.9){$e$}\put(0.6,18.9){$a$}\put(2.9,15.6){$b$}
\put(3.9,15.6){$c$}\put(0.5,15){$a=d-k_3$, $b=k_3-k_2$}
\put(0.5,14.5){$c=d-k_1-k_5$, $e=k_1-k_2+k_5$}\put(7.4,13.7){\rm
(b)}\put(10.4,13.7){\rm (c)}\put(2.7,13.7){\rm (a)}

\put(0.6,9.9){$e$}\put(0.5,10.9){$k_5$}\put(0.6,11.9){$a$}\put(2.9,8.6){$b$}
\put(3.9,8.6){$c$}\put(0.5,8) {$a = d - k_1 - k_3 + k_5$, $b = k_3 -
k_5$} \put(0.5,7.5) {$c = d - k_1 - k_5$, $e = k_4 + k_5 - k_1$}
\put(2.7,6.7){\rm (d)}\put(9,6.7){\rm (e)}

\put(0.5,3.9){$k_5$}\put(0.6,4.9){$a$}\put(2.9,2.6){$b$}\put(3.9,2.6){$c$}
\put(0.5,2){$a=d-k_1-k_3+k_5$}\put(0.5,1.5){$b=k_3-k_5$,
$c=d-k_1-k_5$}\put(2.7,0.7){\rm (f)}\put(9,0.7){\rm (g)}
\end{picture}
\caption{Illustration to the proof of Theorem \ref{tn5}}\label{fn1}
\end{figure}
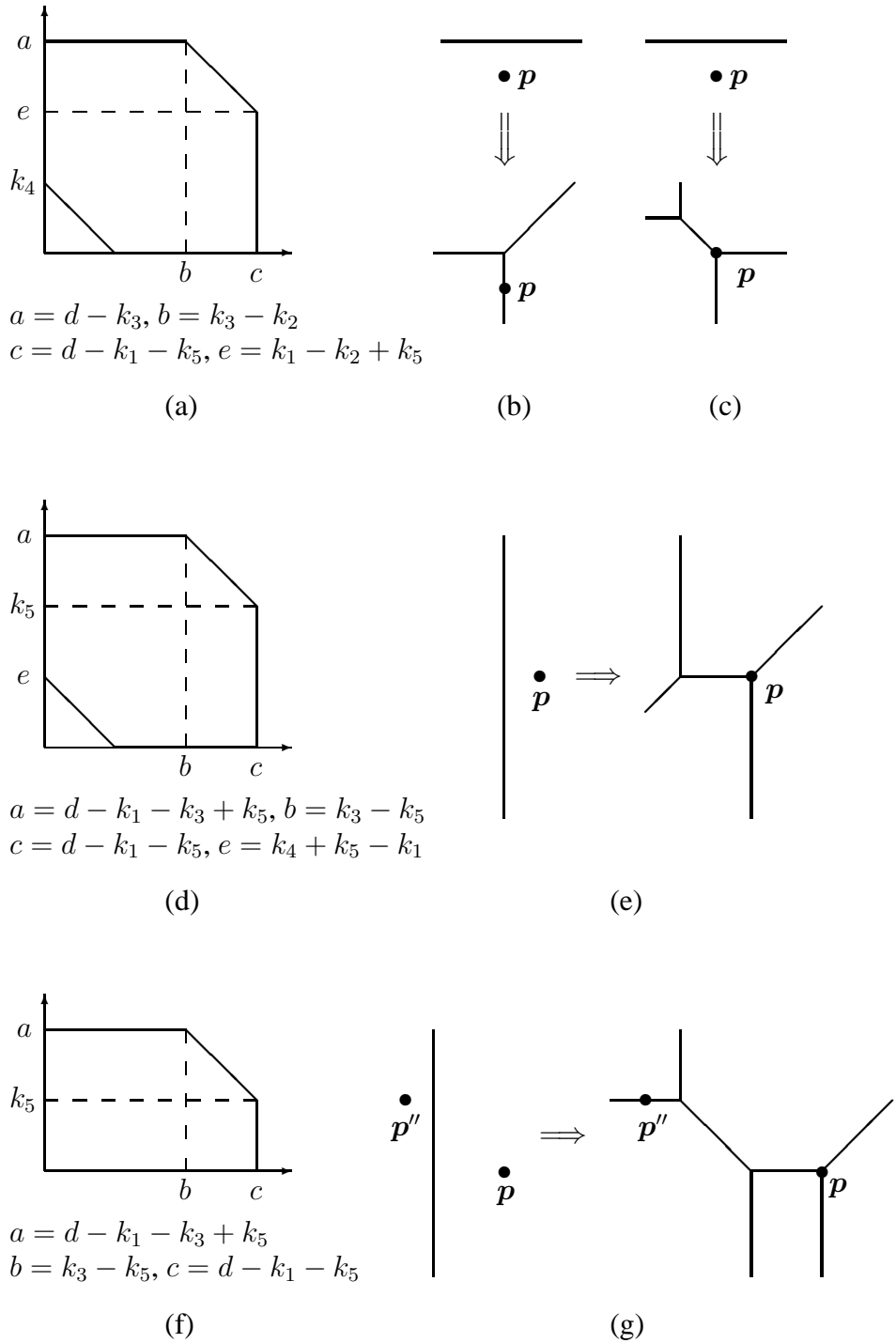

{\bf (2)} Assume that $k_5<k_1\le k_4+k_5$, which is relevant only
to the case $\Sig=\PP^2_{3,1}$, $k_1 = k_2$. Consider the divisor
$$D'= (d-k_1)L-(k_3-k_5)E_3-(k_4+k_5-k_1)E_4-k_5E_5$$ on the real
toric surface $\Sig'=\PP^2_{3,0}$. It is nef in view of our
assumption, inequalities (\ref{en14}) and the relations
$$D'E_i\ge0,\ i=3,4,5,\quad
D'(L-E_3-E_4)=d-k_3-k_4\ge0\ ,$$
$$D'(L-E_4-E_5)=d-k_4-2k_5\ge d-k_3-k_4-k_5\ge0\ ,$$
$$D'(L-E_3-E_5)=d-k_1-k_3\ge0\ .$$ Thus, the
linear system $|D'|$ is naturally associated with a nondegenerate
convex lattice polygon $\Pi'$ (see Figure \ref{fn1}(d)), which
implies that $D'$ is big.

As in the preceding step, we notice that $DK_\Sig=D'K_{\Sig'}$ and
denote by $\Del'_n$ the multi-set of the primitive integral exterior
normal vectors to $n\Pi'$, where the multiplicity of each normal
equals the lattice length of the corresponding side. The subset
${\cal T}_n \subset {\cal T}^c(\Del'_n, 0, 0, 0, 0, n(D'E)\theta_1,
\bx_n)$ formed by the isomorphism classes~$T$ such that $W(T)=1$ is
non-empty and
$$\log\sum_{T\in{\cal T}_n} W(T)=(-D'K_{\Sig'})n\log n+O(n)
=(-DK_\Sig)n\log n+O(n),\quad n\to\infty$$ (see~\cite{IKS}
and~\cite[Theorem 3]{IKS2}). Any tropical curve $(\overline\Gam,
\CV, h, \bpp)$ representing a class $T \in {\cal T}_n$ has $nk_5$
ends directed by vector $(1,0)$. Lemma~\ref{splitting} implies that
each of these ends crosses the segment $I_{r_n} = \{x = \del_{r_n},
\ 0 < y <\eps\}$. Denote by $\Delta''$ the multi-set obtained from
$\Delta$ by removing $k_1 - k_5$ vectors $(-1, -1)$ and $k_1 - k_5$
vectors $(1, 1)$. Consider a marked tropical $\CL$-curve
$(\overline\Gam', \CV', h', \bpp')$ such that
$$
T' = [(\overline\Gam', \CV', h', \bpp')] \in {\cal T}^c(n\Del'',
nk_5, nk_5, 0, 0, n(d - k_3 - k_4)\theta_1, (\bx_n, p))],
$$
and a marked cut of this curve at ${\cal X}_{r_n} =
(h')^{-1}(I_{r_n})$ satisfies the following properties:
\begin{itemize}
\item the only left component of the cut is isomorphic to
$(\overline\Gam, \CV, h, \bpp)$,
\item the number of right components
is $nk_5$, and all these components are isomorphic to those shown on
Figure~\ref{fn1}(c).
\end{itemize}
The tropical curve $(\overline\Gam', \CV', h', \bpp')$ has $n(d -
k_4 - k_5)$ ends directed by the vector $(0, -1)$.
Lemma~\ref{splitting} implies that $n(d - k_4 - 2k_5)$ of these ends
cross the half-line $J = \{x < \del_{r_n}, \ y = \eps'\}$, where
$\eps'$ is a sufficiently small positive number. Consider a marked
tropical $\CL$-curve $(\overline\Gam'', \CV'', h'', \bpp'')$ such
that
$$
T'' = [(\overline\Gam'', \CV'', h'', \bpp'')] \in {\cal T}^c(n\Del,
nk_1, nk_2, 0, 0, n(d - k_3 - k_4)\theta_1, (\bx_n, p))],
$$
and a marked cut of this curve at ${\cal X} = (h'')^{-1}(J)$ has
$n(d - k_1 - k_5) + 1$ irreducible components:
\begin{itemize}
\item one component isomorphic to
$(\overline\Gam', \CV', h', \bpp')$,
\item $n(k_1 - k_5)$ components isomorphic
to those shown on Figure~\ref{fn1}(e),
\item $n(d - k_1 - k_4 - k_5)$ components whose images
are vertical lines,
\item $n(k_4 + k_5 - k_1)$ components whose images
are straight lines of slope~$1$.
\end{itemize}
Define an involution $\xi=\Id$ on $\overline\Gam''$ and thus obtain
a class in ${\cal T}^{c,\text{\rm sym}} (n\Del, nk_1, nk_2, 0, 0,
n(d - k_3 - k_4)\theta_1, (\bx_n, p))$. Finally, notice that $W(T'')
= W(T') = W(T)$, which completes the proof in the case under
consideration.

{\bf (3)} Assume that $k_1 > k_4 + k_5$ which again is relevant only
in the case $\Sig=\PP^2_{3,1}$, $k_1 = k_2$. Impose an additional
condition on each configuration $\bx_n$, $n \geq 1$, namely suppose
that it can be subdivided into two parts $\bx_n=(\bx''_n,\bx'_n)$
such that
\begin{itemize}
\item
the number $|\bx'_n|$ of points in $\bx'_n$ is equal to
$n(3d-3k_1-k_3)-1$,
\item the number
$|\bx''_n|$ of points in $\bx''_n$ is equal to $n(k_1-k_4-k_5)$,
\item
$x_{p''}<x_{p'},\ y_{p''}<y_{p'}$ for all $p'=(x_{p'},
y_{p'})\in\bx'_n$ and $p''=(x_{p''}, y_{p''}) \in\bx''_n$.
\end{itemize}
Consider the divisor $D'=(d-k_1)L-(k_3-k_5)E_3-k_5E_5$ on the real
toric surface $\Sig'=\PP^2_{2,0}$. This divisor is nef due to
$$D'E_3=k_3-k_5\ge 0,\quad D'E_5=k_5\ge0\ ,$$
$$D'(L-E_3-E_5)=d-k_1-(k_3-k_5)-k_5=d-k_1-k_3\ge0\ ,$$ and is big, since
it is associated with a nondegenerate convex lattice polygon $\Pi'$
shown in Figure \ref{fn1}(f). Notice that $|\bx'_n| =
-nD'K_{\Sig'}-1$.

As before, denote by $\Del'_n$ the multi-set of the primitive
integral exterior normal vectors to $n\Pi'$, where the multiplicity
of each normal equals the lattice length of the corresponding side.
The subset ${\cal T}_n \subset {\cal T}^c(\Del'_n, 0, 0, 0, 0,
n(D'E)\theta_1, \bx'_n)$ formed by the isomorphism classes~$T$ such
that $W(T)=1$ is non-empty and
\begin{equation}
\log\sum_{T \in {\cal T}_n} W(T)= (-D'K_{\Sig'})n\log n+O(n)
=(3d-3k_1-k_3)n\log n+O(n)\label{en15}
\end{equation}
(see~\cite{IKS} and~\cite[Theorem 3]{IKS2}). Denote by $\Delta''$
the multi-set obtained from $\Delta$ by removing $n(k_1 - k_4 -
k_5)$ vectors $(0, -1)$, then $n(k_1 - k_4 - k_5)$ vectors (-1, 0),
and finally $n(k_1 - k_4 - k_5)$ vectors $(1, 1)$. Any tropical
curve $(\overline\Gam, \CV, h, \bpp)$ representing a class $T \in
{\cal T}_n$ has $nk_5$ ends directed by the vector $(1, 0)$ and $n(d
- k_1 - k_5)$ ends directed by the vector $(0, -1)$. Applying a
procedure similar to the one described in the previous step, we
obtain a marked tropical $\CL$-curve $(\overline\Gam'', \CV'', h'',
\bpp'')$ such that $T'' = [(\overline\Gam'', \CV'', h'', \bpp'')]
\in {\cal T}^c(n\Del'', n(k_4 + k_5), n(k_4 + k_5), 0, 0, n(d - k_1
- k_3 + k_5)\theta_1, (\bx'_n, p))$ and $W(T'') = W(T)$.

The curve $(\overline\Gam'', \CV'', h'', \bpp'')$ has $n(d - k_1)$
ends directed by the vector $(0, -1)$; choose among them
$n(k_1-k_4-k_5)$ ends not adjacent to~$(\bpp'')^\nu$. There are
$(n(k_1 - k_4 - k_5))!$ possibilities to fix a one-to-one
correspondence between the chosen ends and the points of $\bx''_n$.
Thus, there exist $(n(k_1 - k_4 - k_5))!$ marked tropical
$\CL$-curves $(\overline\Gam''', \CV''', h''', \bpp''')$ such that
\begin{itemize}
\item $T''' = [(\overline\Gam''', \CV''', h''', \bpp''')]
\in {\cal T}^c(n\Del, nk_1, nk_2, 0, 0, n(d - k_3 - k_4)\theta_1,
(\bx_n, p))$,
\item $W(T''') = W(T'') = W(T)$,
\item $(\overline\Gam''', \CV''', h''', \bpp''')$
has a marked cut whose irreducible components are as follows: one of
them is isomorphic to $(\overline\Gam'', \CV'', h'', \bpp'')$, the
other $n(k_1 - k_4 - k_5)$ components are as shown on
Figure~\ref{fn1}(g).
\end{itemize}
Equipping each of these curves with the identity involution, we
obtain representatives of $(n(k_1 - k_4 - k_5))!$ classes in ${\cal
T}^{c,\text{\rm sym}}(n\Del, nk_1, nk_2, 0, 0, (d-k_3-k_4)\theta_1,
(\bx_n,p))$, and the Welschinger multiplicity of all these classes
is equal to $W(T)$. This immediately implies the positivity of
$W(\Sig,D)$ as well as the required asymptotics, since from
(\ref{en15}) and our construction we obtain
$$
\log W(\Sig, nD) \geq \log\left((n(k_1-k_4-k_5))!\sum_{T\in {\cal
T}_n} W(T)\right)$$ $$=(k_1-k_4-k_5)n\log n+(3d-3k_1-k_3)n\log
n+O(n)=(-DK_\Sig)n\log n+O(n)\ .$$ \proofend

\subsection{Mikhalkin's congruence}\label{sec-con}

\begin{thm}\label{tn6}
For any nef and big divisor $D$ on a surface $\Sig=\PP^2$,
$(\PP^1)^2$, or $\PP^2_{q,0}$, $1\le q\le 5$, one has
\begin{equation}W(\Sig,D)\equiv GW_0(\Sig,D)\mod 4\ .\label{en17}
\end{equation}
\end{thm}

\begin{proof} Straightforward from Theorem \ref{tn12}(1) and the definition
of the complex and Wel- schinger multiplicities in sections
\ref{sec42} and \ref{sec43}. \end{proof}

\begin{rem} For $\Sig=\PP^1$, $(\PP^1)^2$, and $\PP^2_{q,0}$, $q=1,2,3$,
congruence (\ref{en17}) has been established by G. Mikhalkin
(\cite{Mip}, {\it cf}.~\cite{BM}).
\end{rem}

\subsection{Monotonicity}\label{sec-mon}

\begin{lem}\label{ln2} Let $D_1,D_2$ be nef and big divisors on $\PP^2_5$.
If $D_2-D_1$ is effective, then $D_2-D_1$ can be decomposed into a
sum $E^{(1)}+...+E^{(k)}$ of smooth rational $(-1)$-curves such that
each of $D^{(i)}=D_1+\sum_{j\le i}E^{(j)}$ is nef and big, and
satisfies $D^{(i)}E^{(i+1)}>0$, $i=0,...,k-1$.
\end{lem}

\begin{proof} As any effective divisor on $\PP^2_5$, the divisor $D_2
- D_1$ has a splitting $D_2-D_1=E^{(1)}+...+E^{(k)}$ into the sum of
smooth rational $(-1)$-curves. It remains to show that an
appropriate reordering of $E^{(1)},...,E^{(k)}$ ensures the
properties asserted in the lemma.

For $i=0$, the divisor $D^{(0)}=D_1$ is nef and big.

Suppose that $D^{(i)}$ is nef and big for some $i=0,...,k-1$. Show,
first, that there is $j\in[i+1,k]$ with $D^{(i)}E^{(j)}>0$. Indeed,
if $i=k-1$ then $D^{(k-1)}+E^{(k)}=D_2$ is nef, and hence
$$
D^{(k-1)}E^{(k)}=D_2E^{(k)}-(E^{(k)})^2\ge1\ .$$ If $i\le k-2$, and
$E^{(i+1)},...,E^{(k)}$ are all orthogonal to $D^{(i)}$, then they
cannot be all pairwise orthogonal, since otherwise one would have
$D_2E^{(j)}=-1$, $j=i+1,...,k$. Thus, there exist $E^{(j)},E^{(l)}$
with $i<j<l\le k$ and $E^{(j)}E^{(l)}\ge1$. Therefore, our
assumption $D^{(i)}E^{(j)}=D^{(i)} E^{(l)}=0$ leads by Hodge index
theorem to $(D^{(i)})^2 \le 0$ contrary to the bigness of $D^{(i)}$.
Hence there is $j>i$ with $D^{(i)}E^{(j)}>0$.

Now we can suppose that $j=i+1$ and put
$D^{(i+1)}=D^{(i)}+E^{(i+1)}$. This divisor is big in view of
$$(D^{(i+1)})^2=(D^{(i)})^2+2D^{(i)}E^{(i+1)}-1>(D^{(i)})^2>0\ .$$ It is
nef, since $D^{(i)}$ is nef and
$$D^{(i+1)}E^{(i+1)}=D^{(i)}E^{(i+1)}-1\ge 0 \qedhere$$\end{proof}

\begin{thm}\label{tn7}
Let $D_1,D_2$ be nef and big divisors on $\Sig=\PP^2_{5,0}$ such
that $D_2-D_1$ is effective. Then $W(\Sig,D_2)\ge W(\Sig,D_1)$.
Moreover, in the notation of Lemma~\ref{ln2},
$$W(\Sig,D_2)\ge\prod_{i=1}^k(D^{(i-1)}E^{(i)})\cdot W(\Sig,D_1)\ .$$
\end{thm}

\begin{rem}\label{rn3}
Theorem \ref{tn7} implies the similar inequalities for divisors on
$\Sig=\PP^2$, $\PP^2_{q,0}$, $1\le q\le 4$, and $(\PP^1)^2$. In
particular, it strengthens the monotonicity result for the surfaces
$\PP^2$, $\PP^2_{q,0}$, $q=1,2,3$, and $(\PP^1)^2$ from
\cite[Corollary 4]{IKS3}.
\end{rem}

{\bf Proof of Theorem \ref{tn7}}. Let $E_1,\dots, E_5$ be the
exceptional curves of the blow up, and $L\in\Pic(\Sig)$ the
pull-back of a line in $\PP^2$.

In view of Lemma \ref{ln2}, it is sufficient to treat the case of
$D_2-D_1=E_1$.

(1) Assume that $D_2E_1=0$. We claim that $W(\Sig,D_1)=W(\Sig,D_2)$.

Indeed, then
$$D_2=dL-k_2E_2-...-k_5E_5,\quad D_1=dL-E_1-k_2E_2-...-k_5E_5\ .$$
Let us blow down the divisor $E_1$. Choosing a generic configuration
$\bp$ of $-D_1K_\Sig-1$ real points in $\Sig$, we obtain a bijection
between ${\cal R}(\Sig,D_1,\bp)$ and ${\cal
R}(\Sig',\pi(D_2),\bp')$, where $\bp'=\pi(\bp)\cup\{\pi(E_1)\}$ is a
generic configuration of $-D_1K_\Sig=-D_2K_\Sig-1$ real points in
$\Sig'=\PP^2_{4,0}$. Hence
$W(\Sig,D_1)=W(\Sig',\pi(D_2))=W(\Sig,D_2)$.

(2) Assume that $D_2E_1=m>0$. Performing a suitable real
automorphism of $\Pic(\Sig)$, we can make $E=E_1$ and apply formula
(\ref{en60}) to $W(\Sig,D_2)=W_\Sig(D_2,0,m\theta_1)$ in the
left-hand side. Then, among the summands of the right-hand side, one
has
$$(m+1)W_\Sig(D_2-E,0,(m+1)\theta_1)=(m+1)W(\Sig,D_1)\ .$$ Thus, we are
done, since all other summands in the formula are non-negative.
\proofend

\bigskip
\footnotesize \noindent\textit{Acknowledgments.} A considerable part
of the work on this text was done during our stays at \emph{Centre
Interfacultaire Bernoulli}, \emph{\'{E}cole Polytechnique
F\'{e}d\'{e}rale de Lausanne} in 2008, at \emph{Mathematisches
Forschungsinstitut Oberwolfach} in 2009, and at the Mathematical
Sciences Research Institute, Berkeley in 2009, as well as during the
visits of the first author to Tel Aviv University and the visits of
the third author to {\it Universit\'e de Strasbourg}. We thank these
institutions for the support and excellent working conditions. We
are grateful to the referee for the careful reading of the first
version of the paper and many useful comments.

The first two authors were partially funded by the ANR-05-0053-01
and ANR-09-BLAN-0039-01 grants of {\it Agence Nationale de la
Recherche}, and are members of FRG Collaborative Research "Mirror
Symmetry \& Tropical Geometry" (Award No. $0854989$). The third
author was supported by the Israeli Science Foundation grant no.
448/09 and by the Hermann-Minkowski-Minerva Center for Geometry at
Tel Aviv University.

\end{document}